\documentclass[12pt,article,oneside]{amsart}
\usepackage[utf8]{inputenc}
\usepackage{amsmath,amscd,hyperref, amsfonts, amssymb, amsthm}
\usepackage{setspace,kantlipsum}
\usepackage{tikz-cd}
\usepackage{mathrsfs}
\usepackage{textcomp}
\usepackage[margin=1in]{geometry}
\usepackage{colonequals}
\usepackage{mathtools}
\DeclarePairedDelimiter{\abs}{\lvert}{\rvert}

\theoremstyle{plain} 
\newtheorem{thm}{Theorem}[subsection]
\newtheorem{lemma}[thm]{Lemma}
\newtheorem{prop}[thm]{Proposition}
\newtheorem{corollary}[thm]{Corollary}
\newtheorem{claim}[thm]{Claim}

\theoremstyle{definition}
\newtheorem{notation}[thm]{Notation}
\newtheorem{setup}[thm]{Set-up}
\newtheorem{definition}[thm]{Definition}
\newtheorem{construction}[thm]{Construction}

\newtheorem{remark}[thm]{Remark}

\newtheorem{assumption}[thm]{Assumption}

\theoremstyle{plain} % just in case the style had changed
\newcommand{\thistheoremname}{}
\newtheorem{genericthm}[thm]{\thistheoremname}

\newtheorem*{genericthm*}{\thistheoremname}
\newenvironment{namedthm*}[1]
  {\renewcommand{\thistheoremname}{#1}%
   \begin{genericthm*}}
  {\end{genericthm*}}

\newtheorem{Lthm}{Theorem}

\newcommand{\C}{\mathbb{C}}

\newcommand{\Z}{\mathbb{Z}}

\newcommand{\cA}{\mathcal{A}}

\newcommand{\cC}{\mathcal{C}}

\newcommand{\cF}{\mathcal{F}}

\newcommand{\cH}{\mathcal{H}}

\newcommand{\cL}{\mathcal{L}}

\newcommand{\cO}{\mathcal{O}}
\newcommand{\cP}{\mathcal{P}}

\newcommand{\cV}{\mathcal{V}}

\newcommand{\cX}{\mathcal{X}}
\newcommand{\cY}{\mathcal{Y}}

\newcommand{\bA}{\mathbf{A}}

\newcommand{\bG}{\mathbf{G}}

\newcommand{\bP}{\mathbf{P}}

\newcommand{\fp}{\mathfrak{p}}

\newcommand{\sD}{\mathscr{D}}

\newcommand{\sX}{\mathscr{X}}
\newcommand{\sY}{\mathscr{Y}}

\newcommand{\Gr}{\mathrm{Gr}}
\newcommand{\IC}{\mathrm{IC}}
\newcommand{\pre}{\mathrm{pre}}
\newcommand{\Wpre}{W^{\mathrm{pre}}}
\newcommand{\prim}{\mathrm{prim}}

\newcommand{\MTC}{\mathrm{MTC}}
\newcommand{\tot}{\mathrm{tot}}
\newcommand{\Gm}{\mathbf{G}_m}
\newcommand{\ex}{\mathrm{ex}}
\newcommand{\Patch}{\mathrm{Patch}}
\newcommand{\Cone}{\mathrm{Cone}}
\newcommand{\Iden}{\mathrm{Iden}}

\renewcommand{\d}{\partial}
\newcommand{\dbar}{\bar{\partial}}
\renewcommand{\t}{\theta}

\newcommand{\Hdual}{H^{\ast}}

\DeclareMathOperator{\Ker}{Ker}

\DeclareMathOperator{\im}{\mathrm{Im}}
\begin{document}

\title[Cohomology of semisimple local systems]{Cohomology of semisimple local systems and the Decomposition theorem}

\author{Chuanhao Wei and Ruijie Yang}

\address{Chuanhao Wei: Institute for Theoretical Sciences, School of Science, Westlake University and Institute of Natural Sciences, Westlake Institute for Advanced Study, 18 Shilongshan Road, Hangzhou 310024, Zhejiang Province, China.}
\email{weichuanhao@westlake.edu.cn}
\address{Ruijie Yang: Max-Planck-Institut f\"ur Mathematik, Vivatsgasse 7, 53111 Bonn, Germany.}
\email{ruijie.yang@hu-berlin.de, ryang@mpim-bonn.mpg.de}
%\thispagestyle{empty}

%\textsc{Institute of Natural Sciences, Westlake Institute for Advanced Study, 18 Shilongshan Road, Hangzhou 310024, Zhejiang Province, China.} \\
%\indent \textit{E-mail address:}{\href{mailto:}{weichuanhao@westlake.edu.cn}}

%\vspace{\baselineskip}

%\textsc{Institut f\"ur Mathematik, Humboldt-Universit\"at zu Berlin, Unter den Linden 6, 10099 Berlin, Germany} \\
%\indent \textit{E-mail address:} {\href{mailto:ruijie.yang@hu-berlin.de}{ruijie.yang@hu-berlin.de}\footnote{Corresponding author.}

\begin{abstract}
    In this paper, we study the cohomology of semisimple local systems in the spirit of classical Hodge theory.
    On the one hand, we construct a generalized Weil operator from the complex conjugate of the cohomology of a semisimple local system to the cohomology of its dual local system, which is functorial with respect to smooth restrictions. This operator allows us to study the Poincar\'e pairing, usually not positive definite, in terms of a positive definite Hermitian pairing. On the other hand, we prove a global invariant cycle theorem for semisimple local systems.
    
    As an application, we give a new proof of Sabbah's Decomposition Theorem for the direct images of semisimple local systems under proper algebraic maps, by adapting the method of de Cataldo-Migliorini, without using the category of polarizable twistor $\sD$-modules. This answers a question of Sabbah.
\end{abstract}

    \date{\today}
%\end{center}

\maketitle
%\section*{Statements and Declarations}

%\textbf{Funding:} not applicable.

%\textbf{Conflicts of interest/Competing interests:} not applicable.

%\textbf{Availability of data and material:} not applicable.

%\textbf{Code availability:} not applicable

%\textbf{Authors' contributions:} not applicable.

\vspace{1cm}

%\tableofcontents

\section*{Introduction}

The classical Hodge theory provides two fundamental tools for studying the cohomology of complex algebraic varieties: polarizations and weights. The polarizations allow one to study the Poincar\'e pairing using a positive definitive pairing; the theory of weights puts strong restrictions on morphisms between cohomology groups of algebraic varieties. On the other hand, the study of nontrivial local systems leads to deeper understanding of topology of algebraic varieties. For example it has applications to K\"ahler groups \cite{Simpson92}, higher dimensional uniformization \cite{Simpson88} and Shafarevich conjecture \cite{EKP}. Therefore it is natural to seek generalization of results on polarizations and weights to nontrivial local systems. The purpose of this paper is to study the cohomology of semisimple local systems in this spirit and provide some new tools. In particular, we construct a (functorial) generalized Weil operator, which induces polarizations on the pure twistor structures associated to semisimple local systems; we also prove a global invariant cycle theorem. As an application, we give a new proof of Sabbah's Decomposition Theorem \cite{Sabbah}, which might offer some new viewpoint on polarizations.

Our guiding principle is Simpson's ``meta-theorem" \cite{Simpson97}, which leads to the solution of Kashiwara's conjecture \cite{Mochizuki3}. But even for the cohomology of semisimple local systems, which is the simplest case, we discover something new: there are some hidden structures related to the polarization, e.g. a pre-Weil operator, which seems invisible in Hodge theory. 

Turning to details, let $X$ be a complex smooth projective variety and let $\cV$ be a local system on $X$. Denote the dual local system by $\cV^{\ast}$. Let $\eta$ be an ample line bundle on $X$ and fix an integer $k\leq \dim X$. The
\emph{twisted Poincar\'e pairing} is defined by
\begin{align*}
S:H^k(X,\cV)&\otimes H^k(X,\cV^{\ast}) \to  \C,\\
[\alpha] &\otimes [\beta] \mapsto (-1)^{k(k-1)/2}\cdot \int_X c_1(\eta)^{\dim X-k}\wedge \alpha \wedge \beta,
\end{align*}
where $\alpha$ is a $k$-form on $X$ with coefficient in $\cV$, same for $\beta$. Among arbitrary local systems, semisimple local systems are special as they are more manageable by analysis. From now on, we assume $\cV$ is semisimple. Corlette \cite{Corlette} showed that the $\cC^{\infty}$-bundle $H=\cV\otimes_{\C} \cC^{\infty}_X$ admits a unique harmonic metric up to a linear transformation. Unlike the constant local system $\C$, the cohomology groups of $\cV$ do not necessarily underlie pure Hodge structures. However, using the harmonic metric, Simpson \cite{Simpson97} proved that there is a natural semistable holomorphic bundle $E$ of slope $k$ over $\bP^1$ whose fiber at $1$ satisfies
\[ E|_{z=1} \cong H^k(X,\cV).\]
This can be viewed as a generalized pure Hodge structure of weight $k$ on $H^k(X,\cV)$. Indeed, Simpson defined a \emph{pure twistor structure} of weight $k$ to be a semistable holomorphic bundle of slope $k$ over $\bP^1$. Our first result is the construction of a pre-Weil operator
\[ \phi: \overline{E|_{z=-1}} \xrightarrow{\sim} H^k(X,\cV^{\ast}),\]
which is an isomorphism and is functorial with respect to restriction to smooth subvarieties. Precomposing an Identification map $E|_{z=1}\to E|_{z=-1}$ associated to the pure twistor structure $E$, this gives an isomorphism
\[ C:\overline{H^k(X,\cV)}  \xrightarrow{\sim} H^k(X,\cV^{\ast}).\]
In particular, it recovers the Weil operator on $H^k(X,\C)$. In the appendix, we reinterpret $C$ using Hodge star operators. Using Simpson's Lefschetz decomposition \cite[Lemma 2.6]{Simpson92} on $H^k(X,\cV)$, we can use $C$ to prove the following generalization of Hodge-Riemann bilinear relation.
\begin{Lthm}[Theorem \ref{cor:polarized by twisted Poincare pairing}]\label{thm: generalized Weil operator}
The bilinear pairing 
\[ S(\bullet,C(\bullet)):H^k(X,\cV)\otimes \overline{H^k(X,\cV)} \to \C\]
is Hermitian, and positive definite on each primitive space $\eta^{m}H^{k-2m}(X,\cV)_{\prim}$, up to a non-zero constant, where $H^{k}(X,\cV)_{\prim}$ denotes the primitive space with respect to $\eta$.  In particular, the bilinear pairing 
\[S( \bullet,\phi(\bullet)):E|_{z=1}\otimes \overline{E|_{z=-1}}\to \C\]
polarizes the pure twistor structure on $H^k(X,\cV)_{\prim}$, up to a non-zero constant.
\end{Lthm}
This theorem means that the pre-Weil operator $\phi$ should be an intrinsic structure associated to the polarization of the pure twistor structure on $H^k(X,\cV)_{\prim}$.  As in Hodge theory, the existence of polarizations enables one to keep track of non-degeneracy of topological pairings under restrictions. To illustrate the idea, let us consider a smooth  subvariety $T\subseteq X$ and the restriction maps
\[ R: H^k(X,\cV)_{\prim} \to H^k(T,\cV|_T), \quad \widetilde{R}: H^k(X,\cV^{\ast})_{\prim} \to H^k(T,\cV^{\ast}|_T).\]
Then the twisted Poincar\'e pairing $S$ restricts to a non-degenerate pairing:
\[ S: \Ker R\otimes \Ker \widetilde{R} \to \C.\]
Even though this is a purely topological statement, it seems difficult to find a topological proof.

Turning to weights, let $X$ be a complex quasi-projective variety and let $\cV$ be a local system on $X$ coming from the restriction of a semisimple local system on a smooth projective compactification $\overline{X}\supseteq X$. In \cite{Simpson97}, Simpson showed that $H^k(X,\cV)$ underlies a natural mixed twistor structure, which is functorial in $X$ and generalizes Deligne's mixed Hodge structures. This means that there is a holomorphic vector bundle $E$ over $\bP^1$ with an increasing filtration $W_{\bullet}E$ by strict subbundles such that $E|_{z=1}\cong H^k(X,\cV)$ and $W_{\ell}E/W_{\ell-1}E$ is a pure twistor structure of weight $\ell$. In \cite{Simpson97}, after stating this result, Simpson remarked that the same yoga of weights also holds, which is stated as follows.
\begin{namedthm*}{Simpson's theory of weights}[Theorem \ref{thm: yoga of weights}]
Let $W_{\bullet}$ denote the weight filtration on $H^k(X,\cV)$ induced by the natural mixed twistor structure. We have
\begin{itemize}
    \item if $X$ is proper, then $\Gr^W_iH^k(X,\cV)=0$ for $i> k$,
    \item if $X$ is smooth, then $\Gr^W_iH^k(X,\cV)=0$ for $i< k$,
    \item if $X$ is smooth and proper, then $\Gr^W_iH^k(X,\cV)=0$ for $i\neq  k$.
\end{itemize}
\end{namedthm*}
Using this, we prove a version of global invariant cycle theorem for semisimple local systems, generalizing Deligne's classical result \cite{Deligne74}.
\begin{Lthm}\label{thm: global invariant cycle}Consider the following chain of inclusion maps:
\[ Z \xrightarrow{\alpha} U \xrightarrow{j} X,\]
where $X$ is a smooth projective variety, $U$ is a Zariski open subset of $X$, and $Z$ is a proper subvariety of $X$ contained in $U$. Let $\cV$ be a semisimple local system on $X$. Then, for any integer $k$, the following two restriction maps have the same image:
\begin{align*}
(j\circ \alpha)^{\ast}:H^k(X,\cV) &\to H^k(Z,(j\circ \alpha)^{\ast}\cV),\\
\alpha^{\ast}:H^k(U,j^{\ast}\cV) &\to H^k(Z,(j\circ \alpha)^{\ast}\cV).
\end{align*}

\end{Lthm}

Using these new results on the cohomology of semisimple local systems, we give a relatively short and geometric proof of Sabbah's Decomposition Theorem for semisimple local systems under proper algebraic maps, by adapting the method of de Cataldo and Migliorini in \cite{DM05}. Sabbah's proof relies on the complicated category of polarizable twistor $\sD$-modules, we only need work in the category of constructible complexes. This answers a question of Sabbah \cite[Page 8]{Sabbah}. We would like to point out that the adaption of the de Cataldo-Migliorini method is nontrivial, which will be explained at the end of introduction. The main point is the compatibility of the pre-Weil operator $\phi$ with the perverse cohomology functor and smooth restrictions: on the one hand it allows us to keep track of non-degeneracy of various Poincar\'e pairings after restrictions to hyperplanes (as illustrated after Theorem \ref{thm: generalized Weil operator}); on the other hand it helps us translate the positivity coming from the polarization to the non-degeneracy of certain adjunction morphisms (see Proposition \ref{prop:splitting}), which is needed for the splitting criterion.

Now let us state the result. Let $f:X\to Y$ be a morphism between projective varieties, where $X$ is smooth. Let $\cV$ be a semisimple local system on $X$ and $\eta$ be an $f$-ample line bundle on $X$. Denote by $K\colonequals \cV[\dim X]$ the associated perverse sheaf. Let ${}^{\fp}\cH$ denote the perverse cohomology functor and $f_{\ast}$ denote the total direct image functor.
\begin{Lthm}\label{thm:main}We have
\begin{enumerate}
    \item[(i)] (Relative Hard Lefschetz Theorem) the following cup product map is an isomorphism:
    \[\eta^{\ell}\colon {}^{\fp}\cH^{-\ell}(f_{\ast}K) \xrightarrow{\sim}  {}^{\fp}\cH^{\ell}(f_{\ast}K). \] \item [(ii)](Decomposition Theorem) There exists an isomorphism in $D^b_c(Y)$:
    \[ f_{\ast}K \cong \bigoplus_{\ell} {}^{\fp}\cH^{\ell}(f_{\ast}K)[-\ell], \]
    where $D^b_c(Y)$ denotes the derived category of constructible sheaves on $Y$.
    \item[(iii)] (Semisimplicity Theorem) For each $\ell$, ${}^{\fp}\cH^{\ell}(f_{\ast}K)$ is a semisimple perverse sheaf. More precisely, given any stratification for $f$ so that $Y=\amalg_{d=\dim S_d}S_{d}$, there is a canonical isomorphism in $\mathrm{Perv}(Y)$:
    \[ {}^{\fp}\cH^{\ell}(f_{\ast}K)\cong \bigoplus_{d=0}^{\dim Y} \mathrm{IC}_{\overline{S_d}}(L_{\ell,d}),\]
    where the local systems $L_{\ell,d}\colonequals \cH^{-d}({}^{\fp}\cH^{\ell}(f_{\ast}K))|_{S_d}$ on $S_d$ are semisimple. Here $\mathrm{Perv}(Y)$ denotes the category of perverse sheaves on $Y$ and $\mathrm{IC}_{\overline{S_d}}(L_{\ell,d})$ denotes the intersection complex associated to the local system $L_{\ell,d}$.
\end{enumerate}
\end{Lthm}
Sabbah's original theorem holds for a proper map from $U\to Y$, where $U$ is a Zariski open subset of $X$ and $Y$ is a complex manifold. We will discuss in Theorem \ref{thm:Sabbah's theorem} how to obtain the Decomposition Theorem for a proper map $U\to Y$, where $Y$ is an algebraic variety. It is not clear to us whether or not our method can recover \cite[Main Theorem 2]{Sabbah}, which is about semisimplicity of nearby and vanishing cycles.

Theorem \ref{thm:main} will be proved by induction together with two more theorems. We use $H^b(X,K)$ to denote the hypercohomology of $K$. There is a perverse Leray filtration $H^b_{\leq \ell}(X,K)\subseteq H^b(X,K)$ induced by $f$ and we denote associated grade spaces by $H^b_{\ell}(X,K)\colonequals H^b_{\leq \ell}(X,K)/H^b_{\leq \ell-1}(X,K)$. Let $A$ be an ample line bundle on $Y$ and set $L\colonequals f^{\ast}A$. Assume the decomposition theorem holds, it follows from \cite[Lemma 4.4.2. and Remark 4.4.3]{DM05} that the cup products with $\eta$ and $L$ induce maps $\eta:H^b_{\ell}(X,K) \to H^{b+2}_{\ell+2}(X,K), \quad L:H^b_{\ell}(X,K) \to H^{b+2}_{\ell}(X,K)$. The next theorem generalizes \cite[Theorem 2.1.5]{DM05}.
\begin{Lthm}\label{thm:Hard Lefschetz Perverse Complexes}
Let $\ell,j \in \Z$. Then the following cup product maps are isomorphisms:
\begin{align*}
\eta^{\ell} \colon  H^{j-\ell}_{-\ell}(X,K) \xrightarrow{\sim} H^{j+\ell}_{\ell}(X,K), \quad L^{j} \colon H^{\ell-j}_{\ell}(X,K) \xrightarrow{\sim} H^{\ell+j}_{\ell}(X,K).
\end{align*}
\end{Lthm}
As a corollary, there is a double Lefschetz decomposition
\begin{equation*}
H^{-\ell-j}_{-\ell}(X,K)=\bigoplus_{i,m\in \mathbb{Z}} \eta^{-\ell+i}L^{-j+m}P^{j-2m}_{\ell-2i},\end{equation*}
where $P^{-j}_{-\ell}\colonequals 
\Ker \eta^{\ell+1} \cap \Ker L^{j+1} \subseteq H^{-\ell-j}_{-\ell}(X,K)$ for $\ell,j\geq 0$ and $P^{-j}_{-\ell}=0$ if $\ell<0$ or $j<0$.

Set $K^{\ast}=\cV^{\ast}[\dim X]$ and we denote by $\widetilde{P}^{-j}_{-\ell}\subseteq H^{-\ell-j}_{-\ell}(X,K^{\ast})$ the corresponding primitive pieces. For $\ell,j\geq 0$, there is a similar bilinear pairing as in \cite[Theorem 2.1.8]{DM05}.
\begin{align}
&S^{\eta L}_{\ell j}: H^{-\ell-j}_{-\ell}(X,K)\otimes_\C H^{-\ell-j}_{-\ell}(X,K^{\ast}) \to \C, \label{eqn:first time twisted Poincare pairing}\\
&[\alpha \otimes e]\otimes [\beta \otimes \lambda ]\mapsto C(n-\ell-j)\cdot \int_X  \lambda(e) \cdot c_1(\eta)^{\ell} \wedge c_1(L)^{j} \wedge \alpha \wedge \beta \nonumber.
\end{align}
Here $C(k)$ denotes the constant $(-1)^{k(k-1)/2}$.  In the next theorem, we prove that the pairing $S^{\eta L}_{\ell j}$ and the pre-Weil operator $\phi$ can be used to polarize the pure twistor structures on primitive pieces.
\begin{Lthm}
\label{thm:Hodge-Riemann} The pairing $S^{\eta L}_{\ell j}$ is well-defined and non-degenerate. The double Lefschetz decompositions of $H^{-\ell-j}_{-\ell}(X,K)$ and $H^{-\ell-j}_{-\ell}(X,K^{\ast})$ are orthogonal with respect to $S^{\eta L}_{\ell j}$. Furthermore, each direct summand $\eta^{-\ell+i}L^{-j+m}P^{j-2m}_{\ell-2i}$ underlies a natural pure twistor structure $F$ induced by the pure twistor structre on $H^{n-\ell-j}(X,\cV)$. The pre-Weil operator $\phi$ induces an isomorphism
\[ \phi:\overline{F|_{z=-1}} \xrightarrow{\sim} \eta^{-\ell+i}L^{-j+m}\widetilde{P}^{j-2m}_{\ell-2i},\]
so that $F$ is polarized by the bilinear pairing $S^{\eta L}_{\ell j}(\bullet,\phi(\bullet))$, up to multiplication by certain power of $\sqrt{-1}$ (see Remark \ref{remark: signs for twisted Poincare pairing}).
\end{Lthm}

To conclude the introduction, let us give a very brief historical discussion of the Decomposition Theorems (see \cite{Sabbah} for a more detailed account).
The Decomposition Theorem of Beilinson-Bernstein-Deligne-Gabber \cite{BBD}, says that Theorem \ref{thm:main} holds for the constant local system $\mathbb{C}$ and any proper map between algebraic varieties. Inspired by the Theorem of BBDG and M. Saito's Decomposition Theorem for Hodge modules \cite{Saito88}, Kashiwara \cite{Kashiwara96} conjectured that the Decomposition Theorem should hold for arbitrary semisimple holonomic $\sD$-modules. 

Theorem \ref{thm:main}, originally due to Sabbah \cite{Sabbah}, establishes Kashiwara's conjecture for semisimple local systems. Sabbah's proof drew on his deep theory of polarizable twistor $\sD$-modules. The main idea is that semisimple local systems on smooth projective varieties underlie polarizable twistor $\sD$-modules, where Sabbah proved a Decomposition Theorem in this category by extending Saito's method in \cite{Saito88}. This allows him to translate the results back to get the Decomposition Theorem for semisimple local systems.  Building on Sabbah's work, T. Mochizuki established the full Kashiwara's conjecture using his seminal work of harmonic bundles on quasi-projective varieties \cite{Mochizuki1,Mochizuki2,Mochizuki3}. Meanwhile, Kashiwara's conjecture for semisimple perverse sheaves was also proved by Drinfeld \cite{Drinfeld}, Gaitsgory \cite{Gaitsgory}, B\"{o}ckle-Khare \cite{BK} using arithmetic methods.

There are other approaches to Decomposition Theorems. Recently, Bhatt-Lurie \cite{BL} construct a $p$-adic Riemann-Hilbert functor from certain constructible sheaves to filtered $\sD$-modules, and then deduce the Decomposition Theorem for filtered $\sD$-modules of geometric origin from BBDG. Budur-Wang \cite{BW} used absolute sets to deduce the Decomposition Theorem for rank one perverse sheaves from the Theorem of BBDG. El Zein-L\^{e}-Ye \cite{ELY} gave another proof for the case of intermediate extensions of polarizable variations of Hodge structures using a local purity theorem.

Finally, let us highlight a couple of differences from the ideas of Sabbah and those of de Cataldo-Migliorini. In addition, we also want to mention how our results differ. 

\begin{remark}\label{remark: extra difficulity}
\begin{enumerate}

\item [(I)] Comparing with Sabbah's proof, instead of putting polarizable twistor $\sD$-modules structures on semisimple local systems, we simply put polarizable twistor structures on their cohomology groups with an extra pre-Weil operator $\phi$. It turns out that to run the de Cataldo-Migliorini method, one only needs to check the compatibility of the pre-Weil operator $\phi$ with perverse cohomology functors and adjunction morphisms. This allows us to run inductive arguments in the category of constructible complexes, which is technically simpler. Furthermore, this topological method reduces the Decomposition Theorem for arbitrary maps to the constant map (i.e. Simpson's Hard Lefschetz theorem and Theorem \ref{thm: generalized Weil operator}), with the help of Simpson's results for smooth projective maps and Theorem \ref{thm: global invariant cycle}. It is also easier than Zucker's theorem for tame harmonic bundles \cite[\S 5]{Sabbah}.

\item [(II)] 
Comparing with de Cataldo and Migliorini's method, the main difficulty we encounter is that there is no off-the-shelf Hodge theory. The idea is to use Simpson's theory of mixed twistor structures as a replacement. However, Simpson's result is not enough, especially there is no obvious relation between his polarization and the twisted Poincar\'e pairing, which is one of the main discovery of this paper. On the other hand, when the local system $\cV$ is $\C$, our argument gives a Hodge-decomposition-free proof of de Cataldo-Migliorini's results. In other words, we realize that it is \emph{polarization}, not Hodge decomposition, that is essential for the Decomposition Theorem. For example, the proof of Theorem \ref{thm:Hard Lefschetz Perverse Complexes} is different from that of \cite[Proposition 5.2.3]{DM05}. The proof of Theorem \ref{thm:Hodge-Riemann} is also different, where we need a new Corollary \ref{non-degeneracy of pairing on Lambda0}. 
\item [(III)] One key property used in the proof of de Cataldo-Migliorini is that the constant local system $\C$ is self-dual, while we need to run the inductive arguments for the local system $\cV$ and its dual $\cV^{\ast}$ together. This is the main reason for introducing notions of weight filtrations for two companion vector spaces in \S \ref{sec:weight filtrations}. Moreover, quite some efforts are needed to check that the pre-Weil operator $\phi$ is functorial and compatible with various geometric operations (see Lemma \ref{lemma:Identification map is functorial} and Corollary \ref{cor:restriction of perverse filtration}) and show non-degeneracy is preserved under restrictions. This enables us to prove non-degeneracy of certain adjunction morphisms using the positivity of polarizations, which is in turn needed to split the perverse cohomologies (see Lemma \ref{prop:splitting} and Proposition \ref{prop:splitting of the 0-th perverse cohomology}).

\item [(IV)] To apply the splitting criterion for the proof of Theorem \ref{thm:main}(iii) in the case of $\ell=0$, one needs to relate adjunction morphisms with Poincare pairings. This is probably  well-known to experts, we give a detailed discussion in \S \ref{splitting criterion} for the lack of appropriate references.
\end{enumerate}
\end{remark}

\subsubsection*{Structure of the paper}

This paper consists of two parts. In the first part \S \ref{sec:cohomology of semisimple local systems}, we study the cohomology of semisimple local systems. In the second part \S \ref{sec: proof of main results}, we give a new proof of Sabbah's Decomposition Theorem.

In \S \ref{sec:Cohomology of smooth projective varieties}, we introduce an equivalent definition of polarization on pure twistor structures.  In \S \ref{sec: cohomology of arbitrary}, we discuss Simpson's theory of weights and prove Theorem \ref{thm: global invariant cycle}. In \S \ref{sec: pre-Weil operator}, we construct the generalized Weil operator and prove Theorem \ref{thm: generalized Weil operator}. In \S \ref{sec:perverse filtrations}, we show that perverse filtrations underlie natural pure twistor structures. In \S \ref{sec:weight filtrations}, we set up various definitions concerning weight filtrations for nilpotent operators and polarizations on the associated graded spaces on two vector spaces related by a non-degenerate pairing.

In \S \ref{sec: proof of main results}, building on results in previous sections, we give a new proof of Sabbah's Decomposition Theorem.  We collect and prove some topological results about constructible complexes in \S \ref{cup product with a line bundle} - \S \ref{splitting criterion}. Then we give the proof in the rest of subsections.

In Appendix \ref{sec:Hodge star operators}, we sketch the construction of Hodge star operators for differential forms with coefficients in harmonic bundles.

\subsubsection*{Acknowledgement} This paper owes a tremendous debt to the work of Simpson on his non-abelian Hodge theory, and the work of de Cataldo and Migliorini on the topology of algebraic maps.

We would like to thank many people whose comments have led to improvements in this paper. We thank Christian Schnell for several long discussion on polarizations. We thank Takuro Mochizuki, Claude Sabbah and Carlos Simpson for several communications and comments on an earlier version of the paper. We thank Mark de Cataldo for the discussion on Theorem \ref{thm:Hodge-Riemann} and valuable suggestions on the organization of the paper. A special thank goes to Shizhang Li, who taught us about the relation between Verdier duality and Poincar\'e duality. Last, we thank Nathan Chen, Tong Chen, Bruno Klingler, Rob Lazarsfeld and Junliang Shen for reading a previous version of the manuscript and their suggestions for improving the exposition.

\section{Cohomology of semisimple local systems}\label{sec:cohomology of semisimple local systems}

In this section, we study the cohomology of semisimple local systems using polarizations and weights, and prove Theorem \ref{thm: generalized Weil operator} and Theorem \ref{thm: global invariant cycle}. The main set-up of this section is
\begin{setup}\label{setup: cohomology of smooth projective}
Let $X$ be a smooth projective variety of dimension $n$ and $\cV$ be a semisimple local system on $X$ with a flat connection $\nabla$. Let $\eta$ be an ample line bundle on $X$. Denote $\cV^{\ast}$ to be the dual local system and $H\colonequals \cV\otimes_\C \cC^{\infty}_X$ to be the associated $\cC^{\infty}$ flat bundle. Let $\cA^{k}_X$ denote the sheaf of $\cC^{\infty}$ $k$-forms on $X$, similarly for $\cA^{p,q}_X$, and let $(\cA^{\bullet}_X(H),\nabla)$ denote the de Rham complex associated to $H$ and $\nabla$
\[ (\cA^{\bullet}_X(H),\nabla)\colonequals [H \xrightarrow{\nabla} \cA^{1}_X\otimes_{\cC^{\infty}_X}H \to \cdots \to \cA^{2\dim X}_X\otimes_{\cC^{\infty}_X}H],\]
with cohomology degree $0$ at the term $H$.  
\end{setup}

\begin{notation}\label{notation: omega}
Denote $\Omega_0$ and $\Omega_{\infty}$ to be the standard $\bA^1$-neighborhoods for $0$ and $\infty$ in $\bP^1$ and set $\Gm\colonequals \Omega_0\cap \Omega_{\infty}$.
\end{notation}

\subsection{Classical results}  
\begin{thm}[Hodge-Simpson]\label{thm:Hodge-Simpson} We have the following results.

\begin{itemize} 

\item (Hard Lefschetz Theorem) For each $j\geq 1$, the cup product map with $\eta$ is an isomorphism
\[ \eta^j \colon H^{n-j}(X,\cV) \xrightarrow{\sim} H^{n+j}(X,\cV).\]

\item (Lefschetz decomposition) For each $k\geq 1$, $H^{k}(X,\cV)$ underlies a natural pure twistor structure of weight $k$. Moreover, for $k\leq n$, there is a direct sum decomposition
\begin{align*}
H^{k}(X,\cV) \cong \bigoplus_{m\geq 0} \eta^mH^{k-2m}(X,\cV)_{\prim},
\end{align*}
where $H^k(X,\cV)_{\prim}\colonequals \Ker \eta^{\dim X-k+1}\subseteq H^{k}(X,\cV)$ and each primitive space $\eta^mH^{k-2m}(X,\cV)_{\prim}$ underlies a natural pure sub-twistor structure.
\end{itemize}
\end{thm}

\begin{proof}
The first statement is \cite[Lemma 2.6]{Simpson92}. The second statement comes from \cite[Theorem 4.1]{Simpson97} (see also \cite[Theorem 2.2.4]{Sabbah}). 
\end{proof}

\begin{thm}[Semisimplicity Theorem for smooth projective maps]\label{thm:Simpson Semisimplicity} Let $f: X \to Y$ be a projective morphism between smooth projective varieties. Let $\cV$ be a semisimple local system on $X$. Let $Y_0\subseteq Y$ be the open subset over which $f$ is smooth and let $X_0=f^{-1}(Y_0)$. Then $R^qf_{\ast}(\cV|_{X_0})$ is a semisimple local system on $Y_0$ for every $q \geq 0$.
\end{thm}

\begin{proof}
This is proved by Simpson in \cite[Corollary 4.5]{Simpson93}.
\end{proof}

\begin{corollary}\label{thm:semisimplicity smooth projective map}
Let $T\subseteq Y_0$ be a smooth subvariety and let $X_T=f^{-1}(Y_T)\subseteq X_0$. Then $R^qf_{\ast}(\cV|_{X_T})$ is a semisimple local system on $T$ for every $q \geq 0$.
\end{corollary}

\begin{proof}
Eliminating the indeterminacies of maps, we can find the following commutative diagram:
\[ \begin{CD}
X_T @>>> X_{\overline{T}} @>h>> X\\
@VVfV @VVf_{\overline{T}}V @VVfV\\
T @>>> \overline{T} @>>> Y,
\end{CD}\]
where $\overline{T}$ is a smooth projective variety containing $T$ as a Zariski open subset and $X_{\overline{T}}$ is smooth projective. Using the metric charaterization of semisimple local systems on smooth projective varieties (cf. Corollary \ref{cor:pull back preserve semisimplicity}), we know the local system $h^{\ast}\cV$ is semisimple. Then by Theorem \ref{thm:Simpson Semisimplicity} and Remark \ref{remark:semisimplicity via Zariski dense open subset}, it follows that
\[ R^qf_{\ast}(\cV|_{X_T})=R^qf_{\ast}((h^{\ast}\cV)|_{X_T})\]
is semisimple on $T$.
\end{proof}

\begin{remark}\label{remark:semisimplicity via Zariski dense open subset}
Let $Y$ be a smooth variety and let $\cV$ be a local system on $Y$. Let $Y'\subseteq Y$ be a Zariski dense open subset, then $\cV$ is semisimple if and only if $\cV|_{Y'}$ is semisimple. This follows from the surjectivity of the natural map $\pi_{1}(Y') \twoheadrightarrow \pi_{1}(Y)$. (For example, see \cite[X, Th\'eor\`eme 2.3]{God71}, and notice that $Y\setminus Y'$ has real codimension at least $2$.)
\end{remark}

\subsection{Cohomology of smooth projective varieties and pure twistor structures}\label{sec:Cohomology of smooth projective varieties}

In this subsection, we review Simpson's notion of pure twistor structures following \cite{Simpson97}. For our purposes, we introduce an equivalent definition of polarization. This allows us to relate the polarization with the twisted Poincar\'e pairing, via the pre-Weil operator $\phi$ constructed in \S \ref{sec: pre-Weil operator}.

\begin{definition}\label{definition:pure twistor structures} 
A \emph{twistor structure} is a holomorphic vector bundle $E$ on $\bP^1$.  Morphisms of twistor structures are morphisms between holomorphic vector bundle over $\bP^1$. We say a complex vector space $V$ \emph{underlies} a twistor structure $E$ if $V\cong E|_{z=1}$ (the fiber over the point $1\in \bP^1$).  We also say that $E$ \emph{is a twistor structure on} $V$. A twistor structure $E$ is \emph{pure} of weight $w$ if $E$ is a direct sum of copies of $\cO_{\bP^1}(w)$.  \end{definition}

\begin{remark}\label{remark:Category of twistor structures}
It is clear that the category of pure twistor structures with a fixed weight is equivalent to the category of complex vector spaces.
\end{remark}

For the study of polarizations, we define a functorial Identification map associated to any pure twistor structure.

\begin{definition}\label{definition:Identification map} Let $E$ be a pure twistor structure. The \emph{Identification map} $\mathrm{Iden}: E|_{z=1} \to E|_{z=-1}$ is defined as follows. If the weight of $E$ is $0$, define  
\[ \Iden: E|_{z=1} \xrightarrow{(\mathrm{ev}_{z=1})^{-1}}  H^0(\bP^1,E)\xrightarrow{\mathrm{ev}_{z=-1}} E|_{z=-1},\]
where $\mathrm{ev}_{z=z_0}$ is the isomorphic evaluation map for global sections.

If the weight of $E$ is $w\neq 0$, we do a Tate twist to weight zero: choose $\mu \in H^0(\bP^1,\cO_{\bP^1}(1))$ to be the unique section up to scaling so that $\mu|_{\Omega_0}$ is nowhere zero. The evaluation of $\mu^{\otimes w}$ at $z=z_0\in \Omega_0$ identifies $\cO_{\bP^1}(w)|_{z=z_0}$ with $\C$.    We use $\mathrm{ev}^{w}_{z=z_0}$ to denote the composition of following maps:
\[ \mathrm{ev}^{w}_{z=z_0}: H^0(\bP^1,E(-w)) \xrightarrow{\mathrm{ev}^{w}_{z=z_0}} E(-w)|_{z=z_0} \xrightarrow{\mu^{\otimes w}|_{z=z_0}} E|_{z=z_0}. \]
Then we define
\[ \mathrm{Iden}:E|_{z=1}\xrightarrow{(\mathrm{ev}^{w}_{z=1})^{-1} } H^0(\bP^1,E(-w)) \xrightarrow{\mathrm{ev}^{w}_{z=-1}} E|_{z=-1}.\]
It is direct to see that it does not depend on the choice of $\mu$.

\end{definition}
\begin{remark}
In fact, one can actually define Identification map between any two fibers of $E$. Here $1$ and $-1$ are special because they are related to Weil operators in Hodge theory (see Corollary \ref{cor:generalized Weil operator}).
\end{remark}
\begin{remark}
It is direct to see that the Identification map is functorial with respect to pure sub-twistor structures.
\end{remark}

Now let us turn to polarizations. Let $E$ be a pure twistor structure of weight $w$. Let $\sigma$ denote the antipodal involution of $\bP^1$, where $\sigma(z) = -1/\overline{z}$. As in \cite[\S 2]{Simpson97}, $\sigma$ induces another pure twistor structure $\sigma^{\ast}(E)$ with an induced isomorphism:
\[\sigma: \overline{H^0(\bP^1,E)} \xrightarrow{\sim} H^0(\bP^1,\sigma^{\ast}E).\]

\begin{definition}[Simpson's Polarization]\label{definition:Simpson's polarization}
Consider a morphism of pure twistor structures
\[ P:E \otimes_{\cO_{\bP^1}} \sigma^{\ast}E \to \cO_{\bP^1}(2w),\]
which is equivalent to $P(-2w):E(-w) \otimes_{\cO_{\bP^1}} \sigma^{\ast}(E(-w)) \to \cO_{\bP^1}$. We say $P$ is a \emph{polarization} on $E$ if the induced morphism on global sections 
\begin{align*}
&H^0(\bP^1,E(-w))\otimes_\C \overline{H^0(\bP^1,E(-w))}
\xrightarrow{\mathrm{Id}\otimes \sigma} H^0(\bP^1,E(-w))\otimes_{\C} H^0(\bP^1,\sigma^{\ast}(E(-w))) \to \C
\end{align*}
is a positive definite Hermitian pairing. 
\end{definition}
In fact, any pure twistor structure is polarizable. The polarization on $H^k(X,\cV)$ we use in this paper is the canonical one that comes from twisted Poincar\'e pairings. For this purpose, we would like to have the following lemma.

\begin{lemma}\label{lemma:characterization of polarization via bilinear pairings}
The pure twistor structure $E$ is polarized by a morphism $P$ if and only if the bilinear pairing $S\colonequals P|_{z=1}$ as
\[ S:E|_{z=1} \otimes_{\C} \overline{E|_{z=-1}} \to \C,\]
induces a positive definite Hermitian pairing
\[ S(\bullet,\overline{\mathrm{Iden}}(\bullet)): E|_{z=1}\otimes_{\C} \overline{E|_{z=1}} \to \C.\]
Here $\mathrm{Iden}$ comes from Definition \ref{definition:Identification map}.

\end{lemma}

\begin{proof}
The proof is straightforward via the definition of Identification map. For details, see \cite[Lemma 2.3.20]{Yang21}.
\end{proof}

As a result, we can use bilinear pairings to polarize pure twistor structures.
\begin{definition}[Polarization by a bilinear pairing]\label{definition:pure twistor structures polarized by bilinear pairings}
Given a bilinear pairing $S: E|_{z=1}\otimes_{\C} \overline{E|_{z=-1}} \to \C$, we say $E$ \emph{is polarized by} $S$, if the pairing
\[ S(\bullet,\overline{\mathrm{Iden}}(\bullet)): E|_{z=1}\otimes_{\C} \overline{E|_{z=1}} \to E|_{z=1}\otimes_{\C} \overline{E|_{z=-1}}  \to \C\]
is positive definite and Hermitian, where $\mathrm{Iden}$ is the Identification map in Definition \ref{definition:Identification map}.
\end{definition}
The following two statements are direct to check. 
\begin{lemma}\label{lemma: polarization implies non-degeneracy}
If $S$ is a bilinear pairing that polarizes $E$, then $S$ is non-degenerate.
\end{lemma}

\begin{lemma}\label{lemma:polarization of subtwistors}
Let $G\subseteq E$ be a pure sub-twistor structure. If $E$ is polarized by a bilinear pairing $ S_E:E|_{z=1}\otimes_\C \overline{E|_{z=-1}}\to \C$, then its restriction
\[ S_G:G|_{z=1}\otimes_\C \overline{G|_{z=-1}}\to \C\]
polarizes the pure twistor structure $G$.
\end{lemma}

\subsubsection{Cohomology of smooth projective varieties with semisimple coefficients}\label{sec:Hodge-Simpson} We use Set-up \ref{setup: cohomology of smooth projective}.  The following result is proved by Corlette \cite[Theorem 3.3]{Corlette} and Simpson \cite[Theorem 1]{Simpson92}. We follow \cite[Page 13, paragraph 2]{Simpson92} with slightly different notations, where Simpson's $D,\d,\dbar,\theta,\overline{\theta},d',d''$ are our $\nabla,\d',\d'',\theta',\theta'',D',D''$ respectively.

\begin{thm}\label{thm:Corlette-Simpson}
The bundle $H$ admits a harmonic metric $h$ inducing the decomposition
\[ \nabla=\d'+\theta'+\d''+\theta'',\]
with
\begin{align*}
&\d': H \to \cA^{1,0}_X \otimes_{\cC^{\infty}_X} H, \quad \theta': H \to \cA^{1,0}_X \otimes_{\cC^{\infty}_X} H,\\
&\d'': H \to \cA^{0,1}_X \otimes_{\cC^{\infty}_X} H,\quad \theta'': H \to \cA^{0,1}_X \otimes_{\cC^{\infty}_X} H,
\end{align*}
so that $\d'$ is a $(1,0)$-connection, $\d''$ is a $(0,1)$-connection and $\d'+\d''$ is a metric connection of $h$. The operator $\theta'$ is $\cC^{\infty}_X$-linear and $\theta''$ is the adjoint of $\theta'$ with respect to $h$. Moreover, if we set $D'\colonequals \d'+\theta''$ and $D''\colonequals \d''+\theta'$, then
\[ (D')^2= (D'')^2=0.\] 
Conversely, if a local system admits a harmonic metric, then it is semisimple.
%The operator $D''$ induces a Higgs bundle structure on the holomorphic bundle associated to the complex structure $\d''+\theta'$ on $H$. 
\end{thm}

\begin{corollary}\label{cor:pull back preserve semisimplicity}
Let $f:Z \to X$ be a morphism from a smooth quasi-projective variety $Z$. Then $f^{\ast}\cV$ is semisimple.
\end{corollary}

\begin{proof}
Use Theorem \ref{thm:Corlette-Simpson}, Remark \ref{remark:semisimplicity via Zariski dense open subset} and the fact that pullback preserves harmonic metrics (see \cite[Page 18]{Simpson88}). 
\end{proof}

The following result is proved by Simpson, see \cite[Page 25, line 8 - line 13]{Simpson97}.
\begin{thm}[Simpson]\label{thm:harmonic representatives}
There is an isomorphism
\[ H^k(X,\cV) \cong \mathrm{Harm}^k(X,H)\colonequals\{ \alpha \in \cC^{\infty}(\cA^k_X\otimes_{\cC^{\infty}_X} H) \colon \Delta_{\nabla}(\alpha)=0\},\]
where $\Delta_{\nabla}$ is the Laplacian of $\nabla$. Moreover, for any $(a,b)\neq (0,0)$, there is an isomorphism
\begin{align*}
H^k(X,\cA^{\bullet}_X(H); aD'+bD'') &\cong \{ \alpha \in \cC^{\infty}(\cA^k_X\otimes_{\cC^{\infty}_X} H) \colon \Delta_{aD'+bD''}(\alpha)=0\}=\mathrm{Harm}^k(X,H).
\end{align*}
\end{thm}

\begin{proof}
For $(a,b)=(1,1)$,  the isomorphism between de Rham cohomology groups and the spaces of harmonic forms with coefficients in $H$ is already proved in \cite[\S 2]{Simpson88}. The proof for other $(a,b)$ is similar, as pointed out by Simpson in \cite[Page 25]{Simpson97}, so that we have
\[ H^k(X,\cA^{\bullet}_X(H); aD'+bD'') \cong \{ \alpha \in \cC^{\infty}(\cA^k_X\otimes_{\cC^{\infty}_X} H) \colon \Delta_{aD'+bD''}(\alpha)=0\}, \quad \forall (a,b)\neq (0,0).\]
It remains to identify different spaces of harmonic forms. For reader's convenience,  we give a brief review of K\"ahler identities following \cite[Page 51]{Sabbah}, where Sabbah's $D'_E,D_E'',\t'_E,\t''_E,D_{\infty},D_0$ are our $\d',\d'',\t',\t'',D',D''$ respectively. In our notation, one has
\[ \Delta_{D'}=\frac{1}{2}\Delta_{\nabla}, \quad\Delta_{D''+z_0D'}=\frac{1+|z_0|^2}{2}\Delta_{\nabla}, \quad \forall z_0 \in \C,\]
where the second equality comes from \cite[(2.2.5)]{Sabbah}. They imply that
\[\Delta_{aD'+bD''}=\frac{|a|^2+|b|^2}{2}\Delta_{\nabla}, \quad \forall (a,b)\in \C^2.\]
Therefore for any $(a,b)\neq (0,0)$, one has 
\[\{ \alpha \in \cC^{\infty}(\cA^k_X\otimes_{\cC^{\infty}_X} H) \colon \Delta_{aD'+bD''}(\alpha)=0\}=\{\alpha : \Delta_{\nabla}(\alpha)=0\}=\mathrm{Harm}^k(X,H).\]
\end{proof}
%\begin{proof}
%The key point is a generalized Kahler identity: 
%\[ \Delta_{\nabla}=2\Delta_{D'}=2\Delta_{D''},\]
%which is proved by Simpson in \cite{Simpson92}. Similar calculations also show that
%\[ \Delta_{aD'+bD''}= \frac{|a|^2+|b|^2}{2}\Delta_{\nabla}.\]
%\end{proof}

\begin{construction}\label{differential geometric construction}
Using Theorem \ref{thm:harmonic representatives}, Simpson \cite[\S 4]{Simpson97} defined a natural pure twistor structure $E^k$ of weight $k$ on $H^k(X,\cV)$ such that  
\begin{align*}
&E^k|_{z=[a,b]}\cong H^k(X,\cA^{\bullet}_X(H); aD'+bD'')\cong \mathrm{Harm}^k(X,H), \quad \forall [a,b]\in \bP^1.
\end{align*}
If we use the coordinate $z_0=[z_0,1]\in \bP^1\setminus\{\infty\}$, then
\begin{equation}\label{eqn: fiber of twistor structures}
E^k|_{z=z_0}\cong H^k(X,\cA^{\bullet}_X(H); z_0\d'+\theta'+\d''+z_0\theta'').
\end{equation}
In particular, $E^k|_{z=1}\cong H^k(X,\cV)$.
\end{construction}

\begin{remark}\label{remark:Identification map for Simpsons twistor structure}
The Identification map for the pure twistor structure $E^k$ in Construction \ref{differential geometric construction} can be calulated in the following way. Let $\psi\in H^k(X,\cV)$ and let $\sum_{p+q=k} \alpha^{p,q}\otimes m_{p,q}$ be a harmonic representative of $\psi$, where $\alpha^{p,q}$ are $(p,q)$-forms and $m_{p,q}$ are sections of $H$. Because choosing the harmonic representative gives the trivialization of $E^k$ over $\bP^1$, then
\begin{align*}
\mathrm{Iden}: E^k|_{z=1}\cong H^k(X,\cV) &\to \mathrm{Harm}^k(X,H) &\to H^k(X,\cA^{\bullet}_X(H);-D'+D'')\cong E^k|_{z=-1}, \\
\psi=\left[\sum_{p+q=k} \alpha^{p,q}\otimes m_{p,q}\right]_{H^k(X,\cV)} &\mapsto \sum_{p+q=k} \alpha^{p,q}\otimes m_{p,q} &\mapsto \left[\sum_{p+q=k} \alpha^{p,q}\otimes m_{p,q}\right]_{H^k(X,\cA^{\bullet}_X(H);-D'+D'')}
\end{align*}
%In particular, one cannot see the sign change, in contrast to Lemma \ref{lemma:Identification map for natural pure twistor structures}.

\end{remark}

\subsection{Cohomology of algebraic varieties and mixed twistor structures}\label{sec: cohomology of arbitrary}

In this subsection, we review Simpson's theory of weights following \cite[\S 5]{Simpson97} and prove the global invariant cycle Theorem \ref{thm: global invariant cycle}. 

\begin{definition}\label{definition:mixed twistor structure}  A \emph{mixed twistor structure} is a twistor structure $E$ which is filtered by an increasing sequence of strict subbundles $W_iE$ such that $\Gr^W_i(E)=W_iE/W_{i-1}{E}$ is a pure twistor structure of weight $i$ for all $i \in \mathbb{Z}$.
A morphism of mixed twistor structures is a morphism of filtered bundles on $\bP^1$ preserving the filtration. 
\end{definition}

\begin{thm}[Simpson]\label{thm:existence of mixed twistor structures}
Let $U$ be a quasi-projective algebraic variety. Let $\cV$ be a local system on $U$ coming from the restriction of a semisimple local system on a smooth projective compactification $\overline{U}\supseteq U$. Then $H^k(U,\cV)$ underlies a natural mixed twistor structure, which is functorial in $U$.
\end{thm}

In the remark under \cite[Theorem 5.2]{Simpson97}, Simpson said that the same theory of weights hold as in \cite[Theor\`eme 8.2.4]{Deligne74}. In particular, we have 
\begin{thm}[Simpson's theory of weights]\label{thm: yoga of weights}
With the assumptions in Theorem \ref{thm:existence of mixed twistor structures}. Let $W_{\bullet}$ denote the weight filtration on $H^k(U,\cV)$ induced by the natural mixed twistor structure. Then we have
\begin{itemize}
    \item if $U$ is proper, then $\Gr^W_iH^k(U,\cV)=0$ for $i> k$,
    \item if $U$ is smooth, then $\Gr^W_iH^k(U,\cV)=0$ for $i< k$,
    \item if $U$ is smooth and proper, then $\Gr^W_iH^k(U,\cV)=0$ for $i\neq  k$.
\end{itemize}
\end{thm} 

In \cite{Simpson97}, Simpson only proved the last statement of Theorem \ref{thm: yoga of weights}. For some proof details, see \cite[Corollary 2.3.87 and Corollary 2.3.97]{Yang21}.

\subsubsection{Mixed Twistor Complexes}\label{sec:mixed twistor complexes}
To construct the mixed twistor structures, Simpson used the notion of mixed twistor complexes, patching construction and an analytic construction for the pure twistor structures.
\begin{definition}\label{definition:mixed twistor complex}
A \textit{mixed twistor complex} is a filtered complex $(M^{\bullet},W^{\pre}_{\bullet})$ of sheaves of $\cO_{\bP^1}$-modules on $\bP^1$ such that $\cH^i(\Gr^{W^\pre}_{\ell}(M^{\bullet}))$ is a pure twistor structure of weight $\ell+i$. $W^{\pre}$ is called the pre-weight filtration.
\end{definition}
\begin{lemma}\label{corollary:description of the lowest weight}
Let $(M^{\bullet},W^{\pre}_{\bullet})$ be a mixed twistor complex. Suppose $\ell$ is the smallest integer so that $\Wpre_{\ell}M^{\bullet}\neq 0$, then
\[ W^{\pre}_{\ell}\cH^{k}(M^{\bullet}) = \mathrm{Im} \left\{\cH^{k}(W^{\pre}_{\ell}M^{\bullet}) \to \cH^{k}(M^{\bullet})\right\}, \]
where the map is induced by $W^{\pre}_{\ell}M^{\bullet} \hookrightarrow M^{\bullet}$.
\end{lemma}  

\begin{proof}
It follows from \cite[Lemma 5.3]{Simpson97} and \cite[Lemma 8.24]{Voisin}. For details, see \cite[Corollary 2.3.74]{Yang21}.
\end{proof}

Simpson introduced a patching construction to glue several filtered complexes over subsets of $\bP^1$ to a filtered complex over $\bP^1$ \cite[Page 30-32]{Simpson97}. For clarity, we briefly recall his construction. Let $M^{\bullet}, N^{\bullet}$ and $P^{\bullet}$ be filtered complexes of sheaves of $\cO$-modules, respectively over $\Omega_0, \Omega_{\infty}$ and $\Gm$.
Assume there are filtered quasi-isomorphisms $M^{\bullet}|_{\Gm} \xleftarrow{f} P^{\bullet} \xrightarrow{g} N^{\bullet}|_{\Gm}$, then we denote by
\[ \Patch(M\leftarrow P \rightarrow N)\colonequals \Cone(P^{\bullet}_{\ex}\xrightarrow{(f,-g)}M^{\bullet}_{\ex}\oplus N^{\bullet}_{\ex}),\]
where $M^{\bullet}_{\ex},N^{\bullet}_{\ex},P^{\bullet}_{\ex}$ are filtered complexes over $\bP^1$ induced by a fixed functorial choice of right derived functor of some extension functor for the inclusion $\Omega_0\hookrightarrow\bP^1$ (as well as for $\Omega_{\infty}$ and $\Gm$).
If there are filtered quasi-isomorphisms $M^{\bullet}|_{\Gm} \xrightarrow{f} P^{\bullet} \xleftarrow{g} N^{\bullet}|_{\Gm}$, then we define
\[ \Patch(M\rightarrow P \leftarrow N)\colonequals \Cone(M^{\bullet}_{\ex}\oplus N^{\bullet}_{\ex} \xrightarrow{(f,-g)}P^{\bullet}_{\ex})[-1].\]

There is also a version for gluing 5 complexes. Suppose $P^{\bullet},Q^{\bullet},R^{\bullet}$ are filtered complexes of $\cO$-modules over $\Gm$ and $M^{\bullet},N^{\bullet}$ are filtered complexes of $\cO$-modules over $\Omega_0$ and $\Omega_{\infty}$ respectively.  Assume we have filtered quasi-isomorphisms
\[ M^{\bullet}|_{\Gm} \leftarrow P^{\bullet} \xrightarrow{f} Q^{\bullet} \xleftarrow{g} R^{\bullet} \rightarrow N^{\bullet}|_{\Gm},\]
then we define
\[ \Patch(M,P,Q,R,N)\colonequals \Patch(M^{\bullet}\leftarrow \Cone(P^{\bullet}\oplus R^{\bullet}\xrightarrow{(f,-g)} Q^{\bullet})[-1]\rightarrow N^{\bullet}).\]

For our purposes, we need the following compatibility statement.
\begin{lemma}\label{lemma:3 equals 5}
Let $M^{\bullet}, N^{\bullet}$ and $P^{\bullet}$ be filtered complexes of sheaves of $\cO$-modules, respectively over $\Omega_0, \Omega_{\infty}$ and $\Gm$. Given a filtered quasi-isomorphisms of complexes 
\[ M^{\bullet}|_{\Gm} \xrightarrow{f} P^{\bullet} \xleftarrow{g} N^{\bullet}|_{\Gm},\]
then there is a natural filtered quasi-isomorphism 
\[ \Patch(M\rightarrow P \leftarrow N) \cong \Patch (M,M|_{\Gm},P,N|_{\Gm},N),\]
where the latter is induced by
\[ M^{\bullet}|_{\Gm} \xleftarrow{\mathrm{Id}} M^{\bullet}|_{\Gm} \xrightarrow{f} P^{\bullet} \xleftarrow{g} N^{\bullet}|_{\Gm} \xrightarrow{\mathrm{Id}} N^{\bullet}|_{\Gm}.\]
\end{lemma}

\begin{proof}
It is straightforward by definition of cones. For details, see \cite[Lemma 2.3.80]{Yang21}.
\end{proof}

It turns out that the Construction \ref{differential geometric construction} can be recovered using the patching construction. Now we work with Set-up \ref{setup: cohomology of smooth projective} and follow the notations in \cite[Page 23-24]{Simpson97}. Simpson constructed a triple $(\cF,\cL,\cF')$ associated to $\cV$ and showed that the bundle $R^kp_{2,\ast}(\xi \Omega^{\bullet}_X(\cF))$ over $\Omega_0$ and the bundle $R^kp_{2,\ast}(\xi \Omega^{\bullet}_{\overline{X}}(\cF'))$ over $\Omega_{\infty}$ can be identified over $\Gm$ with $Rp_{2,\ast}(\cL)$, and therefore glued to a holomorphic bundle $\oplus \cO_{\bP^1}(k)$, which is canonically isomorphic to the twistor structure in Construction \ref{differential geometric construction} (see \cite[Page 25-26]{Simpson97}). It can be summarized as follows.

\begin{lemma}\label{lemma:pure twistor structure via glueing}
Consider the complex of sheaves of $\cO_{\bP^1}$-modules with trivial filtrations:
\[ \Patch\left(Rp_{2,\ast}(\xi\Omega^{\bullet}_X(\cF)) \rightarrow Rp_{2,\ast}(\cL) \leftarrow Rp_{2,\ast}(\xi\Omega^{\bullet}_{\overline{X}}(\cF'))\right),\]
where the morphisms are induced by Dolbeault resolutions. Then the $k$-th cohomology of this complex is the natural pure twistor structure $E^k$ on $H^k(X,\cV)$.
\end{lemma}

\subsubsection{Mixed twistor structures on open varieties}\label{Mixed twistor structures on open varieties}
We work with Set-up \ref{setup: cohomology of smooth projective}. Let $j:U\hookrightarrow X$ be an open subset with a simple normal crossing boundary divisor $D=X\setminus U$. In \cite[Page 34]{Simpson97}, the functorial mixed twistor structure on $H^k(U,\cV|_U)\colonequals H^k(U,j^{\ast}\cV)$ is constructed as follows. Let $p$ denotes the second projection map to the $\bP^1$-direction for subsets of $X^{\mathrm{top}}\times \bP^1$. Using the triple $(\cF,\cL,\cF')$ associated to $\cV$ on $X$, he defined 5 filtered complexes:
\begin{align*}
M^{\bullet}&\colonequals (\xi\Omega^{\bullet}_X(\log D)\otimes_{\cO_{X\times \Omega_0}}\cF,\Wpre_{\bullet}),\quad  N^{\bullet}\colonequals (\xi\Omega^{\bullet}_{\overline{X}}(\log D)\otimes_{\cO_{\overline{X}\times \Omega_{\infty}}}\cF',\Wpre_{\bullet}),\\
P^{\bullet}&\colonequals (\xi\Omega^{\bullet}_X(\log D)\otimes_{\cO_{X\times \Gm}}\cF,\tau), \quad Q^{\bullet}\colonequals (\xi j_{\ast}\cA^{\bullet}_U\otimes_{p^{-1}\cO_{\Gm}}\cL,\tau),\\
R^{\bullet}&\colonequals (\xi\Omega^{\bullet}_{\overline{X}}(\log D)\otimes_{\cO_{\overline{X}\times \Gm}}\cF',\tau),
\end{align*}
where $\tau$ is the filtration induced by the truncation functor and $\Wpre$ is the filtration induced by the weight filtration on $\Omega^{\bullet}_X(\log D)$. Then the isomorphisms $\Omega_X^{\bullet}(\log D)\cong j_{\ast}\Omega^{\bullet}_U \cong j_{\ast}\cA^{\bullet}_U$ induces filtered isomorphisms
\[ M^{\bullet}|_{X\times \Gm} \leftarrow P^{\bullet}\to Q^{\bullet} \leftarrow R^{\bullet} \rightarrow N^{\bullet}|_{\overline{X}\times \Gm}.\]
This gives filtered quasi-isomorphisms on $\bG_m$:
\[ Rp_{\ast}M^{\bullet}|_{\bG_m} \leftarrow Rp_{\ast}P^{\bullet} \rightarrow Rp_{\ast}Q^{\bullet} \leftarrow Rp_{\ast}R^{\bullet} \to Rp_{\ast}N^{\bullet}|_{\bG_m}.\]
Define the patching complex
\[ \MTC(\cV|_U)\colonequals \mathrm{Patch}(Rp_{\ast}M^{\bullet},Rp_{\ast}P^{\bullet},Rp_{\ast}Q^{\bullet},Rp_{\ast}R^{\bullet},Rp_{\ast}N^{\bullet}).\]

\begin{thm}{\cite[Page 34]{Simpson97}}\label{thm: mixed twistor complex}
The complex $\MTC(\cV|_U)$ is a mixed twistor complex.
\end{thm}
As a corollary, for any integer $k$, $H^k(U,\cV|_U)$ underlies a natural mixed twistor structure, which is functorial with respect to $U\hookrightarrow X$.

\begin{remark}
To construct the mixed twistor structure on $H^k(U,\cV|_U)$, it is actually enough to use three complexes $P^{\bullet},Q^{\bullet},R^{\bullet}$, $P^{\bullet}$ and $R^{\bullet}$ naturally extend to $\Omega_0$ and $\Omega_{\infty}$ respectively) and consider $\Patch(Rp_{\ast}P^{\bullet} \to Rp_{\ast}Q^{\bullet} \leftarrow Rp_{\ast}R^{\bullet})$. The advantage of using the extra two complexes is that the explicit weight filtration on $\Omega^{\bullet}_X(\log D)$ is useful for the proof of Lemma \ref{lemma: lowest weight}.
\end{remark}
\begin{corollary}\label{cor:yoga of weights for open}
With Set-up \ref{setup: cohomology of smooth projective}. Let $j:U\hookrightarrow X$ be an arbitrary Zariski open subset. Then $H^k(U,\cV|_U)$ underlies a natural and functorial mixed twistor structure $(E^k_U,W_{\bullet}E^k_U)$. Moreover, for $[z_0,1]\in \bP^1\setminus\{0,\infty\}$, there is an isomorphism
\begin{equation}\label{eqn: fiber of twistor structure on open}
E^k_U|_{z=z_0}\cong H^k(U,\cA^{\bullet}_U(H|_U);(z_0\d'+\d''+\t'+z_0\t'')|_U),
\end{equation}
where $\d',\d'',\t',\t''$ are the operators associated to $\cV$ and $H$ in Theorem \ref{thm:Corlette-Simpson}.
\end{corollary}

\begin{proof}
First we assume $X\setminus U$ is a simple normal crossing divisor and use notation above. The first statement is obtained in Theorem \ref{thm: mixed twistor complex}. Lemma \ref{lemma:pure twistor structure via glueing} and \eqref{eqn: fiber of twistor structures} imply that
\[ H^k(X,\cL|_{X\times \{z_0\}})\cong H^k(X,\cA^{\bullet}_X(H);z_0\d'+\d''+\t'+z_0\t'').\]
Therefore, we have
\[ E^k_U|_{z=z_0}\cong (R^kp_{\ast}Q^{\bullet})_{z_0}=H^k(X,(\xi j_{\ast}\cA^{\bullet}_U\otimes \cL)|_{X\times \{z_0\}})\cong H^k(U,\cA^{\bullet}_U(H|_U);(z_0\d'+\d''+\t'+z_0\t'')|_U),\]
where $p:X^{\mathrm{top}}\otimes \Gm\to \Gm$ is the second projection. This proves \eqref{eqn: fiber of twistor structure on open} for $X\setminus U$ normal crossing.

For arbitrary $U\subseteq X$, consider the following diagram
\[ \begin{CD}
U @>\tilde{j}>> \tilde{X} \\
@V\mathrm{id} VV  @VV\pi V \\
U @>>j> X,
\end{CD}
\]
where $\pi: \tilde{X}\to X$ is a log resolution of $(X,X\setminus U)$, so that $\tilde{X}$ is smooth projective, $\pi$ is an isomorphism over $U$, and $\tilde{X}\setminus \pi^{-1}(U)$ is a simple normal crossing divisor. In particular, we can view $U$ as an open subset of $\tilde{X}$. Let $\tilde{\cV}=\pi^{\ast}\cV$ be the pullback local system. By Corollary \ref{cor:pull back preserve semisimplicity}, $\tilde{\cV}$ is semisimple and the corresponding bundle $\tilde{H}=\tilde{\cV}\otimes \cC^{\infty}_{\tilde{X}}$ admits the harmonic metric $\tilde{h}=\pi^{\ast}h$. Let $\tilde{\d}',\tilde{\d}'',\tilde{\t}',\tilde{\t}''$ be the operators associated to $(\tilde{H},\tilde{h})$. Since $\pi$ is an isomorphism over $U$, we have $\tilde{\cV}|_U=\cV|_U$ and $(\tilde{H},\tilde{h})|_U=(H,h)|_U$, and $(\tilde{\d}',\tilde{\d}'',\tilde{\t}',\tilde{\t}'')|_U=(\d',\d'',\t',\t'')|_U$.

Now we can apply the previous case to conclude that $H^k(U,\cV|_U)= H^k(U,\tilde{\cV}|_U)$ admits a natural mixed twistor structure $(\tilde{E}^k_U,W_{\bullet}\tilde{E}^k_U)$. Moreover, 
\begin{align*}
\tilde{E}^k_U|_{z=z_0}&\cong H^k(U,\cA^{\bullet}_U(\tilde{H}|_U);(z_0\tilde{\d}'+\tilde{\d}''+\tilde{\t}'+z_0\tilde{\t}'')|_U),\\
&=H^k(U,\cA^{\bullet}_U(H|_U);(z_0\d'+\d''+\t'+z_0\t'')|_U).
\end{align*}
This proves \eqref{eqn: fiber of twistor structure on open} for arbitrary $U\subseteq X$ and finishes the proof of the corollary.
\end{proof}

\begin{lemma}\label{lemma: lowest weight}
With the notations in Corollar \ref{cor:yoga of weights for open}, then
\[ W_{k}H^k(U,\cV|_{U})=j^{\ast}H^k(X,\cV).\]
\end{lemma}

\begin{proof}
First, consider the following commutative diagram
\[ \begin{CD}
U @>\tilde{j}>> \tilde{X} \\
@V\mathrm{id} VV  @VV\pi V \\
U @>>j> X,
\end{CD}
\]
where $\pi$ is a log resolution so that $\tilde{X}\setminus U$ is a normal crossing divisor. It is easy to see that $\tilde{j}^{\ast}H^k(\tilde{X},\pi^{\ast}\cV)=j^{\ast}H^k(X,\cV)$ by looking at the Leray spectral sequence for the map $\pi:\tilde{X} \to X$ and $\pi^{\ast}\cV$. Since semisimplicity is preserved under pull-backs (see Corollary \ref{cor:pull back preserve semisimplicity}), we can assume $X\setminus U$ is normal crossing.  Let $E^k$ be the natural twistor structure on $H^k(X,\cV)$. By Lemma \ref{lemma:pure twistor structure via glueing} and Lemma \ref{lemma:3 equals 5},
\begin{align*}
E^k&\cong \cH^k\Patch\left(Rp_{\ast}(\xi\Omega^{\bullet}_X(\cF)) \rightarrow Rp_{\ast}(\cL) \leftarrow Rp_{\ast}(\xi\Omega^{\bullet}_{\overline{X}}(\cF'))\right)\\
&\cong \cH^k\Patch\left(Rp_{\ast}(\xi\Omega^{\bullet}_X(\cF)), Rp_{\ast}( \xi\Omega^{\bullet}_X(\cF))|_{\Gm},Rp_{\ast}(\cL), Rp_{\ast}(\xi\Omega^{\bullet}_{\overline{X}}(\cF'))|_{\Gm}, Rp_{\ast}(\xi\Omega^{\bullet}_{\overline{X}}(\cF'))\right)
\end{align*}
By Theorem \ref{thm: mixed twistor complex}, the map $j^{\ast}:H^k(X,\cV) \to H^k(U,\cV|_U)$ can be lifted to a morphism of mixed twistor structures, which is induced by the following morphisms of complexes:
\[ \xi\Omega^{\bullet}_X\otimes \cF \to  \xi\Omega^{\bullet}_X(\log D)\otimes \cF, \quad \cL \to \xi j_{\ast}\Omega^{\bullet}_U\otimes \cL, \quad   \xi\Omega^{\bullet}_{\overline{X}}\otimes \cF' \to  \xi\Omega^{\bullet}_{\overline{X}}(\log D)\otimes \cF'.\]
By Lemma \ref{corollary:description of the lowest weight}, $W_kH^k(U,\cV|_U)$ underlies the gluing of the images of the inclusion maps over $\Omega_0,\Gm,\Omega_{\infty}$:
\begin{itemize}

\item $R^kp_{\ast}(W^{\pre}_0\xi\Omega^{\bullet}_{X}(\log D)\otimes \cF) \to R^kp_{\ast}(\xi \Omega^{\bullet}_{X}(\log D)\otimes \cF)$,

\item $R^kp_{\ast}(\tau_0\xi\Omega^{\bullet}_{X}(\log D)\otimes \cF) \to R^kp_{\ast}(\xi \Omega^{\bullet}_{X}(\log D)\otimes \cF)$,

\item $R^kp_{\ast}({\tau}_0\xi j_{\ast}\Omega^{\bullet}_U\otimes \cL) \to R^kp_{\ast}(\xi j_{\ast}\Omega^{\bullet}_U\otimes \cL)$.
\end{itemize}
By the residue map in \cite[Page 36]{Simpson97}, we have 
\[ W^{\pre}_0\xi\Omega^{\bullet}_{X}(\log D)\otimes \cF\cong \xi \Omega^{\bullet}_X\otimes \cF.\]
Topological calculations from \cite[3.1.8.1]{Deligne71} show that
\[ {\tau}_0\xi\Omega^{\bullet}_{X}(\log D)\otimes \cF=\cH^0(\xi\Omega^{\bullet}_{X}(\log D)\otimes \cF) \cong \xi\Omega^{\bullet}_X\otimes \cF,\]
and by definition of the truncation functor
\[ {\tau}_0\xi j_{\ast}\Omega^{\bullet}_U\otimes \cL=\cH^0(\xi j_{\ast}\Omega^{\bullet}_U\otimes \cL) \cong \cL. \]
In particular, on each corresponding subset of $\bP^1$, the lowest weight filtration coincides with the image of restriction from $X$. Since the patching construction preserves the quasi-isomorphism, the lemma follows. 
\end{proof}

Now we can prove the global invariant cycle theorem.
\begin{proof}[Proof of Theorem \ref{thm: global invariant cycle}]
Theorem \ref{thm: yoga of weights} implies that
\begin{itemize}
    \item if $i>k$, then $\Gr^W_i H^k(Z,(j\circ \alpha)^{\ast}\cV)=0$ [$Z$ is proper],
    \item if $i<k$, then $\Gr^W_i H^k(U,j^{\ast}\cV) =0$ [$U$ is smooth].
\end{itemize}
By the strictness of morphism between mixed twistor structures and Lemma \ref{lemma: lowest weight}, we have
\begin{align*}
    \textrm{Im } \alpha^{\ast}&=\textrm{Im } \alpha^{\ast}\cap W_{k}H^k(Z,(j\circ \alpha)^{\ast}\cV)\\
    &=\alpha^{\ast}(W_{k}H^k(U,j^{\ast}\cV))=\alpha^{\ast}j^{\ast}H^k(X,\cV)=\textrm{Im } (j\circ \alpha)^{\ast}.
\end{align*}
\end{proof}

\subsection{Twisted Poincar\'e pairings and polarizations}\label{sec: pre-Weil operator}

In this subsection, we construct the pre-Weil operator $\phi$ using harmonic metrics and Sabbah's rescaling map. We show that the twisted Poincar\'e pairing can be used to polarize the pure twistor structures on the cohomology groups semisimple local systems. 

Let $U$ be a smooth quasi-projective variety and $\cV$ be a semisimple local system on $U$ with a flat connection $\nabla$, which comes from the restriction of a semisimple local system on a smooth projective compactification of $U$.
By Theorem \ref{thm:Corlette-Simpson}, there is a harmonic bundle $(H,h)$ associated to $\cV$ with the decomposition $\nabla=\partial'+\theta'+\partial''+\theta''$, where $h$ is unique up to a linear transformation. We view the harmonic metric $h$ as a $\cC^{\infty}_U$-linear morphism
\[ h: H\otimes_{\cC^{\infty}_U}\overline{H} \to \cC^{\infty}_U.\]
Let $\cV^{\ast}$ be the dual local system of $\cV$ and $H^{\ast}$ be the dual $\cC^{\infty}$-bundle of $H$. In \cite[Page 14]{Simpson92}, Simpson showed that $H^{\ast}$ is equipped with the following harmonic structure so that it is the harmonic bundle associated to $\cV^{\ast}$. Following Simpson, we will abuse notation and still use $D'',\nabla$ etc to represent the corresponding operator on $H^{\ast}$. Note that Simpson's dual construction in \cite{Simpson92} actually works without the compactness assumption. This is crucial for us, because in the proof of Decomposition theorem one needs to deal with non-compact varieties.
\begin{construction}[Dual harmonic bundle]\label{construction:dual harmonic bundle}
The dual metric $h^{\ast}:H^{\ast} \otimes \overline{H^{\ast}} \to \cC^{\infty}_U$ is defined as follows: for two sections $\lambda,\mu \in \cC^{\infty}(H^{\ast})$, set
\[ h^{\ast}(\lambda,\overline{\mu}) \colonequals \lambda(e),\]
where $e$ is the unique section in $\cC^{\infty}(H)$ satisfying $\mu(\bullet)=h(\bullet,\overline{e})$. The dual Higgs operator $D''=\d''+\t'$ is defined by
\[ (\t'\lambda)(e)+\lambda(\t'e) = 0, \quad (\d''\lambda)(e) + \lambda(\d''e)=\overline{\d} (\lambda(e)),\]
where $\lambda \in \cC^{\infty}(H^{\ast}), e \in \cC^{\infty}(H)$. The dual flat connection $\nabla$ on $\cV^{\ast}$ is defined by 
\[ \nabla \colonequals \d'+\theta'+\d''+\theta'',\]
where $\d'+\d''$ is a metric connection for $h^{\ast}$ and $\theta''$ is the adjoint of $\theta'$ with respect to $h^{\ast}$.
\end{construction}

\begin{remark}
We want  the reader to be aware that the dual connection $\nabla$ on $\cV^{\ast}$ is \textit{not} the dual of $\nabla$ on $\cV$ with respect to the harmonic metric $h$.
\end{remark}

\begin{lemma}\label{lemma: fiber at z=-1}
With the notations above, the harmonic metric $h$ induces a 
$\C$-linear isomorphism between $\cC^{\infty}_U$-bundles:
\[ h: \overline{H} \to H^{\ast}, \quad \overline{e}\mapsto h(\bullet,\overline{e}), \quad \forall e \in \cC^{\infty}(H),\]
so that it induces a quasi-isomorphism
\[ h:\overline{(\cA^{\bullet}_U(H),\d'-\theta'+\d''-\theta'')} \xrightarrow{\sim} \cV^{\ast},\]
and an isomorphism
\[ h: \overline{H^k(U,\cA^{\bullet}_U(H);\d'-\theta'+\d''-\theta'')} \xrightarrow{\sim} H^k(U,\cV^{\ast}).\]
\end{lemma}

\begin{proof}
There is a commutative diagram
\begin{equation*}
\begin{CD}
\overline{H} @>h >> H^{\ast} \\
@VVh^{\ast}(\nabla) V @VV\nabla V \\
\overline{H}\otimes \cA^1_U @>h\otimes \mathrm{Id}>> H^{\ast}\otimes \cA^1_U.
\end{CD}
\end{equation*}
We claim that
\begin{equation*} 
h^{\ast}(\nabla)= \overline{\d''}-\overline{\t''}+\overline{\d'}-\overline{\t'},
\end{equation*}
and each operator in the equation is the conjugated operator on $\overline{H}$ with respect to $\nabla$ on $H^{\ast}$. Granting this for now, since $H^{\ast}$ is the harmonic bundle associated to $\cV^{\ast}$ and the de Rham complex $\cA^{\bullet}_U(H^{\ast})$ is determined by the first order differential operator $\nabla$, we see that $h$ induces a quasi-isomorphism $\overline{\cV^{\ast}}\cong (\cA^{\bullet}_X(H),\d'-\theta'+\d''-\theta'')$.

To prove the claim, using the isomorphism $h:\overline{H}\to H^{\ast}$, we define a map so that any operator $A: H^{\ast} \to H^{\ast}\otimes \cA^1_U$ is sent to the unique operator $B: \overline{H} \to \overline{H}\otimes \cA^1_U$ satisfying
\[ A\circ h= (h\otimes \mathrm{Id}) \circ B.\]
Let us see how each operator in $\nabla=\d'+\t'+\d''+\t''$ changes under the map defined above. We fix $e_1,e_2\in\cC^{\infty}(H)$ and $\lambda=h(\bullet,\overline{e_2})\in \cC^{\infty}(H^{\ast})$.

\begin{enumerate}

\item [(A)] $\t'\mapsto -\overline{\t''}$. We have
\[ (\t'\lambda)(e_1)=-h(\t' e_1,\overline{e_2})=h(e_1,\overline{-\t''e_2}).\]
The last equality uses that $\theta'$ is the adjoint of $\t''$. Hence
\[ \t'h(\bullet,\overline{e_2})=h(\bullet,\overline{-\t''}\overline{e_2}),\]
which means that the image of $\t'$ is $-\overline{\t''}$.

\item [(B)] $\d'' \mapsto \overline{\d'}$. We have
\[ (\d''\lambda)(e_1)=\dbar h(e_1,\overline{e_2})-h(\d'' e_1,\overline{e_2})=h(e_1,\overline{\d' e_2}).\]
The last equality uses that $\d'+\d''$ is the metric connection with respect to $h$. Hence
\[\d''h(\bullet,\overline{e_2})=h(\bullet,\overline{\d'}\overline{e_2}),\] 
which means that $\d''$ is mapped to $\overline{\d'}$.

\item [(C)] $\t'' \mapsto -\overline{\t'}$ and $\d' \mapsto \overline{\d''}$ can be verified similarly using their definition via the dual harmonic metric $h^{\ast}$.
\end{enumerate}
\end{proof}

To define the pre-Weil operator $\phi$, we recall Sabbah's rescaling map in \cite[equation (2.2.6)]{Sabbah}.
\begin{definition}\label{definition:rescaling map} Define a map of complexes
\begin{align*}
\iota: \C[z]\otimes_{\C} \cA^{p,q}_U \otimes_{\cC^{\infty}_U}H &\to z^{-p}\C[z]\otimes_{\C}\cA^{p,q}_U \otimes_{\cC^{\infty}_U}H\\
\alpha^{p,q}\otimes m &\mapsto z^{-p}\alpha^{p,q}\otimes m.
\end{align*}
Under $\iota$, the differential $z\d'+\theta'+\d''+z\theta''$ is changed into $\d'+z^{-1}\theta'+\d''+z\theta''$. For $z_0\in \Gm$, the evaluation at $z=z_0$ gives
\begin{align*}
\iota_{z_0}: \cA^{p,q}_U \otimes_{\cC^{\infty}_U}H \to \cA^{p,q}_U \otimes_{\cC^{\infty}_U}H, \quad \alpha^{p,q}\otimes m \mapsto z^{-p}_0\alpha^{p,q}\otimes m.
\end{align*}
Abusing notations, we denote the induced map on the cohomology by
\[ \iota_{z_0}: H^k(U,\cA^{\bullet}_U(H);z_0\d'+\theta'+\d''+z_0\theta'') \xrightarrow{\sim} H^k(U,\cA^{\bullet}_U(H);\d'+z^{-1}_0\theta'+\d''+z_0\theta''). \]
\end{definition}
Let $U$ be a smooth quasi-projective variety and $\cV$ be a local system on $U$ coming from the restriction of a semisimple local system on a smooth projective compactification of $U$. By Corollary \ref{cor:yoga of weights for open}, there exists a natural twistor structure $E^k$ on $H^k(U,\cV)$ so that
\[ E^k|_{z=-1}\cong H^k(U,\cA^{\bullet}_U(H); -\d'+\t'+\d''-\t'').\]
\begin{definition}\label{definition: pre-Weil operator}
With the notation above. Let us fix the choice of a harmonic metric on $H$. The \emph{pre-Weil operator} $\phi$ associated to $(U,\cV)$, at the level of de Rham complexes, is defined by
\begin{align*}
\phi:\overline{(\cA^{\bullet}_U(H); -\d'+\t'+\d''-\t'')}
\xrightarrow{\overline{\iota_{-1}}} &\overline{(\cA^{\bullet}_U(H);\d'-\theta'+\d''-\theta'')}\xrightarrow{h} \cV^{\ast}.
\end{align*}
The pre-Weil operator at the level of cohomology is the induced map
\begin{align*}
\phi:\overline{E^k|_{z=-1}}
\xrightarrow{\overline{\iota_{-1}}} &\overline{H^k(U,\cA^{\bullet}_U(H);\d'-\theta'+\d''-\theta'')}\xrightarrow{h} H^k(U,\cV^{\ast}),
\end{align*}
where $h$ is the map in Lemma \ref{lemma: fiber at z=-1} and $\iota_{-1}$ is Sabbah's rescaling map.
\end{definition}

\begin{remark}
In the rest of paper, we also refer pre-Weil operators to these isomorphisms induced by $\phi$ through perverse cohomology functors and smooth restrictions.
\end{remark}
\begin{prop}\label{cor:generalized Weil operator} Let $X$ be a smooth projective variety and $\cV$ be a semisimple local system. Then the map
\[ C:\overline{H^k(X,\cV)} \cong \overline{E|_{z=1}} \xrightarrow{\overline{\mathrm{Iden}}} \overline{E|_{z=-1}} \xrightarrow{\phi} H^k(X,\cV^{\ast}),\]
recovers the Weil operator for the pure Hodge structure on $H^k(X,\C)$, sending $\overline{\alpha^{p,q}}$ to $(-1)^p\alpha^{p,q}$ for $\alpha^{p,q} \in H^{p,q}(X,\C)$, where $\mathrm{Iden}$ is the Identification map in Definition \ref{definition:Identification map}.
\end{prop}

\begin{proof}
When $\cV=\C$, the associated harmonic metric is the constant Hermitian metric. Therefore the map $h$ in Lemma \ref{lemma: fiber at z=-1} is an identity under the identification $\overline{\C}\cong \C^{\ast}$. Then the corollary follows from Remark \ref{remark:Identification map for Simpsons twistor structure} and the construction of the rescaling map in Definition \ref{definition:rescaling map}.
\end{proof}

\begin{lemma}\label{lemma:Identification map is functorial} With the notation in Definition \ref{definition: pre-Weil operator}. The pre-Weil operator $\phi$ is functorial with respect to smooth restrictions and is compatible with cup products:

\begin{itemize}

\item if $N$ is any line bundle on $U$, there is a commutative diagram:
\[ \begin{CD}
  \overline{E^k|_{z=-1}}  @>\overline{F|_{z=-1}}>>  \overline{E^{k+2\ell}|_{z=-1}} \\
  @VV\phi V @VV\phi V \\
  H^k(U,\cV^{\ast}) @>\wedge c_1(N)^{\ell}>>H^{k+2\ell}(U,\cV^{\ast})
\end{CD}\]
where $F:E^k \to E^{k+2\ell}$ is the morphism of twistor structures induced by the cup product with $c_1(N)^{\ell}$, 

\item for any smooth subvariety $T\subseteq U$, denote $E^k_T$ to be the associated twistor structure, then there is a commutative diagram:
\[ \begin{CD}
  \overline{E^k|_{z=-1}}  @>\overline{G|_{z=-1}}>>  \overline{E^k_T|_{z=-1}} \\
  @VV\phi V @VV\phi_TV \\
  H^k(U,\cV^{\ast}) @>\widetilde{R}>>H^k(T,\cV^{\ast}|_T)
\end{CD}\]
where $G:E^k\to E^k_T$ and $\widetilde{R}$ are restriction maps.
\end{itemize}

\end{lemma}

\begin{proof}
The first statement follows from the compatibility of $\phi$ with the cup product with $c_1(N)^{\ell}$ on the level of de Rham complex $\cA_U^{\bullet}(H)$. 

For the second statement, denote by $(H,h)$ the harmonic bundle associated to the semisimple local system $\cV$ on $U$. Consider the restriction $\cV_T=\cV|_T$ and $(H_T,h_T)=(H,h)|_T$. Because harmonic metrics are preserved under pullback (see \cite[Page 18]{Simpson88}), we know that $h_T$ is a harmonic metric on $H_T$. This induces the complex $(\cA^\bullet_T(H_T); \lambda\d'_T+\t'_T+\d''_T+\lambda\t''_T),$ with the differential operators $\d'_T,$ etc., which are associated with the harmonic metric $h_T$. It is clear that they are equal to the restriction of $\d'$, etc. to $T$. This implies that the pre-Weil operator
\[ \overline{H^k(U,\cA^{\bullet}_U(H);-\d'+\theta'+\d''-\theta'')}\xrightarrow{h\circ \overline{\iota_{-1}}} H^k(U,\cV^{\ast}) \]
is compatible with restriction to $T$.

It remains to prove that the isomorphism $E^k|_{z=-1}\cong H^k(U,\cA^{\bullet}_U(H); -\d'+\t'+\d''-\t'')$ from \eqref{eqn: fiber of twistor structure on open} in Corollary \ref{cor:yoga of weights for open} is also compatible with $T\hookrightarrow U$. To do this, we need to recall the construction of this isomorphism. Let $X$ be the smooth projective compactification of $U$ so that $\cV$ is the restriction of a semisimple local system $\cV_X$ on $X$. As in the proof of Corollary \ref{cor:yoga of weights for open}, using sufficiently many blow-ups of $X$ which do not change $T$ and $U$, we can assume that we have the following diagram
\[ \begin{tikzcd}
    T \arrow[r,hook] \arrow[d,hook,"j_T"] &U \arrow[d,hook,"j_U"] \\
    \overline{T} \arrow[r,hook,"i"] &X
\end{tikzcd}
\]
where $X\setminus U$ and $\overline{T}\setminus T$ are simple normal crossing divisors in $X$ and $\overline{T}$, respectively and $\overline{T}\subseteq X$ is a smooth projective variety containing $T$.

Let $\cL_X$ be the holomorphic family of local systems on $X^{\mathrm{top}}\times \Gm$ constructed by Simpson \cite[Page 23-24]{Simpson97}, as a part of the triple $(\cF_X,\cL_X,\cF_X')$ associated to the local system $\cV_X$. Denote by $H_X$ the associated harmonic bundle on $X$ and $E^k_X$ the pure twistor structure on $H^k(X,\cV_X)$. We denote by $()_{\overline{T}}$ for the corresponding objects of $\overline{T}$. Then $\cL_{\overline{T}}=(i\times\mathrm{id})^{\ast}\cL_X$ and $H_{\overline{T}}=i^{\ast}H_X$. Lemma \ref{lemma:pure twistor structure via glueing} and \eqref{eqn: fiber of twistor structures} imply that
\begin{equation}\label{eqn: identification of EkX at -1}
E^k_X|_{z=-1}=H^k(X,\cL_X|_{X\times \{-1\}})\cong H^k(X,\cA^{\bullet}_X(H_X);-\d'_X+\d''_X+\t'_X-\t''_X).
\end{equation}
The same holds for $\overline{T}$. The natural map $\cL_X\to (i\times \mathrm{id})_{\ast}\cL_{\overline{T}}$ and $\cA^{\bullet}_X(H_X) \to i_{\ast}i^{\ast}\cA^{\bullet}_X(H_X)\cong_{\mathrm{qiso}} i_{\ast}\cA^{\bullet}_{\overline{T}}(H_{\overline{T}})$ induce the restriction map $E^k_X\to E^k_{\overline{T}}$ and
\[ H^k(X,\cA^{\bullet}_X(H_X);-\d'_X+\d''_X+\t'_X-\t''_X)\to H^k(\overline{T},\cA^{\bullet}_{\overline{T}}(H_{\overline{T}});-\d'_{\overline{T}}+\d''_{\overline{T}}+\t'_{\overline{T}}-\t''_{\overline{T}})\]
so that they are compatible with the isomorphism \eqref{eqn: identification of EkX at -1} and the one for $\overline{T}$.

Now note that the twistor structures $E^k_T$ and $E^k$ are induced by $\cL_{\overline{T}}$ and $\cL_{X}$ so that
\[ E^k|_{z=z_0}\cong H^k(X,(\xi j_{U,\ast}\cA^{\bullet}_U)\otimes \cL_X)|_{X\times \{z_0\}}),\quad E^k_T|_{z=z_0}\cong H^k(\overline{T},(\xi j_{T,\ast}\cA^{\bullet}_T)\otimes \cL_{\overline{T}}|_{\overline{T}\times\{z_0\}}).\]
This gives the desired compatiblity with $T\hookrightarrow U$, because all of them are induced by $\cL_X$ and all harmonic bundles $H_{-}$ (for $-=U,T,\overline{T}$) are restrictions of $H_X$. This finishes the proof of the Lemma.
%By definition, we have  
%$$E^k_T|_{z=\lambda}\cong H^k(T,\cA^{\bullet}_T(H_T); \lambda\d'_T+\t'_T+\d''_T+\lambda\t''_T).$$
%Due to the compatibility of pushforward functor and de Rham functor, we have the quasi-isomorphism
%$$i_*(\cA^\bullet_T(H_T); \lambda\d'_T+\t'_T+\d''_T+\lambda\t''_T)\simeq (\cA^\bullet_U(i_*H_T); \lambda\d'+\t'+\d''+\lambda\t'').
%$$
%By adjunction, we get the morphism 
%$$(\cA^\bullet_U(H); \lambda\d'+\t'+\d''+\lambda\t'')\to i_*(\cA^\bullet_T(H_T); \lambda\d'_T+\t'_T+\d''_T+\lambda\t''_T).
%$$
%By taking $\lambda=-1$ and taking hypercohomologies, the morphism above induces $G|_{z=-1}$; and replacing $(H, h)$ by its dual, taking $\lambda=1$ and taking hypercohomologies, we get $\widetilde{R}$.
%The compatibility with pre-Weil operator is straightforward by Definition \ref{definition: pre-Weil operator}.

%If we view $h$ as a $\C$-linear isomorphism 
%\[ h:\overline{H} \to H^{\ast},\]
%then the restriction of this morphism to $T$ is exactly the $\C$-linear isomorphism $h_T:\overline{H}_T\to H^{\ast}_T$ induced by the metric $h_T$, because the conjugation and dual operators commute with restriction. Therefore the map $h$ in Lemma \ref{lemma: fiber at z=-1} commutes with restriction. By definition, Sabbah's rescaling operator $\iota_{-1}$ clearly commutes with restriction. Finally, we need to argue that the isomorphism $E^k|_{z=-1}\cong H^k(U,\cA^{\bullet}_U(H); -\d'+\t'+\d''-\t'')$ commutes with restriction. ??
\end{proof}

In the rest of this subsection, we work with Set-up \ref{setup: cohomology of smooth projective} and assume $k\leq \dim X$.
\begin{definition}\label{definition:twisted Poincare pairing}
We define the \emph{twisted Poincar\'e pairing}  $S$ to be the bilinear pairing
\begin{align*}
S: H^k(X,\cV) &\otimes_{\C} H^k(X,\cV^{\ast}) \to \C \\
 [\alpha \otimes e] &\otimes [\beta \otimes \lambda] \mapsto (-1)^{k(k-1)/2} \int_X  \lambda(e)\cdot c_1(\eta)^{\dim X-k}\wedge \alpha \wedge \beta.
\end{align*}
Here $\alpha,\beta$ are $k$-forms on $X$ and $e,\lambda$ are global sections of $H$ and $H^{\ast}$ respectively.
%We often rewrite $S$ as 
%\[ S: H^k(X,\cV) \otimes_{\C} \overline{H^k(X,\overline{\cV}^{\ast})} \to \C\]
%using the canonical isomorphism $H^k(X,\cV^{\ast})\cong \overline{H^k(X,\overline{\cV}^{\ast})}$.
 \end{definition}
Let $E^k_{\prim}$ be the natural pure sub-twistor structure on $H^k(X,\cV)_{\prim}$ induced by the one on $H^k(X,\cV)$. We show that one can use $S$ and $\phi$ to polarize $E^k_{\prim}$.
\begin{thm}\label{cor:polarized by twisted Poincare pairing}
The pre-Weil operator $\phi$ restricts to
\[ \phi: \overline{E^k_{\prim}|_{z=-1}} \xrightarrow{\sim} H^k(X,\cV^{\ast})_{\prim}\colonequals \Ker \eta^{\dim X-k+1}\subseteq H^k(X,\cV^{\ast}).\]
Moreover, $E^k_{\prim}$ is polarized by the bilinear form
\[ i^{-k}\cdot S(\bullet,\phi(\bullet)):E^k_{\prim}|_{z=1}\otimes \overline{E^k_{\prim}|_{z=-1}} \xrightarrow{\mathrm{Id}\otimes \phi} H^k(X,\cV)_{\prim}\otimes H^k(X,\cV^{\ast})_{\prim} \xrightarrow{S} \C. \]
in the sense of Definition \ref{definition:pure twistor structures polarized by bilinear pairings}. Similarly, $i^{-(k-2m)}(-1)^{m(2m-2k+1)}\cdot S(\bullet,\phi(\bullet))$ polarizes the natural pure twistor structure $F$ on the primitive space $\eta^mH^{k-2m}(X,\cV)_{\prim}$, with the induced isomorphism $\phi : \overline{F|_{z=-1}} \xrightarrow{\sim} \eta^mH^{k-2m}(X,\cV^{\ast})_{\prim}$. 
\end{thm}

\begin{proof}
We will only deal with the case of $H^k(X,\cV)_{\prim}$ and leave other cases to the reader. The statement on restriction of $\phi$ follows from Lemma \ref{lemma:Identification map is functorial}. For any element in $H^k(X,\cV)$, consider its harmonic representative $\sum_{p+q=k} \alpha^{p,q}\otimes m_{p,q}$, where $\alpha^{p,q}$ are $(p,q)$-forms and $m_{p,q}$ are sections of the associated harmonic bundle $H$. By Remark \ref{remark:Identification map for Simpsons twistor structure}, we see that
\[ \phi(\overline{\mathrm{Iden}}\overline{\left[\sum_{p+q=k} \alpha^{p,q}\otimes m_{p,q}\right]})=\phi(\overline{\left[\sum_{p+q=k} \alpha^{p,q}\otimes m_{p,q}\right]})=\sum_{p+q=k}(-1)^p\left[\overline{\alpha^{p,q}}\otimes m_{p,q}^{\vee}\right]\]
where $m_{p,q}^{\vee}\in \cC^{\infty}(H^{\ast})$ is a section satisfying $m_{p,q}^{\vee}(\bullet)=h(\bullet,\overline{m_{p,q}})$. Therefore,
\begin{align*}
&i^{-k}\cdot S\left(\left[\sum \alpha^{p,q}\otimes m_{p,q}\right],\phi\circ \overline{\mathrm{Iden}}\overline{\left[\sum \alpha^{p,q}\otimes m_{p,q}\right]}\right)\\
&=\sum_{p+q=k} i^{-k}\cdot S\left([\alpha^{p,q}\otimes m_{p,q}],[(-1)^p\overline{\alpha^{p,q}}\otimes m_{p,q}^{\vee}]\right)\\
&=\sum_{p+q=k} (-1)^p\cdot i^{-k}(-1)^{k(k-1)/2}\int_X  h(m_{p,q},\overline{m_{p,q}})\cdot c_1(\eta)^{n-k}\wedge \alpha^{p,q}\wedge \overline{\alpha^{p,q}}\\&=\sum_{p+q=k} i^{p-q}(-1)^{k(k-1)/2}\int_X  h(m_{p,q},\overline{m_{p,q}})\cdot c_1(\eta)^{n-k}\wedge \alpha^{p,q}\wedge \overline{\alpha^{p,q}}>0.
\end{align*}
The positivity follows from the classical calculation of the Hodge star operators for primitive forms \cite[Theorem 6.29]{Voisin}.
\end{proof}

\begin{corollary}\label{cor: demonstration of polarization of subtwistor}
Let $T\subseteq X$ be a smooth   subvariety. Consider the restrictions maps
\[ R:H^k(X,\cV)_{\prim} \to H^k(T,\cV|_T), \quad \widetilde{R}:H^k(X,\cV^{\ast})_{\prim} \to H^k(T,\cV^{\ast}|_T).\]
Then the twisted Poincar\'e pairing restricts to a non-degenerate pairing
\[ S: \Ker R\otimes \Ker \widetilde{R} \to \C.\]
\end{corollary}

\begin{proof}
It follows from Theorem \ref{cor:polarized by twisted Poincare pairing}, the compatibility Lemma \ref{lemma:Identification map is functorial} and Lemma \ref{lemma:polarization of subtwistors}.
\end{proof}
We conclude this subsection by reinterpreting the pre-Weil operator $\phi$ via the Hodge star operator.
\begin{lemma}\label{lemma:identification is Hodge star operator}
 Assume that $\eta$ is an ample line bundle associated to a K\"ahler metric $g$ on $X$. Consider 
\[ (\eta^{n-k})^{-1}\circ\overline{ \ast} : \overline{H^k(X,\cV)_{\prim}} \to H^{2n-k}(X,{\cV}^{\ast})_{\prim} \to H^k(X,{\cV}^{\ast}_{\prim}).\]
where $\ast$ is the Hodge star operator in Definition \ref{definition:C-linear Hodge star operator}.
Then this map can be identified up to a scaling constant with the composition map
\[ \phi \circ \overline{\mathrm{Iden}}: \overline{H^k(X,\cV)_{\prim}}= \overline{E^k_{\prim}|_{z=1}} \to  \overline{E^k_{\prim}|_{z=-1}}\to H^k(X,\cV^{\ast})_{\prim}. \]
\end{lemma}

\begin{proof}
Let $\sum \alpha^{p,q}\otimes m_{p,q}$ be a primitive harmonic form, where $\alpha^{p,q}$ is a form and $m_{p,q}$ is a global section of $H$. By Lemma \ref{lemma:Hodge star operator for primitive forms} we have 
\[ (\eta^{n-k})^{-1}\circ \ast( \alpha^{p,q}\otimes m_{p,q}) = \frac{(-1)^{k(k+1)/2}i^{p-q}}{(n-k)!}\alpha^{p,q}\otimes \overline{m_{p,q}^{\vee}}=C \cdot(-1)^p\alpha^{p,q}\otimes \overline{m_{p,q}^{\vee}}.\]
Here $C=\frac{(-1)^{k(k+1)/2}i^{-k}}{(n-k)!}$ and $m_{p,q}^{\vee}$ is the section of $H^{\ast}$ so that $m_{p,q}^{\vee}(\bullet) =h(\bullet,\overline{m_{p,q}})$. On the other hand, the proof of Theorem \ref{cor:polarized by twisted Poincare pairing} implies that
\[ \phi\left(\overline{\mathrm{Iden}}\overline{\left[\sum \alpha^{p,q}\otimes m_{p,q}\right]}\right) = \sum(-1)^p \overline{\alpha^{p,q}}\otimes m_{p,q}^{\vee}.\]

\end{proof}

\subsection{Compatibility with perverse filtrations}
\label{sec:perverse filtrations}
In this subsection, we show that the pre-Weil operator $\phi$ is compatible with perverse cohomology functors, using the geometric description found by de Cataldo and Migliorini \cite{DM10}. Let $f:V\to W$ be an algebraic morphism between quasi-projective varieties. Let $K$ be a bounded complex with constructible cohomology sheaves on $V$. Denote ${}^{\fp}\tau_{\leq \ell}$ to be the perverse truncation functor.

\begin{notation}\label{perverse filtration}
We denote the perverse Leray filtration on $H^{b}(V,K)$ by
\[ H^{b}_{\leq \ell}(V,K)\colonequals \mathrm{Im} \{ H^{b}(W,{}^{\fp}\tau_{\leq \ell}f_\ast K) \to H^{b}(W,f_{\ast}K)\} \subseteq H^{b}(V,K).\]
We also set
\[ H^{b}_{\ell}(V,K)\colonequals H^{b}_{\leq \ell}(V,K)/H^{b}_{\leq \ell-1}(V,K).\]
\end{notation}

\subsubsection{Perverse filtrations via flag filtrations}\label{sec: perverse filtration via flag filtrations}
Let us recall the main results of \cite{DM10}. Suppose that $\dim W=k$ and let $W\subseteq \bP^{N}$ be a fixed \emph{affine} embedding. 
\begin{definition}
A linear $k$-flag $\mathfrak{F}$ on $\bP^{N}$ is defined to be
\[ \mathfrak{F} \colonequals \{ \bP^N=\Lambda_0\supseteq \Lambda_{-1} \supseteq \cdots \supseteq \Lambda_{-k}\},\]
where $\Lambda_{-\ell}$ is a codimension $\ell$ linear subspace. A linear $k$-flag $\mathfrak{F}$ is  \emph{general}, if it belongs to a suitable Zariski dense open subset of the corresponding flag variety parametrizing all such $k$-flags. 
\end{definition}
Let $\mathfrak{F}^1=\{ \Lambda^1_{\ast}\},\mathfrak{F}^2=\{ \Lambda^2_{\ast}\}$ be two, possibly identical, linear $k$-flags on $\bP^{N}$. They induce two pre-image flags  $V_{\ast},Z_{\ast}$ on $V$:
\begin{align*}
&V=V_0 \supseteq V_{-1} \supseteq \cdots \supseteq V_{-k}\supseteq V_{-k-1}=\varnothing, \quad \textrm{with } V_{\ell} \colonequals f^{-1}(\Lambda^1_\ell\cap W). \\
&V=Z_0 \supseteq Z_{-1} \supseteq \cdots \supseteq Z_{-k}\supseteq Z_{-k-1}=\varnothing, \quad \textrm{with }  Z_{\ell} \colonequals f^{-1}(\Lambda^2_\ell\cap W).
\end{align*}

\begin{notation}
Let $i \colon T\hookrightarrow V$ be a locally closed embedding. We use
\[ (-)_T \colonequals i_{!}i^{\ast}(-)\]
to denote the complex compactly supported on $T$. We use $H^{\ast}_T(-)$ to denote the local cohomology group with support in $T$.
\end{notation}

\begin{thm}[Theorem 4.2.1 \cite{DM10}]\label{thm:geometric characterization of perverse filtration} Given a pair of general flags $\mathfrak{F}^1,\mathfrak{F}^2$, there is an isomorphism
\[ H^{b}_{\leq \ell}(V,K) \cong \mathrm{Im}\left\{ \bigoplus_{j+i=b+\ell} H^{b}_{Z_{-j}}(V,K_{V\setminus V_{i}}) \to H^{b}(V,K) \right\}.\]
\end{thm}
In \cite{DM10}, the authors use the decreasing filtrations. We transform their results to increasing filtrations, which is consistent with the notations in this paper.

\subsubsection{Perverse filtrations are twistor-theoretic}

From now on, we work with Set-up \ref{setup: cohomology of smooth projective} and a map $f:X\to Y$. Set $K=\cV[\dim X], K^{\ast}=\cV^{\ast}[\dim X]$.

\begin{lemma}\label{lemma:perverse is twistor}
For any integers $\ell$ and $b$, the subspace $H^{b}_{\leq \ell}(X,K) \subseteq H^{b}(X,K)$
underlies a pure sub-twistor structure of the natural pure twistor structure on $H^{b}(X,K)$ in Theorem \ref{thm:Hodge-Simpson}. Therefore the quotient space
\[ H^{b}_{\ell}(X,K)\colonequals H^{b}_{\leq \ell}(X,K)/H^{b}_{\leq \ell-1}(X,K)\]
inherits a pure twistor structure $F$ of weight $(b+\dim X)$. Moreover, the pre-Weil operator $\phi$ in Definition \ref{definition: pre-Weil operator} induces an isomorphism $\phi: 
\overline{F|_{z=-1}} \xrightarrow{\sim} H^{b}_{\ell}(X,K^{\ast})$. In other words, the pre-Weil operator $\phi$ is compatible with perverse filtrations.
\end{lemma}
\begin{proof}
We use Theorem \ref{thm:geometric characterization of perverse filtration} to study $H^b_{\leq \ell}(X,K)$. We use the notation from \S \ref{sec: perverse filtration via flag filtrations} with $V=X$. For each $j$, the local cohomology group
$H^{b}_{Z_{-j}}(X,K)$ fits into the long exact sequence associated to the closed and open embeddings $Z_{-j} \xrightarrow{A} X \xleftarrow{B} U\colonequals X\setminus Z_{-j}$:
\begin{equation}\label{eqn:exact sequence of local cohomology}
H^{b}_{Z_{-j}}(X,K)=H^b(X,A_{!}A^{!}K) \to H^{b}(X,K) \to H^{b}(U,B^{\ast}K). 
\end{equation}

\textbf{Claim}: $H^{b}_{Z_{-j}}(X,K)$ underlies a mixed twistor structure $E_1$ with an isomorphism $\phi_{\mathrm{loc}}$ so that the exact sequence \eqref{eqn:exact sequence of local cohomology} underlies an exact sequence of mixed twistor structures with the following commutative diagram:
\[ \begin{CD}
\overline{E_1|_{z=-1}} @>>> \overline{E_2|_{z=-1}} @>>>\overline{E_3|_{z=-1}} \\
@VV\phi_{\mathrm{loc}}V @VV\phi V @VV\phi_{U}V \\
H^{b}_{Z_{-j}}(X,K^{\ast}) @>>> H^{b}(X,K^{\ast}) @>>> H^{b}(U,B^{\ast}K^{\ast})
\end{CD}\]
where $E_2$ and $E_3$ are natural twistor structures on $H^{b}(X,K)$ and $H^{b}(X,B_{\ast}B^{\ast}K)$, $\phi$ and $\phi_{X\setminus Z_{-j}}$ are the corresponding isomorphisms. 

\begin{proof}[Proof of claim]
Since $U=X\setminus Z_{-j}$ is smooth, the square involving $E_2$ and $E_3$ is commutative by Lemma \ref{lemma:Identification map is functorial}. Recall that $H$ is the harmonic bundle associated to $\cV$. Define a complex
\begin{equation*}
 \cA^{k}_X(H)_{Z_{-j}}\colonequals \mathrm{Cone} \{\cA^{\bullet}_X(H) \to \cA^{\bullet}_U(H|_U) \}[-1].\end{equation*}
We have 
\[  H^b_{Z_{-j}}(X,K) \cong H^b(X,\cA^{\bullet}_X(H)_{Z_{-j}}).\]
Adapting the proof of Lemma \ref{lemma: lowest weight}, one can then put a natural mixed twistor structure on $H^b_{Z_{-j}}(X,K)$. Adapting the proof of Lemma \ref{lemma: fiber at z=-1} and applying the rescaling map $\iota_{-1}$ in Definition \ref{lemma:pure twistor structure via glueing}, we obtain the isomorphism $\phi_{\mathrm{loc}}$ with the desired commutativity. One can also use \cite[Proposition III.2.1]{ElZein} for the construction of mixed twistor structures on local cohomology.
\end{proof}

Observe that $K_{X\setminus X_i}$ fits into the distinguished triangle in $D^b_{c}(X)$ associated to the closed and open embeddings $X_i \xrightarrow{\alpha} X \xleftarrow{\beta} X\setminus X_i$:
\[ K_{X\setminus X_i}=\beta_{!}\beta^{\ast}K= \beta_{!}\beta^{!}K \to K \to \alpha_{\ast}\alpha^{\ast}K \xrightarrow{[1]}. \]
By a similar argument as above, we can show that $H^{b}_{Z_{-j}}(X,K_{X\setminus X_i})$ underlies a mixed twistor structure $E_0$ so that there is an isomorphism $\phi_0:\overline{E_0|_{z=-1}}\xrightarrow{\sim} H^{b}_{Z_{-j}}(X,K^{\ast}_{X\setminus X_i})$. Since $H^b(X,K)$ underlies a pure twistor structure, using Theorem \ref{thm:geometric characterization of perverse filtration} and \cite[Lemma 1.3]{Simpson97} (which says that the category of mixed twistor structures is abelian and any morphism is strict with respect to the weight filtrations), we conclude that the space $H^b_{\leq\ell}(X,K)$ admits a pure sub-twistor structure $E_4$ so that the pre-Weil operator $\phi$ restricts to
\[ \phi: \overline{E_4|_{z=-1}} \xrightarrow{\sim} \mathrm{Im}\left\{ H^{b}_{Z_{-j}}(X,K^{\ast}_{X\setminus X_i}) \to H^{b}(X,K^{\ast}) \right\}.\]
The statement for $H^b_{\ell}(X,K)$ follows immediately.
\end{proof}

We show that the pre-Weil operator $\phi$ on perverse filtrations is compatible with smooth restrictions.
\begin{corollary}\label{cor:restriction of perverse filtration}
With the notation above and let $T$ be a smooth subvariety of $X$. Set $\cV_T=\cV|_T$, $K_T=\cV_T[\dim T]$, and $K^{\ast}_T=\cV^{\ast}_T[\dim T]$. Let $H^b_{\ell}(T,K_T)$ denote the space associated to $f|_T:T\to Y$ as in Notation \ref{perverse filtration}. Then we have
\begin{itemize}

\item $H^b_{\ell}(T,K_T)$ underlies a natural mixed twistor structure $F_T$, with an isomorphism
\[ \phi_T: \overline{F_T|_{z=-1}} \xrightarrow{\sim} H^b_{\ell}(T,K^{\ast}_T).\]

\item The restriction map
\[R: H^b_{\ell}(X,K) \to H^b_{\ell}(T,K_T)\]
underlies a morphism of twistor structures $G:F\to F_T$ and there is a commutative diagram:
\[ \begin{CD}
  \overline{F|_{z=-1}}  @>\overline{G|_{z=-1}}>>  \overline{F_T|_{z=-1}} \\
  @VV\phi V @VV\phi_TV \\
  H^{b}_{\ell}(X,K^{\ast}) @>\widetilde{R}>>H^{b}_{\ell}(T,K^{\ast}_T) 
\end{CD}\]
\end{itemize}
\end{corollary}

\begin{proof}
Apply Theorem \ref{thm: mixed twistor complex}, Lemma \ref{lemma:Identification map is functorial} and Lemma \ref{lemma:perverse is twistor}.
\end{proof}

\subsection{Weight filtrations}\label{sec:weight filtrations}

In this subsection, we discuss weight filtrations in the category of  polarized pure twistor structures following  \cite[\S 4.5]{DM05}. It is important for us to set everything up so that they are compatible with the pre-Weil operator $\phi$. We only give proofs when necessary. The equation \eqref{kernel intersect with its orthogonal complement} in Corollary \ref{cor:nondegeneracy on the kernel} is a new input and will play an important role in the proof of the Decomposition Theorem (via Theorem \ref{thm:Hodge-Riemann}, or more precisely Corollary \ref{non-degeneracy of pairing on Lambda0}). 

\subsubsection{Bilinear pairings}
Let $H$ and $\widetilde{H}$ be two finite dimensional vector spaces and let $S: H \otimes \widetilde{H} \to \C$ be a bilinear pairing. As in the case $H=\widetilde{H}$, one can also talk about notions of orthogonal complement, decompositions  orthogonal with respect to $S$ and non-degenerate pairings etc. For reader's benefit, we record the following statement, which is used to prove Corollary \ref{non-degeneracy of pairing on Lambda0} and is an important piece of the proof of Hodge-Riemann bilinear relation Theorem \ref{thm:Hodge-Riemann}.

\begin{lemma}\label{non-degeneracy of subspaces after quotient}
Let $S: H \otimes \widetilde{H} \to \C$ be a bilinear pairing. Assume there are direct sum decompositions orthogonal with respect to $S$: $H=H_1\oplus H_2, \widetilde{H}=\widetilde{H}_{1}\oplus \widetilde{H}_{2}$. Then the induced direct sum decomposition
\begin{align*}
&H/\widetilde{H}^{\perp}=H_1/(H_1\cap \widetilde{H}^{\perp})\oplus H_2/(H_2\cap \widetilde{H}^{\perp})
\end{align*}
is orthogonal with respect to the descent pairing $\hat{S}: H/\widetilde{H}^{\perp} \otimes \widetilde{H}/H^{\perp} \to \C$. Moreover, the pairing
\begin{align*}
&\hat{S}: H_1/(H_1\cap \widetilde{H}^{\perp}) \otimes  \widetilde{H}_{1}/(\widetilde{H}_{1}\cap H^{\perp}) \to \C
\end{align*}
is non-degenerate. We have the same statement if one switches $H$ with $\widetilde{H}$. We also have
\begin{equation}\label{the crucial equality}
\left(\widetilde{H}_{2}/(\widetilde{H}_{2}\cap H^{\perp})\right)^{\perp}=H_1/(H_1\cap \widetilde{H}^{\perp}),
\end{equation}
where the first orthogonal complement is taken with respect to $\hat{S}$.
\end{lemma}

\subsubsection{One weight filtration}

Let $(H,N)$ be a finite dimensional vector space $H$ equipped with a nilpotent endomorphism $N$. By \cite[Lemma 6.4]{Schmid}, there is a unique increasing filtration $W$ so that we have $NW_i\subseteq W_{i-2}$, hard Lefschetz isomorphism $N^i: \Gr^W_{i}H \xrightarrow{\sim} \Gr^W_{-i}H$, and a Lefschetz decomposition 
\begin{equation}\label{single Lefschetz decomposition}
\Gr^W_i H= \bigoplus_{\ell \in \Z} N^{-i+\ell}P^{i-2\ell}, \quad i\in \Z,
\end{equation}
where $P^{-i}\colonequals \Ker N^{i+1}\subseteq \Gr^W_iH$ for $i\geq 0$ and $P^{-i}\colonequals 0$ for $i<0$.  We use $W^N$ to denote the unique filtration and call it the \emph{weight filtration} of $N$. For ease of notation, we set $\Gr^N_iH \colonequals \Gr^{W^N}_iH$.

\begin{definition}[Infinitesimal automorphisms]\label{definition:infinitesimal automorphisms}
Consider two pairs $(H,N)$ and $(\widetilde{H},\widetilde{N})$ so that $\dim H=\dim \widetilde{H}$. We say that $(N,\widetilde{N})$ are \emph{infinitesimal automorphisms} of $(H,\widetilde{H},S)$ if there is a non-degenerate bilinear pairing 
\begin{equation}\label{eqn: pairing for infinitesimal auto}
S: H \otimes \widetilde{H} \to \C,
\end{equation}
which is either symmetric or skew-symmetric and satifies
\[ S(Na,\widetilde{b}) + S(a,\widetilde{N}\widetilde{b})=0, \quad \forall a\in H, \widetilde{b}\in \widetilde{H}.\]
\end{definition}

\begin{lemma}\label{lemma:self duality of weight filtration}
One has
\[ (W^N_i)^{\perp} = W^{\widetilde{N}}_{-i-1}, \quad \forall i \in \mathbb{Z}.\]
\end{lemma}

\begin{proof}
It is similar to the original case. For details, see \cite[Lemma 2.3.107]{Yang21}.
\end{proof}

\begin{corollary}\label{cor:descend to one filtration}
The pairing \eqref{eqn: pairing for infinitesimal auto} descends to non-degenerate pairings for $\ell\geq 0$:
\[ S_{\ell}: \Gr^N_{\ell}H \otimes \Gr^{\widetilde{N}}_{\ell}\widetilde{H} \to \C,\]
where $S_{\ell}([a],[\widetilde{b}])\colonequals S(a,\widetilde{N}^{\ell} \widetilde{b})$. Moreover, the Lefschetz decompositions \eqref{single Lefschetz decomposition} for $\Gr^N_{\ell}H$ and $\Gr^{\widetilde{N}}_{\ell}\widetilde{H}$ are orthogonal with respect to $S_{\ell}$.
\end{corollary}

\begin{corollary}\label{cor:nondegeneracy on the kernel}
We have
\begin{equation}\label{kernel intersect with its orthogonal complement}
 \Ker N\cap (\Ker \widetilde{N})^{\perp}=\Ker N\cap W^{N}_{-1},
\end{equation}
and the pairing \eqref{eqn: pairing for infinitesimal auto} descends to a non-degenerate pairing
\[ S: \Ker N/(\Ker N\cap W^N_{-1}) \otimes \Ker \widetilde{N}/(\Ker \widetilde{N}\cap W^{\widetilde{N}}_{-1})\to \C.\]
\end{corollary}

\begin{proof}
It suffices to show \eqref{kernel intersect with its orthogonal complement}. On the one hand, the isomorphism $\widetilde{N}: \Gr^{\widetilde{N}}_{1}H \xrightarrow{\sim} \Gr^{\widetilde{N}}_{-1}H$ gives $\Ker \widetilde{N}\subseteq W^{\widetilde{N}}_0$. Combining with Lemma \ref{lemma:self duality of weight filtration}, we see that
\[W^{N}_{-1}=(W^{\widetilde{N}}_0)^{\perp}\subseteq (\Ker \widetilde{N})^{\perp}. \]
On the other hand, because $S$ is non-degenerate, we have $ (\Ker \widetilde{N})^{\perp}=\im N$. Therefore
\[ \Ker N\cap (\Ker \widetilde{N})^{\perp} =\Ker N\cap  \im N\subseteq \Ker N\cap  W^N_{-1}.\]
The last inclusion comes from the convolution formula for $W^N_{-1}$ (see \cite[equation (23)]{DM05}).
\end{proof}

\subsubsection{Two weight filtrations}

Let $H$ and $\widetilde{H}$ be two vector spaces with a non-degenerate bilinear pairing $S: H \otimes \widetilde{H} \to \C$. Let $N,M$ be two commuting nilpotent operators on $H$ and $\widetilde{N},\widetilde{M}$ be two commuting nilpotent operators on $\widetilde{H}$ such that $(N,\widetilde{N})$ and $(M,\widetilde{M})$ are infinitesimal automorphisms of $(H,\widetilde{H},S)$. Assume the shifted weight filtration $W^{M}[j]$ (with $W^{M}[j]_i=W^{M}_{j+i}$) induces the monodromy weight filtration of $M$ on $\mathrm{gr}^N_jH$ for every $j\in \Z$ (see \cite[\S 4.5]{DM05}). Make the same assumption for $\widetilde{M}$ and $\widetilde{N}$. It means that
\begin{align*}
M^\ell: \Gr^{M}_{j+\ell}\Gr^{N}_{j}H\xrightarrow{\sim} \Gr^{M}_{j-\ell}\Gr^{N}_{j}H, \textrm{ whenever }\ell\geq 0.
\end{align*}
For $\ell,j\geq 0$, define 
\[ P^{-j}_{-\ell}\colonequals \Ker M^{\ell+1} \cap \Ker N^{j+1}\subseteq \Gr^{M}_{j+\ell}\Gr^{N}_{j}H.\]
Otherwise, set $P^{-j}_{-\ell}=0$. Then we have the double Lefschetz decomposition:
\begin{align}\label{eqn:double Lefschetz decomposition}
\Gr^{M}_{j+\ell}\Gr^{N}_{j}H \cong \bigoplus_{i, m \in \Z} M^{-\ell+i}N^{-j+m}P^{j-2m}_{\ell-2i}, \quad \ell,j\in \Z.
\end{align}
We also have similar statements for $\widetilde{M},\widetilde{N}$ and define $\widetilde{P}^{ -j}_{-\ell}$ to be the corresponding spaces for $\widetilde{M}$ and $\widetilde{N}$. 

\begin{lemma}\label{lemma:polarization on double Lefschetz decomposition}
The non-degenerate pairing $S_{\ell}$ in Corollary \ref{cor:descend to one filtration} descends to a non-degenerate pairing for $\ell,j\geq0$:
\begin{equation}\label{S ell j for general set up}
S_{\ell j}: \Gr^M_{j+\ell}\Gr^N_{j}H\otimes \Gr^{\widetilde{M}}_{j+\ell}\Gr^{\widetilde{N}}_{j}\widetilde{H} \to \C,
\end{equation}
where $S_{\ell j}([a],[\widetilde{b}])=S(a,\widetilde{M}^{ \ell}\widetilde{N}^{ j}\widetilde{b})$ for any $\ell,j\geq 0$. Moreover, the double Lefschetz decompositions (\ref{eqn:double Lefschetz decomposition}) for $\Gr^M_{j+\ell}\Gr^N_{j}H$ and $\Gr^{\widetilde{M}}_{j+\ell}\Gr^{\widetilde{N}}_{j}\widetilde{H}$ are orthogonal with respect to $S_{\ell j}$.
\end{lemma}

The next corollary will be used in Lemma \ref{lem: intersection of various kernels}, as a part of Theorem \ref{thm:Hodge-Riemann}.
\begin{corollary}\label{cor: nondegeneracy on kernels of double filtrations}
For each $\ell,j\geq 0$, consider the vector spaces
\[ \Ker N^{j+1} \subseteq \Gr^M_{j+\ell}\Gr^N_jH,\quad 
\Ker \widetilde{N}^{j+1} \subseteq \Gr^{\widetilde{M}}_{j+\ell}\Gr^{\widetilde{N}}_j\widetilde{H}.\]
Then the pairing \eqref{S ell j for general set up} restricts to a non-degenerate pairing:
\[ S_{\ell j}: \Ker N^{j+1} \otimes \Ker \widetilde{N}^{j+1} \to \C.\]
\end{corollary}
\begin{proof}
We can use (\ref{eqn:double Lefschetz decomposition}) and the commutability between $M$ and $N$ to obtain decompositions
\[ \Ker N^{j+1} = \bigoplus_{i\in \Z} M^{-\ell+i}P^{-j}_{\ell-2i}, \quad \Ker \widetilde{N}^{j+1} = \bigoplus_{i\in \Z} \widetilde{M}^{-\ell+i}\widetilde{P}^{-j}_{\ell-2i},\]
which are orthogonal with respect to $S_{\ell j}$ by Lemma \ref{lemma:polarization on double Lefschetz decomposition}. Then the non-degeneracy follows.
\end{proof}

\subsubsection{Filtrations on $H^{\ast}(X,\cV)$}\label{sec:twisted Poincare pairing}

Now we apply the previous discussion to the cohomology of local systems. We work with Set-up \ref{setup: cohomology of smooth projective}. Let $f: X\to Y$ be a map between smooth projective varieties, let $A$ be an ample line bundle on $Y$, and set $L\colonequals f^{\ast}A$. Denote $K\colonequals \cV[\dim X]$. Consider the two vector spaces
\[ H=H^{\ast}(X,\cV)\colonequals \bigoplus_{j\in \mathbb{Z}} H^j(X,\cV), \quad \widetilde{H}=H^{\ast}(X,\cV^{\ast})\colonequals \bigoplus_{j\in \mathbb{Z}} H^j(X,\cV^{\ast}).\]
Since $X$ is compact, there is a non-degenerate Poincar\'e pairing
\begin{align}
&S: H \otimes \widetilde{H} \to \C,  \label{Poincare pairing}\\
&S(\sum e_k\otimes \alpha_{k},\sum \lambda_{k}\otimes \beta_{k})\colonequals \sum_{k}  C(k)\int_X \lambda_{2n-k}(e_k)\cdot \alpha_k\wedge \beta_{2n-k}, \nonumber
\end{align}
where $e_k$ and $\lambda_k$ are global sections of the harmonic bundle $H$ and dual harmonic bundle $H^{\ast}$ respectively, and $\alpha_k$, $\beta_k$ are $k$-forms on $X$, $C(k)=(-1)^{k(k-1)/2}$.

Let $\eta$ and $L$ denote the commuting nilpotent operators on $H$, induced by the cup product maps. To keep track the filtrations on $\widetilde{H}$, we use $\widetilde{\eta}$ and $\widetilde{L}$ to denote the corresponding operators on $\widetilde{H}$.  It is direct to check that $(\eta,\widetilde{\eta})$ and $(L,\widetilde{L})$ are infinitesimal automorphisms for $(H,\widetilde{H},S)$:
\begin{equation}\label{eqn: Poincare pairing for double Lefschetz decomposition}
S(\eta a, b)+S(a,\widetilde{\eta}\widetilde{b})=0, \quad S(L a, b)+S(a,\widetilde{L}\widetilde{b})=0.
\end{equation}

Theorem \ref{thm:Hodge-Simpson} implies that the weight filtrations of $W^{\eta}$ on $H$ is
\begin{equation}\label{eqn:weight filtration via eta}
W^{\eta}_i = \bigoplus_{\ell \geq n-i} H^\ell(X,\cV)=\bigoplus_{b\geq -i}H^{b}(X,K). 
\end{equation}
Consider the total filtration on $H$ (see Notation \ref{perverse filtration}):
\begin{equation}\label{eqn:total weight filtration}
W^{\tot}_j\colonequals \bigoplus_{b\in \Z}H^{b}_{\leq b+j}(X,K).
\end{equation}
Let $W^{\widetilde{\eta}}$ and $W^{\widetilde{\tot}}$ denote the corresponding filtrations on $\widetilde{H}$. Then we have for all $\ell,j\in \Z$:
\[ H^{-\ell-j}_{-\ell}(X,K)=\Gr^{\eta}_{j+\ell}\Gr^{\tot}_{j}H, \quad
H^{-\ell-j}_{-\ell}(X,K^{\ast})=\Gr^{\widetilde{\eta}}_{j+\ell}\Gr^{\widetilde{\tot}}_{j}\widetilde{H}. \]
As in \cite{DM05}, we will show $W^{L}=W^{\tot}$ in Corollary \ref{cor:identification of weight filtrations}. 

\begin{remark}\label{remark: signs for twisted Poincare pairing}
By Theorem \ref{cor:polarized by twisted Poincare pairing}, we see that $i^{-(n-\ell-j)}\cdot S_{\ell j}(\bullet,\phi(\bullet))$
polarizes the pure twistor structure on $P^{-j}_{-\ell}=\Ker \eta^{\ell+1}\cap \Ker L^{j+1}$. Therefore, we can get the appropriate signs for other primitive pieces in the double Lefschetz decomposition, as in \cite[Remark 4.5.2]{DM05}. 
\end{remark}

\section{An application to the Decomposition Theorem}\label{sec: proof of main results}

In this section, building on the results developed in the previous sections, we give a new proof of Sabbah's Decomposition Theorem \ref{thm:main}. Here is the main set-up of this section.
\begin{setup}\label{main setup} -
\begin{itemize}
    \item Let $f: X\to Y$ be a morphism between projective varieties, where $X$ is smooth of dimension $n$ and $\dim f(X)=m$.  Let $\eta$ be an ample line bundle on $X$, let $A$ be an ample line bundle on $Y$, and set $L\colonequals f^{\ast}A$.
    \item Let $\cV$ be a semisimple local system on $X$ and denote $\cV^{\ast}$ to be the dual local system. Set $K\colonequals \cV[\dim X], K^{\ast}\colonequals \cV^{\ast}[\dim X]$ to be the associated perverse sheaves on $X$.
    \end{itemize} 
\end{setup}

\subsection{Outline of the proof}\label{sec: outline of the proof}
Since the general strategy is quite close to \cite[\S 2.6]{DM05}, we may skip some details and leave them to the reader. But we will elaborate on extra difficulties and give more details. We prove Theorem \ref{thm:main}, Theorem \ref{thm:Hard Lefschetz Perverse Complexes} and Theorem \ref{thm:Hodge-Riemann}, by double induction on the defect of semismallness $r=r(f)$ and $m=\dim f(X)$. The basic reason for double induction is that in the proof of Theorem \ref{thm:Hard Lefschetz Perverse Complexes} and Theorem \ref{thm:Hodge-Riemann}, there are two different ways of cutting hyperplanes. Cutting on $X$ gives $r'<r$ and cutting on $Y$ gives $r'\leq r$ and $m'<m$. 

Step 1 is similar to the original strategy and Step 2-4 consist of additional difficulties, this is where the pre-Weil operator $\phi$ comes in to overcome the difficulty.
\begin{enumerate}

\item [\textbf{Step 1}] By Deligne's Lefschetz splitting criterion and the construction of universal hyperplanes, it suffices to prove Theorem \ref{thm:main}(i) and Theorem \ref{thm:main}(iii) and the key step is to prove Theorem \ref{thm:main}(iii) for $\ell=0$. To achieve this, we first prove Theorem \ref{thm:Hard Lefschetz Perverse Complexes} and Theorem \ref{thm:Hodge-Riemann} by induction and we elaborate this in Step 2 and 3.

\item [\textbf{Step 2}] Theorem \ref{thm:main}(i) and the inductive Theorem \ref{thm:Hodge-Riemann} imply Theorem \ref{thm:Hard Lefschetz Perverse Complexes}. This needs the new input about the compatibility of the pre-Weil operator $\phi$ and perverse filtration (Lemma \ref{lemma:perverse is twistor}, Lemma \ref{lemma:Identification map is functorial} and Corollary \ref{cor:restriction of perverse filtration}).

\item [\textbf{Step 3}] To prove Theorem \ref{thm:Hodge-Riemann}, we need the setup of weight filtrations on two companion vector spaces from \S \ref{sec:weight filtrations}. This is the most difficult part and our argument is not the same as the original proof. Then by induction, one can reduce it to the case of the constant map (Theorem \ref{thm: generalized Weil operator}).

\item [\textbf{Step 4}] We can complete the proof the Semisimplicity Theorem \ref{thm:main}(iii) using the splitting criterion of de Cataldo-Migliorini. To apply the criterion, we need Lemma \ref{cor: refined intersection is Poincare pairing} to relate the adjunction morphism with the twisted Poincar\'e pairing and then apply the positivity coming from the polarization induced by the twisted Poincar\'e pairing and the pre-Weil operator $\phi$. The semisimplicity of local systems over strata are deduced from Simpson's theorem in the case of smooth projective maps (see Corollary \ref{thm:semisimplicity smooth projective map}).
\end{enumerate}

\subsection{The cup product with a line bundle}\label{cup product with a line bundle}
We need a duality result on cup products. Let $X$ be a projective variety and let $\eta$ be a line bundle on $X$. Let $K$ be a constructible complex of $\C$-vector spaces on $X$. Let $D_X$ denote the Verdier dual functor and set $K^{\ast}\colonequals D_X(K)$. The first Chern class of $\eta$ corresponds to an element in $H^2(X,\C)\cong \mathrm{Hom}_{D^b_c(X)}(\C_X,\C_X[2])$. 
\begin{definition}
The cup product map $\eta: K \to K[2]$ is defined by \[K\cong K\overset{\mathbb{L}}{\otimes}\C_X \xrightarrow{\mathrm{Id}\otimes \eta} K\overset{\mathbb{L}}{\otimes} \C_X[2] \cong K[2].\] 
\end{definition}

\begin{lemma}\label{lemma: dual of cup product}
The dual $D_X(\eta)$ of the morphism $\eta:K \to K[2]$ is isomorphic to the morphism $\eta: (K^{\ast}\to K^{\ast}[2])[-2]$.
\end{lemma}

\begin{proof}
It follows from \cite[Remark 4.4.1]{DM05} and the compatibility of functors involved in the description of $K\to K[2]$ with the Verdier dual functor. For details, see \cite[Lemma 2.5.2]{Yang21}.
\end{proof}

\begin{corollary}\label{corollary:duality of cup product map}
Suppose we are in the Set-up \ref{main setup}, then the dual $D_Y(\eta^{\ell})$ of the morphism $\eta^{\ell}:{}^{\fp}\cH^{-\ell}f_{\ast}K \to {}^{\fp}\cH^{\ell}f_{\ast}K $ is isomorphic to the morphism $ \eta^{\ell}:{}^{\fp}\cH^{-\ell}f_{\ast}K^{\ast} \to  {}^{\fp}\cH^{\ell}f_{\ast}K^{\ast}$.
\end{corollary}

\subsection{Weak-Lefschetz-type results}\label{sec: weak lefschetz type}
To run the inductive proof of Theorem \ref{thm:main} via cutting with hyperplanes, let us recall results from \cite[\S 4.7, \S 5.2, \S 5.3]{DM05} enriched by pure twistor structures and the pre-Weil operator $\phi$.

\begin{prop}\label{prop:weak Lefschetz cohomology}
There exists $m_0$ so that for any $m\geq m_0$, the following statements hold: if $i: X^{1} \to X$ is a general hyperplane section of $\abs{\eta^{\otimes m}}$ and $f^{1}:X^1 \to Y$ is the restricted map, set $K^{1}\colonequals i^{\ast}K[-1]$, then
\begin{itemize}

\item $r(f^1)\leq \max \{r(f)-1,0 \}$.
\item The restriction map ${}^{\fp}\cH^{-\ell}(f_{\ast}K) \to {}^{\fp}\cH^{-\ell+1}(f^{1}_{\ast}K^{
1})$ is iso for $\ell\geq 2$ and mono for $\ell=1$.

\item Assume Theorem \ref{thm:main}(ii) holds for $f$, then there is an injective morphism
\[ i^{\ast}: P^{-j}_{-\ell}(X,K) \to P^{-j}_{-\ell+1}(X^1,K^1), \quad  \ell\geq 1, j\geq 0,
\]
which underlies a morphism of pure twistor structures and are compatible with the pre-Weil operator $\phi$.
\end{itemize}
\end{prop}

\begin{proof}
For the second statement, we need to apply \cite[Lemma 3.5.4(b)]{DM05} to get $i^{\ast}K[-1]=i^{!}K[1]$. The statement on twistor structures follows from Corollary \ref{cor:restriction of perverse filtration} (functoriality with restriction) and Lemma \ref{lemma:Identification map is functorial} (compatibility with cup products).
\end{proof}

Since $f:X\to Y$ is an algebraic map between algebraic varieties, the algebraic version of Thom isotopy lemmas  (for example see \cite[Theorem 3.2.3]{DM05}) imply that there exist finite algebraic Whitney stratifications $\sX$ of $X$ and $\sY$ of $Y$ such that 1) given any connected component $S$ of a $\sY$ stratum $S_{\ell}$ on $Y$, then  $f^{-1}(S)$ is an union of connected components of strata of $\sX$, each of which is mapping submersively to $S$, 2) $\forall y\in S$, there exists an euclidean open neighborhood $U$ of $y$ in $S$ and a stratum-preserving homeomorphism $h:U\times f^{-1}(y)\cong f^{-1}(U)$ such that $f\circ h$ is the projection to $U$. For the next statement, let us fix a choice of such stratifications on $X$ and $Y$.

Applying Bertini Theorem to linear systems $|A|$ and $f^{\ast}|A|$ simultaneously, we can choose a general section $Y_1\in |A|$ such that 1) $Y_1$ is transverse to all positive-dimensional strata of $Y$ and avoids the $0$-dimensional strata, 2) $X_1=f^{-1}(Y_1)$ is smooth.
  
\begin{prop}\label{prop:weak Lefschetz hypercohomology}
With the notation above. Let $f_1:X_1 \to Y_1$ denote the restriction of $f$, let $i:X_1 \to X$ denote the inclusion map and set $K_1\colonequals i^{\ast}K[-1]$. Then \begin{itemize}
\item $r(f_1)\leq r(f)$,
\item the restriction map $i^{\ast}: H^{-j}(Y,{}^{\fp}\cH^0(f_{\ast}K)) \to H^{-j+1}(Y_1,{}^{\fp}\cH^0(f_{1\ast}K_1))$ is an isomorphism for $j\geq 2$ and an injection for $j=1$,

\item the Gysin pushforward $i_{\ast}: H^{j-1}(Y_1,{}^{\fp}\cH^0(f_{1\ast}K_1)) \to H^{-j}(Y,{}^{\fp}\cH^0(f_{\ast}K))$ is an isomorphism for $j\geq 2$ and a surjection for $j=1$,

\item assuming Theorem \ref{thm:main}(ii) holds for $f$, there is an injective morphism 
\[ i^{\ast}(P^{-j}_0(X,K)) \subseteq P^{-j+1}_0(X_1,K_1), \quad j\geq 1,\]
which is an equality for $j\geq 2$.
\end{itemize}
Moreover, all morphisms underlie morphisms of pure twistor structures and are compatible with the pre-Weil operator $\phi$.
\end{prop}

\begin{proof}
    The proposition can be proved using \cite[Lemma 4.7.6 and Proposition 4.7.7]{DM05} as in the proof of \cite[Lemma 5.3.1]{DM05}. The compatibility with $\phi$ follows from Corollary \ref{cor:restriction of perverse filtration}  and Lemma \ref{lemma:Identification map is functorial}.
\end{proof}

\subsection{Polarization and splitting criterions}\label{splitting criterion}
In this section, we relate adjunction morphisms with Poincar\'e-type pairings. Let $T$ be a closed subset of $Y$. We use $T \xrightarrow{\alpha} Y\xleftarrow{\beta} Y\setminus T$ to denote the closed and open embedding. For any perverse sheaf $P$ on $Y$,
we have two distinguished triangles
\[ \alpha_{!}\alpha^{!}P \to P \to \beta_{\ast}\beta^{\ast}P\xrightarrow{[1]}, \quad 
\beta_{!}\beta^{!}P \to P \to \alpha_{\ast}\alpha^{\ast}P\xrightarrow{[1]}.\]
They induce a map of complexes $\alpha_{!}\alpha^{!}P \to P \to \alpha_{\ast}\alpha^{\ast}P$, whose cohomology gives a map
\begin{equation}\label{eqn:the splitting map}
H^k(Y,\alpha_{!}\alpha^{!}P) \to H^k(Y,\alpha_{\ast}\alpha^{\ast}P), \quad \forall k\in \mathbb{Z}.
\end{equation}
This map induces various maps in the splitting criterions of de Cataldo-Migliorini \cite[Lemma 4.1.3]{DM05} and of MacPherson-Vilonen \cite{MV} (see \cite[Remark 5.7.5]{DM09}). We interpret the map (\ref{eqn:the splitting map}) and other related maps via various Poincar\'e-type pairings. This should be known to experts, but we would like to include the discussion here because we cannot find sufficient references in the literature.

First, recall the Poincar\'e pairing from \eqref{Poincare pairing}:
\begin{align}\label{eqn: Poincare pairing for k=n}
S:H^k(X,\cV) &\otimes H^{2n-k}(X,\cV^{\ast}) \to \C,\\ \nonumber
[A]&\otimes [B] \mapsto C(k)\int_X A\wedge B.
\end{align}
where $A,B$ are forms with coefficients in $\cV$ and $\cV^{\ast}$. For $k=n$, the pairing induces
\[ S:H^0(X,K) \to H^0(X,K^{\ast})^{\vee}, \quad A\mapsto (B\mapsto S(A,B)),\]
where $K=\cV[n]$. Let $D_Z$ denote the Verdier dual functor on a space $Z$ and let $p_X: X\to \mathrm{pt}$ denote the constant map. Since $X$ is proper, there is a canonical equivalence of functors $(p_X)_{\ast} \cong D_{\mathrm{pt}} \circ (p_X)_{\ast} \circ D_X$. Applying to $K$, we get the following map
\begin{equation}\label{eqn: Verdier duality map} 
V: H^0(X,K) \to H^0(X,K^{\ast})^{\vee}.
\end{equation}
\begin{lemma}\label{lemma: identification between Poincare and Verdier in the case of local system}
The map $V$ coincides with the map $S$, up to a sign $C(n)=(-1)^{n(n-1)/2}$.
\end{lemma}
\begin{proof}
When $\cV$ is the constant local system $\C_X$, this is mentioned in \cite[\S 3.1]{KS}. To simplify the notation, we denote by $p$ the constant map $p_X:X\to \mathrm{pt}$. Let us denote by $\omega^{\bullet}_X=p^{!}\C_{\mathrm{pt}}$ the dualizing complex on $X$ so that $D_X(-)=R\mathcal{H}om(-,\omega^{\bullet}_X)$.

First let us recall that $p_{\ast} \circ D_X\xrightarrow{\sim} D_{\mathrm{pt}}\circ p_{\ast}$ is the existence of an equivalence of functors
\[ p_{\ast}R\mathcal{H}om(-,p^{!}\C_{\mathrm{pt}})\xrightarrow{\sim}R\mathcal{H}om(p_{\ast}(-),\C_{\mathrm{pt}}),\]
which induces the isomorphism $p_{\ast}\cong D_{\mathrm{pt}} \circ p_{\ast} \circ D_X$ and thus the map $V$. By the adjunction map $p_{\ast}p^{!}\to \mathrm{Id}$, the equivalence above is defined by the composition of morphisms
\begin{equation}\label{eqn: map inducing the Verdier duality}
p_{\ast}R\mathcal{H}om(-,\omega^{\bullet}_X)\xrightarrow{p_{\ast}} R\mathcal{H}om(p_{\ast}(-),p_{\ast}\omega^{\bullet}_X)\to R\mathcal{H}om(p_{\ast}(-),\C_{\mathrm{pt}}),
\end{equation}
where first arrow is induced by $p_{\ast}$. 

Now, let us take a closer look at the first arrow $p_{\ast}$ in \eqref{eqn: map inducing the Verdier duality}. Let $K$ be any constructible complex of $\C$-vector spaces on $X$. The tensor-hom adjunction gives a natural isomorphism
\[ \mathrm{Hom}\left(p_{\ast}R\mathcal{H}om(K,\omega^{\bullet}_X),R\mathcal{H}om(p_{\ast}K,p_{\ast}\omega^{\bullet}_X)\right)\xrightarrow{\sim} \mathrm{Hom}(p_{\ast}R\mathcal{H}om(K,\omega^{\bullet}_X)\overset{L}{\otimes} p_{\ast}K,p_{\ast}\omega^{\bullet}_X). \]
By adjunction, the first arrow $p_{\ast}$ in \eqref{eqn: map inducing the Verdier duality} is transformed to the composition of morphisms
\begin{align}\label{eqn: apply tensor hom adjunction to p*}
    p_{\ast}R\mathcal{H}om(K,\omega^{\bullet}_X)\overset{L}{\otimes} p_{\ast}K \xrightarrow{p_{\ast}\otimes \mathrm{Id}} R\mathcal{H}om(p_{\ast}K,p_{\ast}\omega^{\bullet}_X)\overset{L}{\otimes} p_{\ast}K \to p_{\ast}\omega^{\bullet}_X,
\end{align}
and the second arrow is induced by the natural morphism $R\mathcal{H}om(A,B)\overset{L}{\otimes} A \to B$ for complexes $A,B$. Using the projection formula, one can show that the morphism \eqref{eqn: apply tensor hom adjunction to p*} can also be factorized as 
\begin{equation*}
p_{\ast}R\mathcal{H}om(K,\omega^{\bullet}_X)\overset{L}{\otimes} p_{\ast}K \rightarrow p_{\ast}(R\mathcal{H}om(K,\omega^{\bullet}_X)\overset{L}{\otimes} K) \to p_{\ast}\omega^{\bullet}_X,
\end{equation*}
where the first map is induced by the natural transformation $p_{\ast}(-)\overset{L}{\otimes} p_{\ast}(-)\to p_{\ast}(-\overset{L}{\otimes} -)$. 

Note that the second arrow in \eqref{eqn: map inducing the Verdier duality} is induced by $p_{\ast}\omega_X^{\bullet}\to \C_\mathrm{pt}$. In our case, it takes a special form and relates to integration on $X$. Since $X$ is an orientable manifold, there exists an isomorphism
\[ \C_X[2n] \xrightarrow{\sim} \omega^{\bullet}_X=p^{!}\C_{\mathrm{pt}}.\]
By the adjunction map $p_{\ast}p^{!}\to \mathrm{Id}$, this induces a map
\[ p_{\ast}\C_X[2n] \to p_{\ast}\omega^{\bullet}_X\to \C_{\mathrm{pt}}.\]
Since the first term above sits in degrees $[-2n,0]$, there is a factorization
\begin{equation}\label{eqn: factorization of the map from p* to C}
p_{\ast}\C_X[2n] \to R^{2n}p_{\ast}\C_X \xrightarrow{\mathrm{Tr}}\C_{\mathrm{pt}},
\end{equation}
where the second map is induced by the isomorphism $H^{2n}(X,\C)\cong \C$ coming from the orientation of $X$. In particular, we can understand $p_{\ast}\omega^{\bullet}_X\to \C_{\mathrm{pt}}$ using \eqref{eqn: factorization of the map from p* to C}.

Summarizing the discussion above, if we plug $K=\cV[n]$ into \eqref{eqn: map inducing the Verdier duality}, then it induces the following map
\begin{equation}\label{eqn: Verdier dual now is Poincare pairing}
p_{\ast}R\mathcal{H}om(K,\omega^{\bullet}_X)\overset{L}{\otimes} p_{\ast}(K) \rightarrow p_{\ast}(R\mathcal{H}om(K,\omega^{\bullet}_X)\overset{L}{\otimes} K) \to p_{\ast}\omega^{\bullet}_X \to p_{\ast}\C_{X}[2n] \to \C_{\mathrm{pt}}.
\end{equation}
The last two maps come from \eqref{eqn: factorization of the map from p* to C} and $\omega^{\bullet}_X\xrightarrow{\sim} \C_X[2n]$. Moreover, by taking the $0$-th cohomology, the fact that the map $V$ being isomorphism is equivalent to the statement that the induced pairing\begin{equation}\label{eqn: pairing induced from Verdier duality}
    H^0(X,K^{\ast}) \otimes H^0(X,K)\rightarrow H^0(X,K^{\ast}\overset{L}{\otimes} K) \to H^0(X,\omega_X^{\bullet})\xrightarrow{\sim} H^{2n}(X,\C) \xrightarrow{\int} \C
\end{equation}
is a perfect pairing, so that the map $V$ is induced by \eqref{eqn: pairing induced from Verdier duality}. Here because $X$ is an orientable manifold, we can choose the last isomorphism to be induced by the integration $\omega\mapsto \int_X\omega$, for any top degree form $\omega$. 

Note that we have $K=\cV[n]$, $K^{\ast}\cong \cV^{\ast}[n]$, and thus $K\overset{L}{\otimes} K^{\ast}\cong \cV\otimes \cV^{\ast}[2n]$. Then the map \eqref{eqn: pairing induced from Verdier duality} coincides with the pairing
\begin{equation}\label{eqn: Verdier duality pairing integration}
H^n(X,\cV)\otimes H^n(X,\cV^{\ast})\xrightarrow{\cup} H^{2n}(X,\cV\otimes \cV^{\ast})\to H^{2n}(X,\C)\xrightarrow{\int}\C,
\end{equation}
where the first map is the cup product map, the second map is induced by the natural map $\cV\otimes R\mathcal{H}om(\cV,\C_X)\to \C_X$. Since we can represent each cohomology class in $H^n(X,\cV)$ by forms with coefficients in the vector bundle $H=\cV\otimes \cC^{\infty}_X$ (same for $\cV^{\ast}$), it is then clear that the pairing $\eqref{eqn: Verdier duality pairing integration}$ is equal to the Poincar\'e pairing \eqref{eqn: Poincare pairing for k=n}, up to a sign $C(n)=(-1)^{n(n-1)/2}$. 

Therefore we conclude that the map $V$ coincides with the map $S$, up to a sign $(-1)^{n(n-1)/2}$.
\end{proof}

Now, we extend Lemma \ref{lemma: identification between Poincare and Verdier in the case of local system} to perverse filtrations on $H^0(X,K)$. Let $\alpha: \{ y\} \hookrightarrow  Y$ be a point in $Y$. The adjunction map $\alpha_{!}\alpha^{!}f_{\ast}K \to f_{\ast}K$ induces a cycle map
\[ \mathrm{cl}:H^0(Y,\alpha_{!}\alpha^{!}f_{\ast}K) \to H^0(Y,f_{\ast}K)=H^0(X,K).\]
We denote the corresponding cycle map for $K^{\ast}$ by $ \widetilde{\mathrm{cl}}$ (instead of $\mathrm{cl}^{\ast}$, which may get confused with the pullback of $\mathrm{cl}$). They induce the following pairing
\begin{align}\label{eqn: restricted twisted Poincare pairing for cycle maps}
S_{y}: H^0(Y,\alpha_{!}\alpha^{!}f_{\ast}K)\otimes H^0(Y,\alpha_{!}\alpha^{!}f_{\ast}K^{\ast}) \to \C,\quad A\otimes B \mapsto S(\mathrm{cl}(A), \widetilde{\mathrm{cl}}(B)). 
\end{align}
Let $S_{y}$ also denote the induced map between vector spaces. Using the canonical isomorphisms 
\begin{equation}\label{eqn: canonical isomorphisms between verdier dual functor and others}
D_Y \circ \alpha_{!}\cong \alpha_{\ast}\circ D_Y, \quad D_Y \circ \alpha^{!} \cong \alpha^{\ast}\circ D_Y, \quad D_Y\circ f_{\ast}\cong f_{\ast}\circ D_X,
\end{equation}
we have the following isomorphism
\begin{equation}\label{eqn: maps induced by Verdier duality}
\mathrm{V.D.}: H^0(Y,\alpha_{!}\alpha^{!}f_{\ast}K^{\ast})^{\vee} \xrightarrow{V^{-1}} H^0(Y, D_Y(\alpha_{!}\alpha^{!}f_{\ast}K^{\ast})) \xrightarrow{\sim} H^0(Y,\alpha_{\ast}\alpha^{\ast}f_{\ast}K).
\end{equation}

\begin{lemma}\label{lemma: splitting map for direct image of K}
The map 
\[ H^0(Y,\alpha_{!}\alpha^{!}f_{\ast}K) \xrightarrow{S_{y}} H^0(Y,\alpha_{!}\alpha^{!}f_{\ast}K^{\ast})^{\vee} \xrightarrow{\mathrm{V.D.}}H^0(Y,\alpha_{\ast}\alpha^{\ast}f_{\ast}K) \]
coincides with the map from (\ref{eqn:the splitting map})
\[ H^0(Y,\alpha_{!}\alpha^{!}f_{\ast}K) \to H^0(Y,\alpha_{\ast}\alpha^{\ast}f_{\ast}K),\]
up to a sign.
\end{lemma}

\begin{proof}
The map $ \alpha_{!}\alpha^{!}f_{\ast}K \to \alpha_{\ast}\alpha^{\ast}f_{\ast}K$ can be decomposed as
\begin{align*}
\alpha_{!}\alpha^{!}f_{\ast}K \to f_{\ast}K \xrightarrow{\sim} D_Y\circ D_Y(f_{\ast}K) \xrightarrow{\sim}
D_Y(f_{\ast}K^{\ast}) \to D_Y(\alpha_{!}\alpha^{!}f_{\ast}K^{\ast})\xrightarrow{\sim} \alpha_{\ast}\alpha^{\ast}f_{\ast}K.
\end{align*}
The corresponding map on hypercohomology can be decomposed as
\begin{align*}
&H^0(\alpha_{!}\alpha^{!}f_{\ast}K) \xrightarrow{\mathrm{cl}} H^0(Y,f_{\ast}K)=H^0(X,K) \xrightarrow{V} H^0(X,K^{\ast})^{\vee}\\
=&H^0(Y,f_{\ast}K^{\ast})^{\vee}\xrightarrow{(\widetilde{\mathrm{cl}})^{\vee}}H^0(Y,\alpha_{!}\alpha^{!}f_{\ast}K^{\ast})^{\vee} \xrightarrow{\mathrm{V.D.}} H^0(Y,\alpha_{\ast}\alpha^{\ast}f_{\ast}K).
\end{align*}
Using Lemma \ref{lemma: identification between Poincare and Verdier in the case of local system}, for any $A\in H^0(Y,\alpha_{!}\alpha^{!}f_{\ast}K)$ and $B\in H^0(Y,\alpha_{!}\alpha^{!}f_{\ast}K^{\ast})$, we have
\[ \left((\widetilde{\mathrm{cl}})^{\vee}\circ V \circ \mathrm{cl}\right)(A): B \mapsto C(n)\cdot S(\mathrm{cl}(A), \widetilde{\mathrm{cl}}(B)),\]
which is equal to the map $S_{y}$ induced by (\ref{eqn: restricted twisted Poincare pairing for cycle maps}), up to a sign.
\end{proof}

Lastly, let us explain the adaption of Lemma \ref{lemma: splitting map for direct image of K} to ${}^{\fp}\cH^0f_{\ast}K$. Assume Theorem \ref{thm:main}(ii)  holds for $f_{\ast}K$, and we have the induced cycle map  
\[ \mathrm{cl}: H^0(Y,\alpha_{!}\alpha^{!}{}^{\mathfrak{p}}\cH^{0}f_{\ast}K) \to H^0(Y,{}^{\mathfrak{p}}\cH^{0}f_{\ast}K).\]
Similarly we use $\widetilde{\mathrm{cl}}$ to denote the cycle map of $K^{\ast}$. The pairing \eqref{eqn: restricted twisted Poincare pairing for cycle maps} induces a pairing 
\begin{align*}
S^{\eta L}_{00}:H^0(Y,\alpha_{!}\alpha^{!}{}^{\mathfrak{p}}\cH^{0}f_{\ast}K) \otimes H^0(Y,\alpha_{!}\alpha^{!}{}^{\mathfrak{p}}\cH^{0}f_{\ast}K^{\ast}) \to \C,\quad A\otimes B \mapsto S(\mathrm{cl}(A) ,\widetilde{\mathrm{cl}}(B)).
\end{align*}
Using the canonical isomorphisms from \eqref{eqn: canonical isomorphisms between verdier dual functor and others} and
\[ D_Y\circ {}^{\fp}\tau_{\leq 0} \cong {}^{\fp}\tau_{\geq 0} \circ D_Y, \quad D_Y\circ {}^{\fp}\tau_{\geq 0} \cong {}^{\fp}\tau_{\leq 0} \circ D_Y, \quad D_Y\circ {}^{\fp}\cH^k \cong {}^{\fp}\cH^{-k}\circ D_Y, \]
we have the following map as in \eqref{eqn: maps induced by Verdier duality}:
\begin{equation}\label{eqn: Verdier duality for 0th perverse cohomology}
     \mathrm{V.D.}: H^0(Y,\alpha_{!}\alpha^{!}{}^{\mathfrak{p}}\cH^{0}f_{\ast}K^{\ast})^{\vee} \to  H^0(Y,\alpha_{\ast}\alpha^{\ast}{}^{\mathfrak{p}}\cH^{0}f_{\ast}K).
\end{equation}
\begin{lemma}\label{cor: refined intersection is Poincare pairing}
Assume Theorem \ref{thm:main}(ii) holds for $f_{\ast}K$, then the map 
\[ H^0(Y,\alpha_{!}\alpha^{!}{}^{\mathfrak{p}}\cH^{0}f_{\ast}K) \xrightarrow{S^{\eta L}_{00}}H^0(Y,\alpha_{!}\alpha^{!}{}^{\mathfrak{p}}\cH^{0}f_{\ast}K^{\ast})^{\vee} \xrightarrow{\mathrm{V.D.}} H^0(Y,\alpha_{\ast}\alpha^{\ast}{}^{\mathfrak{p}}\cH^{0}f_{\ast}K)   \]
coincides with the map from (\ref{eqn:the splitting map})
\[ H^0(Y,\alpha_{!}\alpha^{!}{}^{\mathfrak{p}}\cH^{0}f_{\ast}K) \to H^0(Y,\alpha_{\ast}\alpha^{\ast}{}^{\mathfrak{p}}\cH^{0}f_{\ast}K),\]
where $k=0$ and $P={}^{\mathfrak{p}}\cH^{0}f_{\ast}K$, up to a sign.
\end{lemma}

\begin{proof}
Note that the map (\ref{eqn:the splitting map}) $ \alpha_{!}\alpha^{!}{}^{\fp}\tau_{\leq 0}f_{\ast}K \to \alpha_{\ast}\alpha^{\ast}{}^{\fp}\tau_{\leq 0}f_{\ast}K$  can be decomposed as
\begin{align*}
\alpha_{!}\alpha^{!}{}^{\fp}\tau_{\leq 0}f_{\ast}K &\to {}^{\fp}\tau_{\leq 0}f_{\ast}K \xrightarrow{\sim} D_Y\circ D_Y({}^{\fp}\tau_{\leq 0}f_{\ast}K)\xrightarrow{\sim}\\
D_Y({}^{\fp}\tau_{\geq 0}f_{\ast}K^{\ast}) &\to D_Y(\alpha_{!}\alpha^{!}{}^{\fp}\tau_{\geq 0}f_{\ast}K^{\ast}) \xrightarrow{\sim} \alpha_{\ast}\alpha^{\ast}{}^{\fp}\tau_{\leq 0}f_{\ast}K.
\end{align*}
In addition, the map from (\ref{eqn:the splitting map}) for $f_{\ast}(K)$ and ${}^{\mathfrak{p}}\cH^{0}f_{\ast}K$ are compatible in the sense that one can represent an element in $H^0(Y,\alpha_{!}\alpha^{!}{}^{\mathfrak{p}}\cH^{0}f_{\ast}K)$ using its lifting to $H^0(Y,\alpha_{!}\alpha^{!}f_{\ast}K)$. Since the bilinear pairing $S^{\eta L}_{00}$ is defined in the same way, one proceeds similarly as in the proof of Lemma \ref{lemma: splitting map for direct image of K}. Again, it is important to notice that these two maps coincide up to a sign.
\end{proof}

\subsection{Set up of the proof}\label{sec: Set up of the inductive proof}
Let us fix two finite algebraic Whitney stratifications $\mathfrak{X}$ on $X$ and $\mathfrak{Y}$ on $Y$ adapted to the map $f$, so that all perverse sheaves we work with are constructible with respect to these stratifications. For precise definitions, the reader can consult \cite[\S 6.1]{DM05}.

For the base case, we start from the case $m=0$ and arbitrary $r$. Then all the statements follow from Hodge-Simpson Theorem \ref{thm:Hodge-Simpson} and Theorem \ref{thm: generalized Weil operator}. 

\begin{assumption}\label{inductive assumption}
Let $r=r(f)\geq 0$ and $m>0$. Assume that the results of Theorem \ref{thm:main}, Theorem \ref{thm:Hard Lefschetz Perverse Complexes} and Theorem \ref{thm:Hodge-Riemann} hold for every projective map $g:X\to Y$ and $\cV$, where either $r(g)<r$, or $\dim g(X)<m$ and $r(g)\leq r$.
\end{assumption}
We will prove that if we are in Assumption \ref{inductive assumption}, then all three Theorems hold for $f:X \to Y$ and $\cV$ with $r(f)=r$ and $\dim f(X)=m$.

\subsection{Proof of Theorem \ref{thm:main}, except the Semisimplicity Theorem for $\ell=0$}\label{sec:Theorem A(a)}
There are two cases. If $r(f)=0$, then $f$ is semismall. Since the proof of \cite[Prop 8.2.30]{HTT} also works for any local system, we conclude that $f_{\ast}K={}^{\fp}\cH^0(f_{\ast}K)$ is a perverse sheaf. In particular, Theorem \ref{thm:main} except for Theorem \ref{thm:main}(iii) $\ell=0$ automatically hold. 

If $r(f)>0$, choose an integer $k$ large enough so that $\eta^{\otimes k}$ is very ample, and fix the projective embedding $X\subseteq \bP$ induced by the linear system $|\eta^{\otimes k}|$. Consider the 
following commutative diagram from \cite[\S 4.7]{DM05}:
\[ \begin{CD}
X @<p_X<< \cX \\
@VVfV  @VVgV \\
Y @<p_Y<< \cY
\end{CD}\]
Here $\cY=Y\times \bP^{\vee}$, $\cX=\{ (x,h) \colon h(x)=0 \} \subseteq X\times \bP^{\vee}$,
the map $g$ is defined by $g(x,h)=(f(x),h)$ and the two horizontal maps $p_X$ and $p_Y$ are projections to the first factors. Let $d=\dim \bP$ and $M=p_X^{\ast}\cV[\dim \cX]$. By Corollary \ref{cor:pull back preserve semisimplicity}, $p_X^{\ast}\cV$ is a semisimple local system on $\cX$. Since we know that $r(g)<r(f)$ by\cite[Lemma 4.7.4]{DM05}, we can apply the inductive assumption on $g$ to proceed as follows. For Theorem \ref{thm:main}(i), since the functor $p_Y^{\ast}[d]$ is fully-faithful, it suffices to show that
\[ p_Y^{\ast}(\eta^{\ell})[d] \colon p_Y^{\ast}{}^{\fp}\cH^{-\ell}(f_{\ast}K)[d] \to p_Y^{\ast}{}^{\fp}\cH^{\ell}(f_{\ast}K)[d]\]
is an isomorphism. There are two cases.

\textbf{Case I}: $\ell \geq 2$. It follows by induction on $g$ and \cite[Proposition 4.7.8(i)]{DM05}. Note that we actually need Theorem \ref{thm:main}(i) to hold for $f$-ample line bundles, but this can be deduced from the ample line bundle case as in \cite[Remark 5.1.2]{DM05}.

\textbf{Case II}: $\ell =1$. We need some additional care. The cup product can be factored as
\[ p_Y^{\ast}(\eta)[d] \colon p_Y^{\ast}{}^{\fp}\cH^{-1}(f_{\ast}K)[d] \to {}^{\fp}\cH^{0}(g_{\ast}M) \to p_Y^{\ast}{}^{\fp}\cH^{1}(f_{\ast}K)[d].\]
Since ${}^{\fp}\cH^0(g_{\ast}M)$ is semisimple, the proof of \cite[Lemma 5.1.1]{DM05} implies that
\begin{equation}\tag{Cup}\label{eqn: cup product}
    \eta \colon {}^{\fp}\cH^{-1}(f_{\ast}K) \to {}^{\fp}\cH^{1}(f_{\ast}K)
\end{equation} 
is a monomorphism. Since $K^{\ast}=\cV^{\ast}[\dim X]$ and $\cV^{\ast}$ is semisimple, we also know that
\begin{equation}\tag{Cup*}\label{eqn: cup product for dual}
    \eta \colon {}^{\fp}\cH^{-1}(f_{\ast}K^{\ast}) \to {}^{\fp}\cH^{1}(f_{\ast}K^{\ast})
\end{equation}
is a monomorphism. By Corollary \ref{corollary:duality of cup product map},  the Verdier dual of $(\text{\ref{eqn: cup product for dual}})$ can be identified with the morphism (\ref{eqn: cup product}). Hence the morphism  (\ref{eqn: cup product}) is also an epimorphism. This finishes the inductive proof of Theorem \ref{thm:main}(i).

By Deligne's Lefschetz splitting criterion \cite[Theorem 1.5]{Deligne68}, Theorem \ref{thm:main}(i) implies Theorem \ref{thm:main}(ii). Regarding Theorem \ref{thm:main}(iii) $\ell \neq 0$, the semisimplicity of ${}^{\fp}\cH^{\ell}(f_{\ast}K)$ follows from \cite[Proposition 4.7.8]{DM05} and the inductive semisimplicity of ${}^{\fp}\cH^{\ell+1}(g_{\ast}M)$ and ${}^{\fp}\cH^{\ell-1}(g_{\ast}M)$.

\subsection{Proof of Theorem \ref{thm:Hard Lefschetz Perverse Complexes} and Theorem \ref{thm:Hodge-Riemann}}\label{sec:auxiliary results}

In order to prove Theorem \ref{thm:main}(iii) in the case of $\ell=0$, we first need to prove Theorem \ref{thm:Hard Lefschetz Perverse Complexes} and Theorem \ref{thm:Hodge-Riemann}. 

\begin{prop}\label{prop: proof of Theorem D}
With the assumption \ref{inductive assumption}, then Theorem \ref{thm:Hard Lefschetz Perverse Complexes}
holds for $f$.
\end{prop}

\begin{proof}
At this point, Theorem \ref{thm:main}(ii) (the Decomposition Theorem) holds for $f_{\ast}K$, hence the complex $f_{\ast}K$ is $p$-split in the sense of \cite[Definition 4.3.1]{DM05} and let us fix a choice of the splitting. We obtain an isomorphism
\begin{equation*}\label{perverse filtration on X and cohomology of perverse complex}
\nu: H^{b}_{\ell}(X,K) \xrightarrow{\sim} H^{b-\ell}(Y,{}^{\fp}\cH^{\ell}(f_\ast K)).
\end{equation*}
By \cite[Lemma 4.4.2 and Remark 4.4.3]{DM05}, the cup product maps with the first Chern classes of $\eta$ and $A$ are compatible with the isomorphism $\nu$, respectively. Therefore Theorem \ref{thm:Hard Lefschetz Perverse Complexes} is equivalent to the following two statements:
\begin{align*}
\eta^{\ell} \colon H^{j}(Y, {}^{\fp}\cH^{-\ell}(f_{\ast}K)) &\xrightarrow{\sim} H^{j}(Y,{}^{\fp}\cH^{\ell}(f_{\ast}K)), \textrm{ whenever $\ell\geq 0,j\in \Z$},\\
A^j \colon H^{-j}(Y, {}^{\fp}\cH^{\ell}(f_{\ast}K)) &\xrightarrow{\sim} H^{j}(Y,{}^{\fp}\cH^{\ell}(f_{\ast}K)), \textrm{ whenever $j\geq 0,\ell\in \Z$}.
\end{align*}

The statement for $\eta$ follows from Theorem \ref{thm:main}(i) for $f$. For $A$, the plan is to cut $X$ by hyperplane sections in $\abs{\eta}$ or $\abs{L}$ and use the corresponding weak Lefschetz theorem. 

\textbf{Case I}: $\ell \neq 0$. By Theorem \ref{thm:main}(i), we can assume $\ell<0$.
Choose a general hyperplane section $X^1\in \abs{\eta}$ and set $f^1:X^1\to Y$. For proving that $A^j$ is an isomorphism, using Proposition \ref{prop:weak Lefschetz cohomology} we can replace $\eta$ by $\eta^{\otimes m}$ for some integer $m$, so that $r(f^{1})\leq r(f)-1$ or $r(f^{1})=r(f)=0$ and $\dim f^1(X^1)< \dim f(X)$. The injectivity of $A^j$ follows from inductive assumption on the map $f^1$ and Proposition \ref{prop:weak Lefschetz cohomology}. The surjectivity follows from a dual argument as in the Case II of the proof of Theorem \ref{thm:main}(i).

\textbf{Case II}:  $\ell=0$ and $j \geq 2$. Using Bertini Theorem, we can choose $Y_1$ to be a sufficiently general hyperplane section in $\abs{A}$ so that $f^{-1}(Y_1)$ is nonsingular and $Y_1$ is transversal to all strata of $Y$. Set $f_1:X_1\colonequals f^{-1}(Y_1)\to Y_1$. Proposition \ref{prop:weak Lefschetz hypercohomology} implies that $r(f_1)\leq r(f)$ and  $\dim f_1(X_1)<\dim f(X)$. The bijectivity of $A^j$ follows from the inductive assumption on $(A|_{Y_1})^{j-1}$ and Proposition \ref{prop:weak Lefschetz hypercohomology}.

\textbf{Case III}: $\ell=0$ and $j=1$. This step requires a different argument since there is no Hodge decomposition to use, compared with \cite[Proposition 5.2.3]{DM05}. Instead, we use the pre-Weil operator $\phi$ to keep track of non-degeneracy of Poincar\'e pairings. We use the same choice as in Case II. By Theorem \ref{thm:main}(i) we have $\eta^{\ell}:{}^{\fp}\cH^{-\ell}(f_{\ast}K) \xrightarrow{\sim} {}^{\fp}\cH^{\ell}(f_{\ast}K)$, which gives a Lefschetz decomposition
\begin{equation}\label{lefschetz decomposition on the sheaf level}
{}^{\fp}\cH^{0}(f_{\ast}K)=\bigoplus_{m\geq 0} \eta^{m}\cP^{-2m},\quad \cP^{-2m}\colonequals \Ker \eta^{2m+1}\subseteq {}^{\fp}\cH^{-2m}(f_{\ast}K).
\end{equation}
If $m\geq 1$, then $A: H^{-1}(Y,\eta^{m}\cP^{-2m}) \to H^1(Y,\eta^{m}\cP^{-2m})$ is an isomorphism by Case I. Therefore we just need to deal with $m=0$, i.e. 
\begin{equation}\label{eqn: cup product with A for m=0}
    A:  H^{-1}(Y,\cP^{0}) \to H^1(Y,\cP^{0})
\end{equation}
is an isomorphism. The idea is to decompose it into restriction and Gysin maps, which come from the corresponding maps on $X$.

Let us denote by $i_X:X_1\hookrightarrow X$ the inclusion map. Since $X_1\in |L|$, the cup product with $L$ on $X$ can be decomposed into 
\begin{equation}\label{eqn: the decomposition of cup product of L}
\begin{tikzcd}
H^{n-1}(X,\cV) \arrow[r, "R=i^{\ast}_X"] \arrow[d,equal]&  H^{n-1}(X_1,\cV|_{X_1}) \arrow[r,"G=i_{X,\ast}"] \arrow[d,equal] &  H^{n+1}(X,\cV)\arrow[d,equal] \\
H^{-1}(X,K) \arrow[r, "R"] &  H^{0}(X_1,K[-1]|_{X_1}) \arrow[r,"G"] &  H^{1}(X,K) 
\end{tikzcd}
\end{equation}
where $R$ and $G$ denote the restriction map and the Gysin map respectively for $i_X$. Similarly, for $\cV^{\ast}$ we denote
\[ \widetilde{L}\colon H^{n-1}(X,\cV^{\ast}) \xrightarrow{\widetilde{R}} H^{n-1}(X_1,\cV^{\ast}|_{X_1}) \xrightarrow{\widetilde{G} } H^{n+1}(X,\cV^{\ast}).\]
Note that $G$ can be identified with the dual of $\widetilde{R}$ using the Poincar\'e duality. 

\begin{claim} The maps $R$ and $G$ in \eqref{eqn: the decomposition of cup product of L} give a decomposition of the cup product map \eqref{eqn: cup product with A for m=0} into
\begin{equation}\label{eqn: cup product on Y}
    A \colon H^{-1}(Y,\cP^{0}) \xrightarrow{R} H^0(Y_1,\cP^{0}_1) \xrightarrow{G} H^{1}(Y,\cP^{0}),
\end{equation}
where $\cP^{0}_1\colonequals \Ker \eta|_{X_1}\subseteq {}^{\fp}\cH^{0}(f_{\ast}K)[-1]|_{Y_1}$. Similarly, we have
\[ A\colon H^{-1}(Y,\tilde{\cP}^{0}) \xrightarrow{\tilde{R}} H^0(Y_1,\tilde{\cP}^{0}_1) \xrightarrow{\tilde{G}} H^{1}(Y,\tilde{\cP}^{0})\]
for corresponding objects associated with $K^{\ast}$. Moreover, $G$ is the dual of $\widetilde{R}$ under the Poincar\'e duality.
\end{claim}
\begin{proof}[Proof of claim]

Step 1: we first show that $R$ and $G$ restrict to the cup product with $A$ on ${}^{\fp}\cH^0(f_{\ast}K)$. Since we have proven that $ f_{\ast}K\cong \oplus {}^{\fp}\cH^{\ell}(f_{\ast}K)[-\ell]$, by \cite[Remark 4.4.3]{DM05} again, the cup product map $L:H^{-1}(X,K)\to H^{1}(X,K)$ restricts to the cup product map with $A$ on the component ${}^{\fp}\cH^0(f_{\ast}K)$:
\begin{equation*}
A:H^{-1}(Y,{}^{\fp}\cH^0(f_{\ast}K)) \to H^{1}(Y,{}^{\fp}\cH^0(f_{\ast}K)). 
\end{equation*}
Denote by $i_Y:Y_1\hookrightarrow Y$ the inclusion. Since $Y_1\in |A|$, the cup product with $A$ can also be decomposed as
\begin{equation}\label{eqn: the decomposition of cup product of A} A:H^{-1}(Y,{}^{\fp}\cH^0(f_{\ast}K))\xrightarrow{R=i^{\ast}_Y} H^{0}(Y_1,{}^{\fp}\cH^0(f_{\ast}K)[-1]|_{Y_1}) \xrightarrow{G=i_{Y,\ast}} H^{1}(Y,{}^{\fp}\cH^0(f_{\ast}K)).
\end{equation}
We claim that the maps $R$ and $G$ in \eqref{eqn: the decomposition of cup product of A} are the restriction of $R$ and $G$ in \eqref{eqn: the decomposition of cup product of L}. Since $L=f^{\ast}A$, one has a commutative diagram
\[ \begin{tikzcd}
X_1 \arrow[r,"i_X"] \arrow[d,"f_1"] &X \arrow[d,"f"]\\
Y_1 \arrow[r,"i_Y"]  &Y
\end{tikzcd}
\]
Since $Y_1$ is chosen to be transverse to the strata of $Y$ (adapted to $f$), the embedding $Y_1\hookrightarrow Y$ is a normally nonsingular inclusion (see \cite[Page 714]{DM05} for the definition and discussion). Furthermore, by \cite[Remark 3.5.1]{DM05} (or see the proof of \cite[Lemma 4.3.8]{DM05}) $i_Y^{\ast}[-1]$ is $t$-exact and one has
\[ i_Y^{\ast}{}^{\fp}\cH^0(f_{\ast}K)[-1]\cong {}^{\fp}\cH^0(f_{1,\ast}(i_{X}^{\ast}K[-1]))\]
This implies that the map $R$ in \eqref{eqn: the decomposition of cup product of L} induces the map $R$ in \eqref{eqn: the decomposition of cup product of A}. On the other hand, since the closed embedding $i_X$ is affine and quasi-finite, one has $i_{X,\ast}$ is $t$-exact and thus ${}^{\fp}\cH^0f_{\ast}\circ i_{X,\ast}\cong i_{Y,\ast}\circ {}^{\fp}\cH^0f_{1,\ast}$. This implies that
\[ {}^{\fp}\cH^0f_{\ast}(i_{X,\ast}(i_X^{\ast}K[-1]))=i_{Y,\ast} {}^{\fp}\cH^0(f_{1,\ast}(i_X^{\ast}K[-1])=i_{Y,\ast}i_Y^{\ast}{}^{\fp}\cH^0(f_{\ast}K)[-1],\]
and hence the map $G$ in \eqref{eqn: the decomposition of cup product of L} restricts to the map $G$ in \eqref{eqn: the decomposition of cup product of A}.

Step 2: Now it suffices to show that $R$ and $G$ in \eqref{eqn: the decomposition of cup product of A} restrict to $R$ and $G$ in \eqref{eqn: cup product on Y}. This requires addtional two substeps.

Step 2.1: we show that $R$ and $G$ in \eqref{eqn: the decomposition of cup product of L} are compatible with the cup product map $\eta: K \to K[2]$. Explicitly, we claim that the map $i_X^{\ast}(K\xrightarrow{\eta} K[2])$ is the cup product map $ \eta|_{X_1}:i_X^{\ast}K \to i_X^{\ast}K[2]$. The proof is similar to the proof of Lemma \ref{lemma: dual of cup product}, for reader's convenience, we give some details here. By \cite[Remark 4.4.1]{DM05}, one can choose a section $s\in \Gamma(X,\eta)$ whose zero locus defines a normally nonsingular inclusion $\alpha:\{s=0\} \hookrightarrow X$, so that the map $\eta:K\to K[2]$ can be described as
\[ K \to \alpha_{\ast}\alpha^{\ast}K\xrightarrow{\sim} \alpha_{!}\alpha^{!}K[2] \to K[2].\]
On the other hand, consider the following diagram
\[ \begin{tikzcd}
    \{s=0\}\cap X_1 \arrow[r,"\alpha_1"] \arrow[d,"i_{X,1}"] & X_1 \arrow[d,"i_X"]\\
    \{s=0\}\arrow[r,"\alpha"] &X
\end{tikzcd}
\]
We can choose $s$ such that $\alpha_1$ is also a normally nonsingular inclusion and so the cup product with $\eta|_{X_1}$ can be described using $\alpha_1$. Since $\alpha$ is proper, we have
\[ i_X^{\ast}\alpha_{\ast}=\alpha_{1,\ast}i_{X,1}^{\ast}.\]
Then
\begin{align*}
    i_X^{\ast}(K\xrightarrow{\eta} K[2])&=i_X^{\ast}K \to i_X^{\ast}\alpha_{\ast}\alpha^{\ast}K\xrightarrow{\sim} i_X^{\ast}\alpha_{!}\alpha^{!}K[2] \to i_X^{\ast}K[2]\\
    &=i_X^{\ast}K \to \alpha_{1,\ast}\alpha_1^{\ast}(i_X^{\ast}K)\xrightarrow{\sim} \alpha_{1,!}\alpha_1^{!}(i^{\ast}K)[2] \to i_X^{\ast}[2]\\
    &=i_X^{\ast}K\xrightarrow{\eta|_{X_1}} i_X^{\ast}K[2].
\end{align*}
Here we use the additional fact that $\alpha^{!}=\alpha^{\ast}[-2]$ \cite[Lemma 3.5.4]{DM05}. Using similar arguments, we can also show that the map $i_{X,\ast}(i_X^{\ast}K\xrightarrow{\eta|_{X_1}}i_X^{\ast}K[2])$ is the map $i_{X,\ast}i_X^{\ast}K\xrightarrow{\eta}i_{X,\ast}i_X^{\ast}K[2]$. We conclude that the Lefschetz decomposition with respect to $\eta$ on $X$ is compatible with $R$ and $G$ in \eqref{eqn: the decomposition of cup product of L}.

Step 2.2: we need to argue that this compatibility descends to $Y$, i.e. the maps $R$ and $G$ in \eqref{eqn: the decomposition of cup product of A} are compatible with the cup product map ${}^{\fp}\cH^0f_{\ast}(\eta):{}^{\fp}\cH^0f_{\ast}K\to {}^{\fp}\cH^2f_{\ast}K$, which would imply that they are compatible with the Lefschetz decomposition \eqref{lefschetz decomposition on the sheaf level}. For $R$, we need to show that the map 
\[i_Y^{\ast}\left({}^{\fp}\cH^0(f_{\ast}K)\xrightarrow{{}^{\fp}\cH^0f_{\ast}(\eta)}{}^{\fp}\cH^2(f_{\ast}K)\right)\]
is the cup product map
\[i_Y^{\ast}{}^{\fp}\cH^0(f_{\ast}K)\xrightarrow{{}^{\fp}\cH^0f_{1,\ast}(\eta|_{X_1})}i_Y^{\ast}{}^{\fp}\cH^2(f_{\ast}K).\]
This follows from the $t$-exactness of $i_Y^{\ast}$ and the compatiblity between the cup product map $\eta:K \to K[2]$ and $R$ in Step 2.1. For the compatibility of $G$, the argument is similar and we leave the details to the reader.

We conclude that the restriction map $R$ and Gysin map $G$ with respect to $L$ in \eqref{eqn: the decomposition of cup product of L} reduce to the corresponding ones with respect to $A$ on $\cP^0$ in \eqref{eqn: cup product on Y} and we finish the proof of claim.
\end{proof}

Now it suffices to show that the cup product map $A$ in \eqref{eqn: cup product on Y} is bijective. To do this, we will use the pre-Weil operator $\phi$ to produce a natural non-degenerate pairing $S:\Ker G\otimes \Ker \widetilde{G}  \to \C$, where $G$ and $\widetilde{G} $ are Gysin maps for $Y_1\hookrightarrow Y$ and $\cP^0$.

\textbf{Step 1}: Consider the  Poincair\'e pairing in \eqref{Poincare pairing}:
\[ S: H^{\ast}(X_1,\cV|_{X_1}) \otimes H^{\ast}(X_1,\cV^{\ast}|_{X_1}) \to \C.\]
By definition, $R$ and $\widetilde{G} $ are adjoint to each other with respect to $S$:
\begin{equation}\label{adjointness}
S(R(\alpha),\widetilde{\beta})=S(\alpha,\widetilde{G} (\widetilde{\beta})),
\end{equation}
for any $\alpha \in H^{n-1}(X,\cV)$ and $\widetilde{\beta}\in H^{n-1}(X_1,\cV^{\ast}|_{X_1})$.

\textbf{Step 2}.  Consider the vector spaces 
\[ \Ker A|_{Y_1} \subseteq H^0(Y_1,\cP^{0}_1), \quad \Ker \widetilde{A}|_{Y_1}\subseteq H^0(Y_1,\widetilde{\cP}^{0}_1).
\]
Denote by $K_1=K[-1]|_{X_1}$ and $f_1:X_1\to Y_1$ the induced map. By \cite[Theorem 4.4.4(c)]{DM05}, the isomorphism $H^0(Y_1,{}^{\mathfrak{p}}\cH^0(f_{1,\ast}K_1))\cong H^0_0(X_1,K_1)$ identifies
\[ H^0(Y_1,\cP^{0}_1)\cong \Ker \eta|_{X_1}\subseteq H^0_0(X_1,K_1).\]
Using \cite[Remark 4.4.3]{DM05}, one can further identify
\begin{equation}\label{eqn: identification between KerAY1 and P00} \Ker A|_{Y_1}\cong P^0_0(X_1)\colonequals\Ker \eta|_{X_1} \cap \Ker L|_{X_1}\subseteq H^0_0(X_1,K_1),
\end{equation}
where $P^0_0(X_1)$ is the primitive piece with respect to the double Lefschetz decomposition in Theorem \ref{thm:Hard Lefschetz Perverse Complexes}. The same applies for $\Ker \widetilde{A}|_{Y_1}$. Moreover, by Lemma \ref{lemma:polarization on double Lefschetz decomposition}, $S$ induces a pairing 
\[S^{\eta L}_{00}(X_1):H^0_0(X_1,K_1)\otimes H^0_0(X_1,K_1^{\ast})\to \C.\]
The inductive Theorem \ref{thm:Hodge-Riemann} for $f_1:X_1\to Y_1$ and $K_1$ implies that $S^{\eta L}_{00}(X_1)(\bullet,\phi(\bullet))$ polarizes the natural pure twistor structure $F$ on $P^0_0(X_1)$, where $\phi:\overline{F|_{z=-1}}\to \tilde{P}^0_0(X_1)$ is induced by the pre-Weil operator. Therefore the pairing
\[ S^{\eta L}_{00}(X_1): P^0_0(X_1)\otimes \tilde{P}^0_0(X_1) \to 
\C\]
is non-degenerate by Lemma \ref{lemma: polarization implies non-degeneracy}. To simplify the notation, we still denote by $S:\Ker A|_{Y_1} \otimes \Ker \widetilde{A}|_{Y_1} \to \C$ the pairing corresponding to $S^{\eta L}_{00}(X_1)$, under the identification \eqref{eqn: identification between KerAY1 and P00}.

Summarizing the discussion above, we know that $S$ induces a non-degenerate pairing
\[ S: \Ker A|_{Y_1} \otimes \Ker \widetilde{A}|_{Y_1} \to \C,\]
and $\Ker A|_{Y_1}$ underlies a pure twistor structure $ F$ polarized by $S(\bullet,\phi(\bullet))$, where $\phi:\overline{F|_{z=-1}} \xrightarrow{\sim} \Ker \widetilde{A}|_{Y_1}$ is induced by the pre-Weil operator $\phi$.   By Corollary \ref{cor:restriction of perverse filtration} and the fact that $G$ is the dual of $\widetilde{R} $, the Gysin map $G:H^0(Y_1,\cP^{0}_1) \to H^{1}(Y,\cP^{0})$ underlies a morphism of pure twistor structures. Moreover since $A|_{Y_1}=R\circ G$, we have an inclusion map 
\[ \Ker G\subseteq \Ker A|_{Y_1}, \]
which underlies a morphism of pure twistor structure. Then by Corollary \ref{cor:restriction of perverse filtration} and Lemma \ref{lemma:polarization of subtwistors}, $\Ker G$ underlies a pure twistor structure $E$ with an induced isomorphism $\phi_E: \overline{E|_{z=-1}}\xrightarrow{\sim} \Ker \widetilde{G}$ so that $E$ is polarized by 
\[S(\bullet,\phi_E(\bullet)): \Ker G\otimes \overline{E|_{z=-1}}\to \Ker G \otimes \Ker \widetilde{G} \to \C.\]
By Lemma \ref{lemma: polarization implies non-degeneracy} and the fact $\phi_E$ is an isomorphism, we conclude that the restricted pairing
\begin{equation}\label{eqn: S restricts to Ker G}
    S: \Ker G \otimes \Ker \widetilde{G} \to \C
\end{equation} 
is non-degenerate.

Now we prove the injectivity of $A$. Suppose by contradiction that there is a nonzero $\alpha \in\Ker A=\Ker G\circ R$, then $R(\alpha)\in \Ker G$. Since the pairing \eqref{eqn: S restricts to Ker G} is non-degenerate, we can find an element $\widetilde{\beta}\in \Ker \widetilde{G} $ so that
\[ 0\neq S(R(\alpha),\widetilde{\beta})=S(\alpha,\widetilde{G} (\widetilde{\beta}))=0,\]
where the first equality comes from \eqref{adjointness}. But this is a contradiction! Therefore $A$ is injective. The surjectivity of $A$ follows from the injectivity of $A$ for  $\cV^{\ast}$.
\end{proof}

\begin{corollary}\label{cor:identification of weight filtrations}
Let $W^{L}$ and $W^{\eta}$ be the weight filtrations on $H^{\ast}(X,\cV)$ in \S \ref{sec:twisted Poincare pairing}. Then
\begin{itemize}

\item $W^{\eta}_i=\bigoplus_{b \geq -i} H^{b}(X,K), \quad \Gr^{\eta}_{i}=H^{-i}(X,K)$.
\item $W^{L}_i=\bigoplus_{b \in \mathbb{Z}}H^{b}_{\leq b+i}(X,K), \quad \Gr^{L}_i=\bigoplus_{b \in \mathbb{Z}}H^{b}_{b+i}(X,K)$.
\item The filtration $W^{\eta}[j]$ induces the monodromy weight filtration of $\eta$ on $\Gr^L_j$. Therefore, for $\ell,j\in\Z$, we have a double Lefschetz decomposition:
\begin{equation}\label{double lefschetz in the proof}
H^{-\ell-j}_{-\ell}(X,\cV) = \bigoplus_{m,i\in \Z} \eta^{-\ell+i}L^{-j+m}P^{j-2m}_{\ell-2i},
\end{equation}
where $P^{-j}_{-\ell}$ are the primitive subspaces in Theorem \ref{thm:Hard Lefschetz Perverse Complexes}.

\item $\Gr^{\eta}_{j+\ell}\Gr^{L}_jH^{\ast}(X,\cV)=H^{-\ell-j}_{-\ell}(X,K)$ 
so that the bilinear pairing $S^{\eta L}_{\ell j}$ in (\ref{eqn:first time twisted Poincare pairing}) is well-defined and non-degenerate.

\end{itemize}
\end{corollary}

\begin{proof}
Since at this stage we have Theorem  \ref{thm:Hard Lefschetz Perverse Complexes}, we can just proceed as in the proof of \cite[Proposition 5.2.4]{DM05}, using Lemma \ref{lemma:polarization on double Lefschetz decomposition} and Theorem \ref{thm:Hodge-Simpson}.
\end{proof}

\begin{prop}\label{prop: proof of Hodge Riemann}
With the assumption \ref{inductive assumption}, Theorem \ref{thm:Hodge-Riemann}
holds for $f$, i.e. we have
\begin{itemize}

\item The double Lefschetz decomposition \eqref{double lefschetz in the proof} satisfies
\[ S^{\eta L}_{\ell j}(\eta^{-\ell+i}L^{-j+m}P^{j-2m}_{\ell-2i},\eta^{-\ell+i'}L^{-j+m'}P^{j-2m'}_{\ell-2i'})=0, \quad  \forall (i,m)\neq (i',m').\]

\item Each direct summand $\eta^{-\ell+i}L^{-j+m}P^{j-2m}_{\ell-2i}$ underlies a natural pure sub-twistor structure $F$ so that the pre-Weil operator $\phi$ restricts to $\phi: \overline{F|_{z=-1}} \xrightarrow{\sim} \eta^{-\ell+i}L^{-j+m}\widetilde{P}^{j-2m}_{\ell-2i}$.
\item The pure twistor structure $F$ is polarized by $S^{\eta L}_{\ell j}(\bullet,\phi(\bullet))$ up to a non-zero constant and $S^{\eta L}_{\ell j}$ is non-degenerate.
\end{itemize}
\end{prop}

\begin{proof}
The first two statements follow from Lemma \ref{lemma:polarization on double Lefschetz decomposition} and Lemma \ref{lemma:perverse is twistor}. We focus on the last statement, whose proof is divided into two cases.

\textbf{Case I}: Consider all $\eta^{-\ell+i}L^{-j+m}P^{j-2m}_{\ell-2i} \neq P^{0}_0$. It suffices to deal with the case $i\geq \ell, m\geq j$ by definition and Remark \ref{remark: signs for twisted Poincare pairing}. We can also assume $i=\ell,m=j$. Then one uses Proposition \ref{prop:weak Lefschetz cohomology} and Proposition \ref{prop:weak Lefschetz hypercohomology} and the compatibility of the bilinear paring $S^{\eta L}_{\ell j}$ between $X$ and its hyperplanes, similar to the proof of \cite[Proposition 5.3.2]{DM05}.

\textbf{Case II}: $P^0_0$. This is the most difficult part of the entire proof, as in \cite[\S 5.4]{DM05}, where we will show that $i^{-n}\cdot S^{\eta L}_{00}(\bullet,\phi(\bullet))$ polarizes the twistor structure on $P^0_0$. For any $\epsilon>0$, define
\begin{align*}
\Lambda_{\epsilon}\colonequals \Ker (L+\epsilon \eta) \subseteq H^0(X,K), \quad \widetilde{\Lambda}_{\epsilon}\colonequals \Ker(L+\epsilon \eta) \subseteq H^0(X,K^{\ast}).
\end{align*}
Let $E$ be the natural pure twistor structure on $H^0(X,K)$, with the pre-Weil operator $\phi: \overline{E|_{z=-1}} \xrightarrow{\sim}H^0(X,K^{\ast})$.
By Theorem \ref{thm:Hodge-Simpson}, each $\Lambda_{\epsilon}$ has the same dimension $b\colonequals \dim H^0(X,K)-\dim H^2(X,K)$ and underlies a pure sub-twistor structure $F_\epsilon$. By Lemma \ref{lemma:Identification map is functorial} the pre-Weil operator restricts to
\[ \phi_\epsilon: \overline{F_{\epsilon}|_{z=-1}} \xrightarrow{\sim}\widetilde{\Lambda}_{\epsilon}.\]
Recall from \eqref{Poincare pairing} there is a Poincar\'e pairing:
\[ S: H^0(X,K)\otimes H^0(X,K^{\ast}) \to \C.\]
\begin{lemma}\label{lemma:twistor structure on limiting space}
Consider the limiting spaces in the Grassmannians:
\begin{align*}
\Lambda \colonequals \lim_{\epsilon\to 0} \Lambda_{\epsilon} \in G(b,H^0(X,K)), \quad \widetilde{\Lambda} \colonequals \lim_{\epsilon\to 0}\widetilde{\Lambda}_{\epsilon} \in G(b,H^0(X,K^{\ast})).
\end{align*}
Then $\Lambda$ underlies a pure sub-twistor structure $F\subseteq E$, with the isomorphism:
\[ \phi_0: \overline{F|_{z=-1}} \xrightarrow{\sim} \widetilde{\Lambda},\]
as the limit of $\phi_\epsilon$.
Moreover, the bilinear pairing
\[ i^{-n}\cdot S(\bullet,\phi_0\circ \overline{\mathrm{Iden}}(\bullet)):\Lambda \otimes \overline{\Lambda} \to \Lambda \otimes \overline{F|_{z=-1}}\to \Lambda\otimes \widetilde{\Lambda}\xrightarrow{S} \C,\]
is positive semi-definite, where $\Iden$ is the Identification map from Definition \ref{definition:Identification map}. 
\end{lemma}

\begin{proof}
Remark \ref{remark:Category of twistor structures} implies that $\Lambda$ underlies a twistor structure $F$ and the pairing is positive semi-definite, as it is the limit of positive definite pairings. The only thing we need to argue is that $\phi_0=\lim_{\epsilon\to 0}\phi_\epsilon$ is still an isomorphism. By Lemma \ref{lemma:Identification map is functorial} and the definition of $\phi_0$, we have a commutative diagram
\[\begin{tikzcd}
    \overline{F|_{z=-1}} \arrow[r,hook]\arrow[d,"\phi_0"] & \overline{E|_{z=-1}} \arrow[d,"\phi"]\\
    \tilde{\Lambda} \arrow[r,hook] & H^0(X,K^{\ast})
\end{tikzcd}
\]
where both horizontal maps are inclusions. Since $\phi$ is injective ($\phi$ is an isomorphism) and $\phi_0$ is the restriction of $\phi$, we know that $\phi_0$ is also injective.  On the other hand, we know
\[ \dim \overline{F|_{z=-1}}=\dim F|_{z=1}=\dim \Lambda=b=\dim \tilde{\Lambda}.\]
Therefore $\phi_0$ must be an isomorphism.

%It follows from Remark \ref{remark:Category of twistor structures} and Theorem \ref{cor:polarized by twisted Poincare pairing}.
\end{proof}

\begin{notation} Denote the cup product operators with $L$ by
\begin{align*}
L^{k}_r \colon H^{-r}(X,K) \to H^{-r+2k}(X,K), \quad \widetilde{L}^{k}_r \colon H^{-r}(X,K^{\ast}) \to H^{-r+2k}(X,K^{\ast}).
\end{align*}
Similarly, the cup products with $\eta$ are denoted by 
\[ \eta:H^i(X,K)\to H^{i+2}(X,K), \quad \tilde{\eta}:H^i(X,K^{\ast})\to H^{i+2}(X,K^{\ast}).\]
\end{notation}

For any subspace $V\subseteq H^0(X,K^{\ast})$, recall that the orthogonal complement of $V$ in $H^0(X,K)$ with respect to $S$ is defined to be
\[ V^{\perp}\colonequals \{ w\in H^0(X,K) | S(w,v)=0, \forall v\in V. \}\]
The following two results are direct adaptions of \cite[Lemma 5.4.1, Lemma 5.4.2]{DM05}. For details, see \cite[Lemma 2.7.8 and Lemma 2.7.9]{Yang21}.
\begin{lemma}\label{lem: intersection of various kernels}
With the notations above, we have
\[\eta \Ker L^{1}_2 \cap (\widetilde{\eta} \Ker \widetilde{L}^{ 1}_2)^{\perp} \cap \cdots \cap (\widetilde{\eta}^{ i} \Ker \widetilde{L}^{ i}_{2i})^{\perp}=\{ 0 \} \subset H^{0}(X,K), \quad i \gg 0.\]
\end{lemma}

Since $\Ker L^1_0\subseteq H^0(X,K)$, we have a restricted pairing
\begin{equation}\label{eqn: S restricts to kernel}
    S:\Ker L^1_0 \otimes \Ker \tilde{L}^1_0 \to \C.
\end{equation}
\begin{lemma}\label{lemma:structure of Lambda}
We have
\[ \Lambda=\Ker L^{1}_{0} \cap \left( \bigcap_{i\geq 1} (\widetilde{\eta}^{ i}\Ker \widetilde{L}^{ i}_{2i})^{\perp}\right), \quad \widetilde{\Lambda}=\Ker \widetilde{L}^{1}_{0} \cap \left( \bigcap_{i\geq 1} (\eta^{ i}\Ker L^{ i}_{2i})^{\perp}\right).\]
We also have direct sum decompositions
\[ \Ker L^{1}_0 = \Lambda \oplus\eta \Ker L^{1}_2, \quad \Ker \widetilde{L}^{1}_0 = \widetilde{\Lambda} \oplus\widetilde{\eta} \Ker \widetilde{L}^{1}_2\]
which are orthogonal with respect to \eqref{eqn: S restricts to kernel}.
\end{lemma}

\begin{remark}
Starting from this point, our argument is different from \cite[\S 5.4]{DM05}. It seems to us more appropriate to prove the counterpart of \cite[Lemma 5.4.4]{DM05} first and then prove the counterpart of \cite[Lemma 5.4.3]{DM05}.
\end{remark}

Corollary \ref{cor:identification of weight filtrations} implies that
\[ \frac{\Ker L^1_0}{\Ker L^1_0\cap H^0_{\leq -1}(X,K)}\subseteq \frac{W^L_0\cap H^0(X,K)}{W^L_0\cap H^0_{\leq -1}(X,K)}=\frac{H^0_{\leq 0}(X,K)}{H^0_{\leq -1}(X,K)}=H^0_0(X,K).\]
Define 
\[\Lambda_0\colonequals \Lambda/(\Lambda \cap H^{0}_{\leq -1}(X,K))\]
and similarly for $\widetilde{\Lambda_0}$. 
\begin{corollary}\label{non-degeneracy of pairing on Lambda0}
The pairing \eqref{eqn: S restricts to kernel} induces a non-degenerate pairing
\begin{equation}\label{eqn: new pairing on the quotient of Ker}
    S^{\eta L}_{00}: \Ker L^1_0/(\Ker L^1_0\cap H^0_{\leq -1}(X,K))\otimes\Ker \widetilde{L}^1_0/ (\Ker \widetilde{L}^1_0\cap H^0_{\leq -1}(X,K^{\ast})) \to \C,
\end{equation}
which is the restriction of the pairing $S^{\eta L}_{00}$ from \eqref{eqn:first time twisted Poincare pairing}. There are direct sum decompositions
\[\Ker L^1_0/ (\Ker L^1_0\cap H^0_{\leq -1}(X,K))=\Lambda_0\oplus (\eta \Ker L^1_2/\eta \Ker L^1_2\cap H^0_{\leq -1}(X,K))\]
\[ \Ker \widetilde{L}^1_0/ (\Ker \widetilde{L}^1_0\cap H^0_{\leq -1}(X,K^{\ast}))=\widetilde{\Lambda}_0\oplus (\widetilde{\eta} \Ker \widetilde{L}^1_2/\widetilde{\eta} \Ker \widetilde{L}^1_2\cap H^0_{\leq -1}(X,K^{\ast}))\]
which are orthogonal with respect to \eqref{eqn: new pairing on the quotient of Ker}. Moreover, the restricted bilinear pairing
\[ S^{\eta L}_{00}:\Lambda_0 \otimes \widetilde{\Lambda_0} \to \C\]
is non-degenerate and we have
\begin{equation}\label{the crucial equality in the Hodge Riemann proof}
 \Lambda_0=\frac{(\widetilde{\eta} \Ker \widetilde{L}^1_2)^{\perp}\cap \Ker L^1_0}{(\widetilde{\eta} \Ker \widetilde{L}^1_2)^{\perp}\cap \Ker L^1_0\cap H^0_{\leq -1}(X,K)}.
\end{equation}
\end{corollary}

\begin{remark}
We decide to give a careful proof here because there are some subtleties about taking orthogonal complement with respect to a general (not necessarily positive definite) pairing. This is where \eqref{eqn: new pairing on the quotient of Ker} is used.
\end{remark}
\begin{proof}
By Lemma \ref{lemma:structure of Lambda}, we can apply Lemma \ref{non-degeneracy of subspaces after quotient} to $H=\Ker L^1_0, H_1=\Lambda,H_2=\eta \Ker L^{1}_2$ and the pairing \eqref{eqn: S restricts to kernel}. 

To deduce the first two statements, we need to calculate the orthogonal complement of $\Ker \widetilde{L}^1_0$ inside $\Ker L^1_0$ with respect to \eqref{eqn: S restricts to kernel}, which is 
\[\Ker L^1_0\cap (\Ker \widetilde{L}^1_0)^{\perp}=\Ker L^1_0\cap (W^L_{-1}\cap H^0(X,K))=\Ker L^1_0\cap H^0_{\leq -1}(X,K).\]
Here the equalities follows from \eqref{kernel intersect with its orthogonal complement} and  $W^L_{-1}=W^{\tot}_{-1}$ in Corollary \ref{cor:identification of weight filtrations}. Similarly \[\Lambda \cap (\Ker \widetilde{L}^1_0)^{\perp}=\Lambda \cap H^0_{\leq -1}(X,K).\] Therefore the non-degeneracy of $S^{\eta L}_{00}:\Lambda_0\otimes \widetilde{\Lambda_0}\to \C$ follows from Lemma \ref{non-degeneracy of subspaces after quotient}. 

To obtain \eqref{the crucial equality in the Hodge Riemann proof}, one needs to calculate the orthogonal complement of
\[\widetilde{\eta} \Ker \widetilde{L}^1_2/\widetilde{\eta} \Ker \widetilde{L}^1_2\cap H^0_{\leq -1}(X,K^{\ast})\]
with respect to the pairing \eqref{eqn: new pairing on the quotient of Ker} inside $\Ker L^1_0/ (\Ker L^1_0\cap H^0_{\leq -1}(X,K))$, which is 
\[ \frac{(\widetilde{\eta} \Ker \widetilde{L}^1_2)^{\perp}\cap \Ker L^1_0}{(\widetilde{\eta} \Ker \widetilde{L}^1_2)^{\perp}\cap \Ker L^1_0\cap H^0_{\leq -1}(X,K)}.\]
This is because the pairing \eqref{eqn: new pairing on the quotient of Ker} is induced from $S:H^0(X,K)\otimes H^0(X,K^{\ast})\to \C$. 
\end{proof}

\begin{lemma}\label{lemma:Gamma_0 is polarized}
With the notations above, $\Lambda_0$ underlies a pure twistor structure $F_0$, with a descent isomorphism from Lemma \ref{lemma:twistor structure on limiting space} to $\phi:\overline{F_0|_{z=-1}} \to \widetilde{\Lambda}_0$ so that $F_0$ is polarized by the form $i^{-n}\cdot S^{\eta L}_{00}(\bullet,\phi(\bullet))$. 
\end{lemma}

\begin{proof}
Since $H^0_{\leq -1}(X,K)$ underlies a natural pure twistor structure by Lemma \ref{lemma:perverse is twistor}, $\Lambda_0$ underlies a pure twistor structure $F_0$ with the pre-Weil operator $\phi:\overline{F_0|_{z=-1}} \to \widetilde{\Lambda_0}$.
Lemma \ref{lemma:twistor structure on limiting space} implies that the pairing
\[ i^{-n}\cdot S(\bullet,\phi\circ \overline{\mathrm{Iden}}(\bullet)): \Lambda \otimes \overline{\Lambda} \to \Lambda \otimes \widetilde{\Lambda} \xrightarrow{S}   \C\]
is positive semi-definite. On the other hand, the pairing $S^{\eta L}_{00}:\Lambda_0\otimes \widetilde{\Lambda_0} \to \C$
is non-degenerate by Corollary \ref{non-degeneracy of pairing on Lambda0}. Therefore the pairing on the quotients 
\[ i^{-n}\cdot S^{\eta L}_{00}(\bullet,\phi\circ \overline{\mathrm{Iden}}(\bullet)): \Lambda_0 \otimes \overline{\Lambda_0}\to \Lambda_0\otimes \widetilde{\Lambda_0}\to \C\]
is positive definite, since it is  positive semi-definite and non-degenerate at the same time. In particular, $\Lambda_0$ is polarized by $i^{-n}\cdot S^{\eta L}_{00}(\bullet,\phi(\bullet))$.
\end{proof}

Now we can finish the proof of Theorem \ref{thm:Hodge-Riemann} for $P^0_0$. By definition
\[ P^0_0 = \Ker \eta \cap \Ker L \subseteq \Gr^{\eta}_0\Gr^L_0H^{\ast}(X,\cV).\]
It is direct to check that
\[ P^0_0 \subseteq \Ker L^1_0/(\Ker L^1_0\cap W^L_{-1}).\]
The next claim is mentioned without proof in \cite[Proof of Theorem 2.1.8]{DM05}, again we would like to add some details here.
\begin{claim}\label{claim: for Hodge Riemann} 
$P^0_0 \subseteq (\widetilde{\eta} \Ker \widetilde{L}^1_2)^{\perp}/(\widetilde{\eta} \Ker \widetilde{L}^1_2)^{\perp}\cap W^L_{-1}$.
\end{claim}
\begin{proof}[Proof of claim]
Let $a \in H^{0}(X,K)=\Gr^{\eta}_0H^{\ast}(X,\cV)$ be an element so that $[a]\in P^0_0$. Then $\eta[a]=0\in \Gr^{L}_0\Gr^{\eta}_0$, which means that $\eta a\in W^{L}_{-1}H^{\ast}(X,\cV)$. Using the hard Lefschetz  $L^i:\Gr^L_i\xrightarrow{\sim} \Gr^L_{-i}$, we can find an element $b\in W^{L}_1$ so that $\eta a=Lb$. Now let $\widetilde{c}=\widetilde{\eta}\widetilde{d}$ be an element so that $\widetilde{d}\in \Ker \widetilde{L}^1_2$. Then
\[ S(a,c)=S(a,\widetilde{\eta}\widetilde{d})=-S(\eta a,\widetilde{d})=-S(Lb,\widetilde{d})=S(a,\widetilde{L}\widetilde{d})=0.\] We conclude that $a\in (\widetilde{\eta} \Ker \widetilde{L}^1_2)^{\perp}$.
\end{proof}
Now combining with  \eqref{the crucial equality in the Hodge Riemann proof}, we have
\[ P^0_0 \subseteq \frac{(\widetilde{\eta} \Ker \widetilde{L}^{ 1}_{2})^{\perp}\cap \Ker L^1_{0}}{(\widetilde{\eta} \Ker \widetilde{L}^{ 1}_{2})^{\perp}\cap \Ker L^1_0\cap W^L_{-1}}=\frac{(\widetilde{\eta} \Ker \widetilde{L}^{ 1}_{2})^{\perp}\cap \Ker L^1_{0}}{(\widetilde{\eta} \Ker \widetilde{L}^{ 1}_{2})^{\perp}\cap \Ker L^1_0\cap H^0_{\leq -1}(X,K)}=\Lambda_0.\]
Moreover, the inclusion $P^0_0\subseteq \Lambda_0$
underlies a morphism of pure twistor structures. Since $ \Lambda_0$ is polarized by $i^{-n}\cdot S^{\eta L}_{00}(\bullet,\phi(\bullet))$ by Lemma \ref{lemma:Gamma_0 is polarized}, we conclude by Lemma \ref{lemma:polarization of subtwistors} that $P^0_0$ is polarized by the restriction of $i^{-n}\cdot S^{\eta L}_{00}(\bullet,\phi(\bullet))$.
\end{proof}

\subsection{Splitting criterion of perverse sheaves}
After the proof of Theorem \ref{thm:Hodge-Riemann}, now we can prove some auxiliary facts needed for the proof of Theorem \ref{thm:main}(iii) for $\ell=0$, where the functoriality of the pre-Weil operator $\phi$ is again crucial.

Let us first recall the splitting criterion of perverse sheaves due to de Cataldo-Migiliorini \cite{DM05}. Let $Y$ be an algebraic variety and $T$ is a closed subset of $Y$. We use $T \xrightarrow{\alpha} Y\xleftarrow{\beta} Y\setminus T$ to denote the closed and open embeddings. For any perverse sheaf $P$ on $Y$, as in \S \ref{splitting criterion}, the distinguished triangles associated to $\alpha$ and $\beta$ induces the following map
\begin{equation}\label{eqn: splitting map in the proof}
H^{k}(Y,\alpha_{!}\alpha^{!}P) \to H^{k}(Y,\alpha_{\ast}\alpha^{\ast}P), \quad \forall k\in \mathbb{Z}.
\end{equation}
Furthermore, suppose there is a stratification $\mathfrak{Y}$ of $Y$ with
\[ Y = U \cup T, \quad U=\bigcup_{d'>d} S_{d'}, \quad T=S_{d},\]
where $S_{d'}$ denotes a $d'$-dimensional stratum.  
\begin{lemma}{\cite[Lemma 4.1.3]{DM05}}\label{lem:DM splitting}
Let $P$ be a perverse sheaf on $Y$, constructible with respect to $\mathfrak{Y}$. Assume 
\begin{align}\label{eqn: assumption on the dimension of stalks}
\dim \cH^{-d}(\alpha_{!}\alpha^{!}P)_y = \dim \cH^{-d}(\alpha_{\ast}\alpha^{\ast}P)_y
\end{align}
for any $y \in T$. Then the following statements are equivalent:
\begin{itemize}
\item $\cH^{-d}(\alpha_{!}\alpha^{!}P) \to \cH^{-d}(P)$ is an isomorphism.

\item $P\cong \beta_{!\ast}\beta^{\ast}P\oplus \cH^{-d}(P)[d]$.
\end{itemize}
Here $\beta_{!\ast}$ is the intermediate extension functor.
\end{lemma}

Now we prove that for $T=\{y\}$, $P={}^{\fp}\cH^0(f_\ast K)$ and $k=0$, the map in \eqref{eqn: splitting map in the proof} is an isomorphism. This implies that \eqref{eqn: assumption on the dimension of stalks} is satisfied in this case.
\begin{lemma}\label{lemma:surjectivity absolution version}
For each point $\alpha:\{y\} \hookrightarrow Y$ in the support of the sheaf $\cH^0({}^{\fp}\cH^0(f_\ast K))$ on $Y$, the restriction map
\[ H^0(Y,{}^{\fp}\cH^0(f_\ast K)) \to H^0(Y,\alpha_{\ast}\alpha^{\ast}{}^{\fp}\cH^0(f_\ast K))\]
is surjective. Moreover, the following cycle map 
\[ \mathrm{cl}:H^0(Y,\alpha_{!}\alpha^{!}{}^{\fp}\cH^0(f_\ast K)) \to H^0(Y,{}^{\fp}\cH^0(f_\ast K))\]
is injective.
\end{lemma}

\begin{proof}
Since we have Theorem \ref{thm:main}(ii) for $f_{\ast}K$ and Theorem \ref{thm: global invariant cycle}, the proof of the first statement follows the same line as in \cite[Proposition 6.2.2]{DM05}. For the second statement, apply the surjectivity statement to $K^{\ast}$.
\end{proof}

\begin{prop}\label{prop:splitting}
Let $\alpha: 
\{y\} \hookrightarrow Y$ be a point in the support of $\cH^0({}^{\fp}\cH^0(f_{\ast}K))$. Then the map \eqref{eqn: splitting map in the proof} for $k=0$ and $P={}^{\fp}\cH^0(f_{\ast}K)$:
\[ H^{0}(Y,\alpha_{!}\alpha^{!}{}^{\fp}\cH^0(f_{\ast}K)) \to H^{0}(Y,\alpha_{\ast}\alpha^{\ast}{}^{\fp}\cH^0(f_{\ast}K)) \]
is an isomorphism.
\end{prop}

\begin{proof}
By construction in \S \ref{splitting criterion}, this map decomposes as 
\begin{equation}\label{eqn:refined intersection form}
H^0(Y,\alpha_{!}\alpha^{!}{}^{\fp}\cH^{0}(f_{\ast}K) ) \xrightarrow{\mathrm{cl}} H^0(Y,{}^{\fp}\cH^0(f_{\ast}K)) \to H^0(Y,\alpha_{\ast}\alpha^{\ast}{}^{\fp}\cH^{0}(f_{\ast}K)).
\end{equation}
where $\mathrm{cl}$ is the injective cycle map in Lemma \ref{lemma:surjectivity absolution version}. Let $\beta:U=Y\setminus \{y\} \hookrightarrow Y$ be the open embedding. The distinguished triangle associated to $\{y\} \xrightarrow{\alpha} Y \xleftarrow{\beta} U$ gives rise to the following short exact sequence:
\[ 0\to H^0(Y,\alpha_{!}\alpha^{!}{}^{\fp}\cH^{0}(f_{\ast}K) ) \xrightarrow{\mathrm{cl}} H^0(Y,{}^{\fp}\cH^0(f_{\ast}K)) \to H^0(Y,\beta_{\ast}\beta^{\ast}{}^{\fp}\cH^0(f_{\ast}K)). \]
Corollary \ref{cor:restriction of perverse filtration} implies that $H^0(Y,\beta_{\ast}\beta^{\ast}{}^{\fp}\cH^0(f_{\ast}K))$ underlies a mixed twistor structure $E_U$ with an isomorphism $\overline{E_U|_{z=-1}} \xrightarrow{\sim} H^0(Y,\beta_{\ast}\beta^{\ast}{}^{\fp}\cH^0(f_{\ast}K^{\ast}))$. Therefore $H^0(Y,\alpha_{!}\alpha^{!}{}^{\fp}\cH^{0}(f_{\ast}K))$, as the kernel of the restriction map, underlies a pure sub-twistor structure $F$ so that the isomorphism $\phi$ in Corollary \ref{cor:identification of weight filtrations} restricts to
\[ \phi: \overline{F|_{z=-1}} \xrightarrow{\sim} H^0(Y,\alpha_{!}\alpha^{!}{}^{\fp}\cH^{0}(f_{\ast}K^{\ast})).\]
Now we want to use $\phi$ to polarize different pieces of $H^0(Y,\alpha_{!}\alpha^{!}{}^{\fp}\cH^{0}(f_{\ast}K))$. Note that
\[ \mathrm{Im(cl)} \subseteq \Ker L\subseteq H^0_0(X,K). \] 
By the additivity of $\alpha_{!}\alpha^{!}$, this inclusion is compatible with the double Lefschetz decomposition with respect to $\eta$ and $L$. Therefore, each direct summand of $\mathrm{Im}(\mathrm{cl})$ is polarized by the restriction of $S^{\eta L}_{00}(\bullet,\phi(\bullet))$ by Theorem \ref{thm:Hodge-Riemann} and Lemma \ref{lemma:polarization of subtwistors}, up to a non-zero constant. Then the pairing $S^{\eta L}_{00}$ in \eqref{eqn:first time twisted Poincare pairing} restricts to a non-degenerate pairing 
\[ S^{\eta L}_{00}: H^0(Y,\alpha_{!}\alpha^{!}{}^{\fp}\cH^{0}(f_{\ast}K) ) \otimes H^0(Y,\alpha_{!}\alpha^{!}{}^{\fp}\cH^{0}(f_{\ast}K^{\ast})) \to \C. \]
Using Verdier duality map \eqref{eqn: Verdier duality for 0th perverse cohomology}, this pairing gives rise to the map
\[ H^0(Y,\alpha_{!}\alpha^{!}{}^{\fp}\cH^{0}(f_{\ast}K) ) \xrightarrow{S^{\eta L}_{00}} H^0(Y,\alpha_{!}\alpha^{!}{}^{\fp}\cH^{0}(f_{\ast}K^{\ast}) )^{\vee}\xrightarrow{\mathrm{V.D.}} H^0(Y,\alpha_{\ast}\alpha^{\ast}{}^{\fp}\cH^{0}(f_{\ast}K)),\]
which coincides with the map (\ref{eqn:refined intersection form}) by Lemma \ref{cor: refined intersection is Poincare pairing}. Therefore, the non-degeneracy of $S^{\eta L}_{00}$ implies the desired isomorphism.
\end{proof}

\subsection{Semisimplicity Theorem for $\ell=0$}
\label{sec:ST}
In this subsection, we finish the proof of Theorem \ref{thm:main} by showing that ${}^{\fp}\cH^0(f_{\ast}K)$ is a semisimple perverse sheaf. To do this, we first show that ${}^{\fp}\cH^0(f_{\ast}K)$ decomposes into the direct sum of intersection complexes associated to local systems over the strata of $Y=\amalg_d S_d$, where $S_d$ is a smooth locally closed subset of pure dimension $d$. An extra care is also needed here compared to \cite[Proposition 6.3.2 and Lemma 6.1.3]{DM05}, because the assumption in \cite[Lemma 4.1.3]{DM05} is not automatically satisfied.
\begin{prop}\label{prop:splitting of the 0-th perverse cohomology}
There is a canonical isomorphism in $\mathrm{Perv}(Y)$:
\[ {}^{\fp}\cH^{0}(f_{\ast}K) \cong \bigoplus_{d =\dim S_d} \IC_{\overline{S_d}}(\cH^{-d}({}^{\fp}\cH^{0}(f_{\ast}K)|_{S_d})). \]
\end{prop}

\begin{proof}
For any $d$, consider a $d$-dimensional stratum $S_{d}$ and set $U_{d} \colonequals \bigcup_{d'\geq d} S_{d'}$. We denote two embeddings by $S_{d} \xrightarrow{\alpha} U_{d} \xleftarrow{\beta} U_{d+1}$. By Deligne's formula, one has $\beta_{\ast !}\cong \tau_{\leq -d-1}\beta_{\ast}$. Then it suffices to show that
\begin{equation}\label{eqn: after Delignes formula it suffices to prove this}
{}^{\fp}\cH^0(f_{\ast}K)|_{U_{d}} \cong \beta_{\ast !}({}^{\fp}\cH^0(f_{\ast}K)|_{U_{d+1}})\oplus \cH^{-d}({}^{\fp}\cH^0(f_{\ast}K)|_{S_{d}})[d].
\end{equation}
To achieve this, let us prove the following two claims.
\begin{itemize}
\item [(A)] $\dim \cH^{-d}(\alpha_{!}\alpha^{!}{}^{\fp}\cH^0(f_{\ast}K))_y=\dim \cH^{-d}(\alpha_{\ast}\alpha^{\ast}{}^{\fp}\cH^0(f_{\ast}K))_y$,
\item [(B)] the morphism 
\[\cH^{-d}(\alpha_{!}\alpha^{!}{}^{\fp}\cH^0(f_{\ast}K)) \to \cH^{-d}({}^{\fp}\cH^0(f_{\ast}K))\]
    is an isomorphism at $y$.
\end{itemize}
For both statements, we use the following induction method. If $d\geq 1$ and $y\in S_d$, choose a generic $d$-dimensional complete intersection $Y_{d}\subseteq Y$ which contains $y$ and is transversal to all strata of $Y$ (adapted to $f$) and $X_d=X\times_Y Y_d$ is smooth. It induces the following Cartesian diagram
\[ \begin{CD}
X_{d} @>i_d>> X \\
@VVf_{d}V @VVfV  \\
Y_{d} @>>> Y
\end{CD}
\]
By the repeated use of Proposition \ref{prop:weak Lefschetz hypercohomology}, we have $\dim Y_{d} <\dim Y$ and $r(f_{d})\leq \max \{r(f)-d,0\}$. Set $K_d\colonequals i_{d}^{\ast}K[-d]$ and denote the inclusion by $i:\{y\} \hookrightarrow Y_{d}$. We have the following commutative diagram
\begin{equation}\label{cut by hyperplanes}
\begin{CD}
\cH^{0}(i_{!}i^{!}{}^{\fp}\cH^0(f_{d\ast}K_d))_y @>>> \cH^{0}({}^{\fp}\cH^0(f_{d\ast}K_d))_y@>>> \cH^{0}(i_{\ast}i^{\ast}{}^{\fp}\cH^0(f_{d\ast}K_d))_y\\
@VV\cong V @VV\cong V @VV\cong V\\
\cH^{-d}(\alpha_{!}\alpha^{!}{}^{\fp}\cH^0(f_{\ast}K))_y @>>> \cH^{-d}({}^{\fp}\cH^0(f_{\ast}K))_y @>>>
\cH^{-d}(\alpha_{\ast}\alpha^{\ast}{}^{\fp}\cH^0(f_{\ast}K))_y
\end{CD}
\end{equation}
The isomorphisms in the vertical direction follow from \cite[Lemma 4.3.8]{DM05}. 
By the inductive Theorem \ref{thm:main}(iii) for the morphism $f_{d}$ (recall that we are under Assumption \ref{inductive assumption} and $X_d$ is smooth), ${}^{\fp}\cH^0(f_{d \ast}K_d)$ is a direct sum of intersection cohomology complexes. Therefore by \cite[Remark 4.1.2]{DM05} we have
\begin{equation}\label{eqn: dimension equality from induction}
\dim \cH^{0}(i_{!}i^{!}{}^{\fp}\cH^0(f_{d\ast}K_d))_y= \dim \cH^{0}(i_{\ast}i^{\ast}{}^{\fp}\cH^0(f_{d\ast}K_d))_y.
\end{equation}
Moreover, by Lemma \ref{lem:DM splitting} the morphism
\begin{equation}\label{eqn: statement B induction}
\cH^{0}(i_{!}i^{!}{}^{\fp}\cH^0(f_{d\ast}K_d))_y \to  \cH^{0}({}^{\fp}\cH^0(f_{d\ast}K_d))_y
\end{equation}
is an isomorphism.

We first prove (A) for any $d$. The case $d=0$ is obtained in Proposition \ref{prop:splitting of the 0-th perverse cohomology}. For $d\geq 1$ and $y\in S_d$, we apply the left and right vertical isomorphisms in the diagram  \eqref{cut by hyperplanes} to reduce (A) to the equality \eqref{eqn: dimension equality from induction}, which holds by induction.

At this point (A) is proved. Then by Lemma \ref{lem:DM splitting} (the splitting criterion for perverse sheaves), to prove \eqref{eqn: after Delignes formula it suffices to prove this}, it suffices to prove that (B) holds at any point of any stratum $S_d$ for any $d\geq 0$. For convenience, we denote this statement by $\mathrm{(B)}_d$. For $d=0$, $\mathrm{(B)}_0$ follows from Proposition \ref{prop:splitting of the 0-th perverse cohomology}. When $d\geq 1$, we apply the left vertical isomorphism in the diagram \eqref{cut by hyperplanes} to reduce $\mathrm{(B)}_d$ to \eqref{eqn: statement B induction} being an isomorphism, which also holds by induction. Therefore \eqref{eqn: after Delignes formula it suffices to prove this} holds and we finish the proof of this proposition.

\end{proof}

\begin{prop}\label{prop:semisimplicity} ${}^{\fp}\cH^0(f_{\ast}K)$ is semisimple.
\end{prop}

\begin{proof}
By Proposition \ref{prop:splitting of the 0-th perverse cohomology}, it suffices to show the local system 
\[ \cH^{-d}({}^{\fp}\cH^0(f_{\ast}K)|_{S_{d}})\] is semisimple over $S_d$ for each $d=\dim S_{d}$. Let $i: y \hookrightarrow Y$ be a point lying in $S_d$, the stalk of the local system at $y$ is 
\[ H^{-d}(Y,i_{\ast}i^{\ast}{}^{\fp}\cH^0(f_{\ast}K)) \subseteq H^{-d}(f^{-1}(y),K|_{f^{-1}(y)})= H^{n-d}(f^{-1}(y),\cV|_{f^{-1}(y)}).\]

\textbf{Case I}. $d=0$ or $d=m=\dim f(X)$. If $d=0$, the semisimplicity is trivial. Suppose $d=m$. By the construction of Whitney stratification, $S_d$ is a smooth Zariski dense open subset of the open subset of $Y$ over which $f$ is smooth. By passing to a resolution of singularity of $Y$ which does not modify $S_d$, we can assume $Y$ is smooth projective.
It follows from Corollary \ref{thm:semisimplicity smooth projective map} that $R^{n-d}f_{\ast}\cV|_{S_{d}}$ is semisimple, so is the sub-local system $\cH^{-d}({}^{\fp}\cH^0(f_{\ast}K)|_{S_{d}})$.

\textbf{Cast II}. $1\leq d=\dim S_d \leq m-1$. We use a geometric construction from \cite[\S 6.4]{DM05}. Extra care is needed to keep track of semisimplicity. Choose $A$ to be a very ample line bundle on $Y$ and let $\bP^{\vee}=|A|$ be the associated linear system. Set $\Pi \colonequals (\bP^{\vee})^{d}$. Consider the universal $d$-fold complete intersection families
\[ \cY\colonequals \{ (y,(H_1,\cdots,H_d)) \colon  y\in \bigcap^d_{j=1}H_j, \quad \text{$H_j$ is a hyperplane of $Y$} \} \subseteq Y\times \Pi\]
and set $\cX\colonequals \cY\times_{Y\times \Pi} (X\times \Pi)$. For any subset $W\subseteq Y$, we write 
\[\cY_W\colonequals \cY\times_Y W, \quad \cX_W\colonequals \cX\times_{\cY} \cY_W.\]
In \cite[\S 6.4, Page 743]{DM05}, de Cataldo and Migliorini showed that there is a commutative diagram with Cartesian squares:
 \[ \begin{CD}
    \cX_T @>p'>> \cX @>p''>>X \\
    @VV \Phi V   @VVV@VV f V \\
    \cY_T @>q'>> \cY @>q''>>Y \\
    @VV\pi V @VVV\\
    T  @>>> Y
    \end{CD} \]
satisfying the following conditions:
\begin{itemize}
    \item $T\subseteq S_d$ is a Zariski open subset.
    \item $\cY_T\to T$ admits a section $\theta$ and $\Phi$ inherits a stratification from the Whitney stratification on $f$ chosen in \S \ref{sec: Set up of the inductive proof}. Moreover all strata on $\cY_T$ map surjectively and smoothly onto $T$.
    \item The morphism $\pi\circ \Phi$ is surjective and smooth projective of relative dimension $(\dim \cX_T-\dim T)$ and $\dim \cX_T=\dim X=n$. 
 %  \item $a$ is a smooth morphism of relative dimension zero. The image $a(T)$ is a Zariski open subset of an open subset over which the map $\cX \xrightarrow{h\circ g} \Pi$ is smooth. (de Cataldo and Migliorini applied the generic smoothness to $\cX \xrightarrow{h\circ g} \Pi$ and shrinked further).
\end{itemize}
Set 
\[ p_X = p'' \circ p', \quad p_Y=q'' \circ q'.\]
Since $\cV$ is semisimple on $X$, by Corollary \ref{cor:pull back preserve semisimplicity}, we have $p_X^{\ast}\cV$ is semisimple on $\cX_T$. At this point, we already know that the Decomposition Theorem holds for $f_{\ast}(\cV[\dim X])$, by the base-change property, the Decomposition holds for $\Phi_{\ast}(p_X^{\ast}\cV[\dim \cX_T])$ as well. Therefore, we can apply Lemma \ref{lemma:surjectivity relative version} to $p_X^{\ast}\cV$ and the diagram
\[ \cX_T \xrightarrow{\Phi} \cY_T \xrightarrow{\pi} T \xrightarrow{\theta} \cY_T,\]
so that we obtain a surjection map of local systems over $T$:
\[ R^{\dim \cX_T-d}(\pi \circ \Phi)_{\ast}(p_X^{\ast}\cV) \twoheadrightarrow\cH^{-d}(\theta^{\ast}{}^{\fp}\cH^0(\Phi_{\ast}p_X^{\ast}\cV[\dim \cX_T]))=\cH^{-d}(\theta^{\ast}{}^{\fp}\cH^0(\Phi_{\ast}p_X^{\ast}K))\]
As in Case I, by passing to a suitable resolution of $Y$, we can assume $Y$ is smooth projective. Since $\cX$ is smooth projective, $T$ is a smooth subvariety contained in the open subset over which the map $\cX \to \cY \to Y$ is smooth, we can apply Corollary \ref{thm:semisimplicity smooth projective map} to conclude that 
\[ R^{\dim \cX_T-d}(\pi \circ \Phi)_{\ast}(p_X^{\ast}\cV)=R^{\dim \cX_T-d}(\pi \circ \Phi)_{\ast}((p'')^{\ast}\cV|_{\cX_T})\] 
is semisimple on $T$, where $(p'')^{\ast}\cV$ is semisimple on $\cX$ because of Corollary \ref{cor:pull back preserve semisimplicity}. Therefore the quotient local system $\cH^{-d}(\theta^{\ast}{}^{\fp}\cH^0(\Phi_{\ast}p_X^{\ast}K))$ is also semisimple. Now the standard base change formula implies that
\begin{align*}
    \cH^{-d}(\theta^{\ast}{}^{\fp}\cH^0(\Phi_{\ast}p_X^{\ast}K)) &\cong  \cH^{-d}(\theta^{\ast}{}^{\fp}\cH^0(p_Y^{\ast}f_{\ast}K))\\
    &\cong  \cH^{-d}(\theta^{\ast}p_Y^{\ast}{}^{\fp}\cH^0(f_{\ast}K)) \quad [\text{$p_Y^{\ast}$ is $t$-exact}]\\
    &\cong \cH^{-d}({}^{\fp}\cH^0(f_{\ast}K)|_T)
\end{align*}
Since $T$ is Zariski-dense in $S_{d}$, by Remark \ref{remark:semisimplicity via Zariski dense open subset} we conclude that $\cH^{-d}({}^{\fp}\cH^0(f_{\ast}K)|_{S_{d}})$ is semisimple on $S_{d}$. In particular, ${}^{\fp}\cH^0(f_{\ast}K)$ is semisimple and this finishes the proof of Theorem \ref{thm:main}.
\end{proof}

\begin{lemma}\label{lemma:surjectivity relative version}
Consider the following diagram
\[ \begin{CD} 
\cX @>\Phi>> \cY \\
@VVFV @VV\pi V \\
T @>=>> T
\end{CD} \]
and $\theta:T \to \cY$ is a section of $\pi$. Suppose
\begin{itemize}
    \item $\cX$ is smooth of dimension $n$, $T$ is smooth of dimension $d$, and $F$ is surjective and smooth projective of relative dimension $n-d$.
    \item The map $\Phi$ is stratified and the strata of $\cY$ map smoothly and surjectively onto $T$. $\theta(T)$ is a stratum of $\cY$.
    \item The decomposition theorem holds for $\Phi_\ast \cV[\dim \cX]$, where $\cV$ is a semisimple local system on $\cX$.
\end{itemize}
Then, there is a surjective map of local systems on $T$:
\[ R^{\dim \cX-d}F_{\ast}\cV \to \cH^{-d}(\theta^{\ast}{}^{\fp}\cH^0(\Phi_{\ast}\cV[\dim \cX])).\]
%In particular, the latter local system on $T$ is semisimple.
\end{lemma}

\begin{proof}
The proof is basically identical to \cite[Lemma 6.4.1]{DM05}: by working stalkwise, one can reduce to Lemma \ref{lemma:surjectivity absolution version}, which is the semisimple local system version of \cite[Proposition 6.2.2]{DM05}.

\end{proof}

%\begin{lemma}\label{lemma: induced semisimplicity}
%Let $Y$ be a smooth variety and let $\cV$ be a local system on $Y$. Let $a:Y' \to Y$ be a smooth morphism of relative dimension zero, then $\cV$ is semisimple if and only if $a^{\ast}\cV$ is semisimple. 
%\end{lemma}

%\begin{proof}
%Since $\pi_{1}(Y')\subseteq \pi_1(Y)$ is a finite index subgroup, this follows from the fact that a representation of a group is semisimple if and only if its restriction to a finite index sub-group is semisimple. For example, see \cite[Lemma 2.7]{BH}.
%\end{proof}

\subsection{Sabbah's Theorem}\label{sec:Sabbah's theorem}

Using the standard reductions in \cite[Page 71-74]{DM03} and \cite{de16}, we obtain the following version of Sabbah's Decomposition Theorem in \cite{Sabbah}. 
\begin{thm}\label{thm:Sabbah's theorem}
Let $f:U \to Y$ be a proper map between algebraic varieties, where $U$ is a Zariski open subset of a smooth projective variety $X$. Let $\cV$ be a semisimple local system on $X$. Then Theorem \ref{thm:main}(ii) and Theorem \ref{thm:main}(iii) hold for $f$ and $\cV|_U$. If in addition, $f$ is projective and $\eta$ is an $f$-ample line bundle, then Theorem \ref{thm:main}(i) holds as well.
\end{thm}

\begin{proof}
One needs to keep track of the semisimplicity of $\cV$. For details, see \cite[\S 2.9]{Yang21}.
\end{proof}

\appendix

\section{Hodge star operators}\label{sec:Hodge star operators}
In this appendix, for the lack of appropriate references in the literature, we state the construction and basic properties of Hodge star operators for differential forms with Harmonic bundle coefficients, adapting \cite[\S 5]{Voisin}. We only give formulas, and the proof can be found in \cite[\S 2.4]{Yang21}. It is used to understand the natural pure twistor structures on $H^k(X,\cV)$ in Theorem \ref{thm:Hodge-Simpson} and the pre-Weil operator in Definition \ref{definition: pre-Weil operator}.

Let $X$ be a compact K\"ahler manifold of complex dimension $n$ with a K\"ahler metric $g$. Let $\cA^{k}_X$ denote the sheaf of $\cC^{\infty}$ $k$-forms on $X$ and let $\cA^{p,q}_X$ denote the sheaf of $(p,q)$-forms. Following \cite[\S 5.1.1]{Voisin}, there is an induced Hermitian metric on $\cA^k_X$: if $e_1,\ldots,e_n$ is an orthonormal basis for $(T_{X,x},g_x)$ and $e_i^{\ast}$ is the dual basis, then $e_{i_1}^{\ast}\wedge \cdots \wedge e_{i_k}^{\ast}$ form an orthonormal basis for the Hermitian metric on $\cA^k_{X,x}$. Let $H$ be a harmonic bundle on $X$ with harmonic metric $h$. Together, $h_x$ and the metric on $\cA^{k}_{X,x}$ induce a Hermitian metric $(,)_x$ on $H_x\otimes_{\C} \cA^{k}_{X,x}$. Denote $H^{\ast}$ to be the dual harmonic bundle from Construction \ref{construction:dual harmonic bundle}.

\begin{definition}[$L^2$ metric]\label{L2 metric}
For each integer $k$, there is a $L^2$ Hermitian metric on the space of $k$-forms with coefficients in $H$ defined by
\[ (A,\overline{B})_{L^2}\colonequals \int_X (A,\overline{B}) \mathrm{Vol},\]
where $\mathrm{Vol}$ is the volume form of $X$ relative to $g$, $A,B\in \cC^{\infty}(H\otimes \cA^k_X)$, and $(A,\overline{B})$ is the function $x\mapsto (A_x,\overline{B_x})_x$. 
\end{definition}

\begin{definition}[Hodge star operator]\label{definition:C-linear Hodge star operator}
The Hodge star operator for harmonic bundles is defined to be the composition of the following $\C$-linear map
\[ \ast: H_x \otimes \cA^k_{X,x} \to \text{Hom}(H_x \otimes \cA^k_{X,x}, \C) \to \overline{H^{\ast}_x \otimes \cA^{2n-k}_{X,x}},\]
where the first map is induced by $(,)_x$ and the second map is induced by the exterior product
\[ (H_x\otimes \cA^k_{X,x})\otimes (\overline{H^{\ast}_x \otimes \cA^{2n-k}_{X,x}})\to \cA^{2n}_{X,x}\cong \C.\]
so that we have
\[ (A_x,\overline{B_x})_x \mathrm{Vol_x} = A_x \wedge \overline{\ast B_x}, \]
where $A,B \in \cC^{\infty} (H\otimes  \cA^k_{X})$. In particular,
\[ (A,\overline{B})_{L^2} = \int_X A \wedge \overline{\ast B}. \] 
\end{definition}

%\begin{remark}
%The $L^2$ norm on the space of forms with bundle coefficients depend on the global geometry of $X$. For example, the definition makes sense since we assume that $X$ is compact. Otherwise we need to work with compactly supported forms.
%\end{remark}
\begin{lemma} The Hodge star operator restricts to 
\[ \ast : H \otimes \cA^{p,q}_X \to \overline{\Hdual \otimes \cA^{n-p,n-q}_X}=\overline{\Hdual} \otimes \cA^{n-q,n-p}_X.\]
\end{lemma}

Let $e$ be a $\cC^{\infty}$ section of $H$. One constructs a section $\overline{e^{\vee}}\in \cC^{\infty}(\overline{\Hdual})$ via
\[ \overline{e^{\vee}}(\bullet)= h(e,\bullet).\]
where $h$ is the harmonic metric. 

\begin{lemma}\label{lemma:Hodge star operator} The Hodge star operator for $H$ satisfies
\[ \ast (e\otimes \alpha) = \overline{e^{\vee}}\otimes (\ast \alpha), \]
where $e\otimes \alpha \in \cC^{\infty}(H\otimes  \cA^k_X)$ and $\ast \alpha$ is the Hodge star operator for $k$-forms so that
\[ (\alpha,\overline{\beta})_{L^2} = \int_X \alpha \wedge \ast \overline{\beta}.\]
\end{lemma}

\begin{lemma}\label{lemma:Hodge star operator for primitive forms}
Let $L$ be the ample line bundle on $X$ associated to the K\"ahler metric $g$. Let $e\otimes \alpha$ be a primitive $(p,q)$-form with coefficient in $H$ so that $p+q=k$. Then for $k\leq n$, we have
\[\ast(e\otimes \alpha) = C \overline{e^{\vee}}\otimes L^{n-k}\wedge \alpha,\]
where $C=\frac{(-1)^{k(k+1)/2}i^{p-q}}{(n-k)!}$ is a constant.
For $k\geq n$ we have
\[ \ast(e\otimes \alpha) = C \overline{e^{\vee}}\otimes (L^{k-n})^{-1}(\alpha),\]
here $(L^{k-n})^{-1}$ represents the inverse of the cup product map
\[ L^{k-n}: \cA^{2n-k}_X \to \cA^{k}_X.\]
\end{lemma}

\vspace{3pt}

\bibliographystyle{abbrv}
\bibliography{decomposition}{}

\begin{thebibliography}{10}

\bibitem{BBD}
A.~A. Be\u{\i}linson, J.~Bernstein, P.~Deligne, and O.~Gabber.
\newblock Faisceaux pervers.
\newblock volume 100 (2nd ed.) of {\em Ast\'{e}risque}, pages vi +180. Soc.
  Math. France, Paris, 2018.

\bibitem{BL}
B.~Bhatt and J.~Lurie.
\newblock {$p$}-adic {R}iemann-{H}ilbert correspondence.
\newblock {\em to appear}.

\bibitem{BK}
G.~B\"{o}ckle and C.~Khare.
\newblock Mod {$l$} representations of arithmetic fundamental groups. {II}. {A}
  conjecture of {A}. {J}. de {J}ong.
\newblock {\em Compos. Math.}, 142(2):271--294, 2006.

\bibitem{BW}
N.~Budur and B.~Wang.
\newblock Absolute sets and the decomposition theorem.
\newblock {\em Ann. Sci. \'{E}c. Norm. Sup\'{e}r. (4)}, 53(2):469--536, 2020.

\bibitem{Corlette}
K.~Corlette.
\newblock Flat {$G$}-bundles with canonical metrics.
\newblock {\em J. Differential Geom.}, 28(3):361--382, 1988.

\bibitem{de16}
M.~A.~A. de~Cataldo.
\newblock Decomposition theorem for semi-simples.
\newblock {\em J. Singul.}, 14:194--197, 2016.

\bibitem{DM03}
M.~A.~A. de~Cataldo and L.~Migliorini.
\newblock The {H}odge theory of algebraic maps.
\newblock {\em arXiv preprint math/0306030}, 2003.

\bibitem{DM05}
M.~A.~A. de~Cataldo and L.~Migliorini.
\newblock The {H}odge theory of algebraic maps.
\newblock {\em Ann. Sci. \'{E}cole Norm. Sup. (4)}, 38(5):693--750, 2005.

\bibitem{DM09}
M.~A.~A. de~Cataldo and L.~Migliorini.
\newblock The decomposition theorem, perverse sheaves and the topology of
  algebraic maps.
\newblock {\em Bull. Amer. Math. Soc. (N.S.)}, 46(4):535--633, 2009.

\bibitem{DM10}
M.~A.~A. de~Cataldo and L.~Migliorini.
\newblock The perverse filtration and the {L}efschetz hyperplane theorem.
\newblock {\em Ann. of Math. (2)}, 171(3):2089--2113, 2010.

\bibitem{Deligne68}
P.~Deligne.
\newblock Th\'{e}or\`eme de {L}efschetz et crit\`eres de
  d\'{e}g\'{e}n\'{e}rescence de suites spectrales.
\newblock {\em Inst. Hautes \'{E}tudes Sci. Publ. Math.}, (35):259--278, 1968.

\bibitem{Deligne71}
P.~Deligne.
\newblock Th\'{e}orie de {H}odge. {II}.
\newblock {\em Inst. Hautes \'{E}tudes Sci. Publ. Math.}, (40):5--57, 1971.

\bibitem{Deligne74}
P.~Deligne.
\newblock Th\'{e}orie de {H}odge. {III}.
\newblock {\em Inst. Hautes \'{E}tudes Sci. Publ. Math.}, (44):5--77, 1974.

\bibitem{Drinfeld}
V.~Drinfeld.
\newblock On a conjecture of {K}ashiwara.
\newblock {\em Math. Res. Lett.}, 8(5-6):713--728, 2001.

\bibitem{ElZein}
F.~El~Zein.
\newblock Mixed {H}odge structures.
\newblock {\em Trans. Amer. Math. Soc.}, 275(1):71--106, 1983.

\bibitem{ELY}
F.~El~Zein, L.~D. {u}ng Tr\'{a}ng, and X.~Ye.
\newblock Decomposition, purity and fibrations by normal crossing divisors,
  2018.

\bibitem{EKP}
P.~Eyssidieux, L.~Katzarkov, T.~Pantev, and M.~Ramachandran.
\newblock Linear {S}hafarevich conjecture.
\newblock {\em Ann. of Math. (2)}, 176(3):1545--1581, 2012.

\bibitem{Gaitsgory}
D.~Gaitsgory.
\newblock On de {J}ong's conjecture.
\newblock {\em Israel J. Math.}, 157:155--191, 2007.

\bibitem{God71}
C.~Godbillon.
\newblock {\em \'{E}l\'{e}ments de topologie alg\'{e}brique}.
\newblock Hermann, Paris, 1971.

\bibitem{HTT}
R.~Hotta, K.~Takeuchi, and T.~Tanisaki.
\newblock {\em {$D$}-modules, perverse sheaves, and representation theory},
  volume 236 of {\em Progress in Mathematics}.
\newblock Birkh\"{a}user Boston, Inc., Boston, MA, 2008.
\newblock Translated from the 1995 Japanese edition by Takeuchi.

\bibitem{Kashiwara96}
M.~Kashiwara.
\newblock Semisimple holonomic {$D$}-modules.
\newblock In {\em Topological field theory, primitive forms and related topics
  ({K}yoto, 1996)}, volume 160 of {\em Progr. Math.}, pages 267--271.
  Birkh\"{a}user Boston, Boston, MA, 1998.

\bibitem{KS}
M.~Kashiwara and P.~Schapira.
\newblock {\em Sheaves on manifolds}, volume 292 of {\em Grundlehren der
  mathematischen Wissenschaften [Fundamental Principles of Mathematical
  Sciences]}.
\newblock Springer-Verlag, Berlin, 1990.
\newblock With a chapter in French by Christian Houzel.

\bibitem{MV}
R.~MacPherson and K.~Vilonen.
\newblock Elementary construction of perverse sheaves.
\newblock {\em Invent. Math.}, 84(2):403--435, 1986.

\bibitem{Mochizuki1}
T.~Mochizuki.
\newblock Asymptotic behaviour of tame harmonic bundles and an application to
  pure twistor {$D$}-modules. {I}.
\newblock {\em Mem. Amer. Math. Soc.}, 185(869):xii+324, 2007.

\bibitem{Mochizuki2}
T.~Mochizuki.
\newblock Asymptotic behaviour of tame harmonic bundles and an application to
  pure twistor {$D$}-modules. {II}.
\newblock {\em Mem. Amer. Math. Soc.}, 185(870):xii+565, 2007.

\bibitem{Mochizuki3}
T.~Mochizuki.
\newblock Wild harmonic bundles and wild pure twistor {$D$}-modules.
\newblock {\em Ast\'{e}risque}, (340):x+607, 2011.

\bibitem{Sabbah}
C.~Sabbah.
\newblock Polarizable twistor {$D$}-modules.
\newblock {\em Ast\'{e}risque}, (300):vi+208, 2005.

\bibitem{Saito88}
M.~Saito.
\newblock Modules de {H}odge polarisables.
\newblock {\em Publications of the Research Institute for Mathematical
  Sciences}, 24(6):849--995, 1988.

\bibitem{Schmid}
W.~Schmid.
\newblock Variation of {H}odge structure: the singularities of the period
  mapping.
\newblock {\em Inventiones mathematicae}, 22(3):211--319, 1973.

\bibitem{Simpson88}
C.~Simpson.
\newblock Constructing variations of {H}odge structure using {Y}ang-{M}ills
  theory and applications to uniformization.
\newblock {\em J. Amer. Math. Soc.}, 1(4):867--918, 1988.

\bibitem{Simpson92}
C.~Simpson.
\newblock Higgs bundles and local systems.
\newblock {\em Inst. Hautes \'{E}tudes Sci. Publ. Math.}, (75):5--95, 1992.

\bibitem{Simpson93}
C.~Simpson.
\newblock Some families of local systems over smooth projective varieties.
\newblock {\em Ann. of Math. (2)}, 138(2):337--425, 1993.

\bibitem{Simpson97}
C.~Simpson.
\newblock Mixed twistor structures.
\newblock {\em arXiv preprint alg-geom/9705006}, 1997.

\bibitem{Voisin}
C.~Voisin.
\newblock {\em Hodge theory and complex algebraic geometry. {I}}, volume~76 of
  {\em Cambridge Studies in Advanced Mathematics}.
\newblock Cambridge University Press, Cambridge, 2002.
\newblock Translated from the French original by Leila Schneps.

\bibitem{Yang21}
R.~Yang.
\newblock {\em Decomposition Theorem for Semisimple Local Systems}.
\newblock PhD thesis, Stony Brook University, 2021.

\end{thebibliography}
\end{document}